\documentclass[12pt]{amsart}

\usepackage{graphicx, overpic}
\usepackage[below]{placeins}
\usepackage[colorlinks=true, linkcolor=blue, citecolor=blue]{hyperref}
\usepackage[]{algorithm2e}
\usepackage{comment}

\usepackage[T1]{fontenc}
\usepackage{amsmath,amsthm,amscd,amssymb,eucal}

\usepackage{enumerate, amsfonts, latexsym, color, url}
\usepackage{epstopdf}
\usepackage{pinlabel}

\makeatletter

\newcommand{\dashover}[2][\mathop]{#1{\mathpalette\df@over{{\dashfill}{#2}}}}
\newcommand{\fillover}[2][\mathop]{#1{\mathpalette\df@over{{\solidfill}{#2}}}}
\newcommand{\df@over}[2]{\df@@over#1#2}
\newcommand\df@@over[3]{%
  \vbox{
    \offinterlineskip
    \ialign{##\cr
      #2{#1}\cr
      \noalign{\kern1pt}
      $\m@th#1#3$\cr
    }
  }%
}
\newcommand{\dashfill}[1]{%
  \kern-.5pt
  \xleaders\hbox{\kern.5pt\vrule height.8pt width \dash@width{#1}\kern.5pt}\hfill
  \kern-.5pt
}
\newcommand{\dash@width}[1]{%
  \ifx#1\displaystyle
    2pt
  \else
    \ifx#1\textstyle
      1.5pt
    \else
      \ifx#1\scriptstyle
        1.25pt
      \else
        \ifx#1\scriptscriptstyle
          1pt
        \fi
      \fi
    \fi
  \fi
}
\newcommand{\solidfill}[1]{\leaders\hrule\hfill}

\newsavebox{\@brx}
\newcommand{\llangle}[1][]{\savebox{\@brx}{\(\m@th{#1\langle}\)}%
  \mathopen{\copy\@brx\kern-0.5\wd\@brx\usebox{\@brx}}}
\newcommand{\rrangle}[1][]{\savebox{\@brx}{\(\m@th{#1\rangle}\)}%
  \mathclose{\copy\@brx\kern-0.5\wd\@brx\usebox{\@brx}}}
\makeatother

\DeclareMathOperator\Homeo{Homeo}
\DeclareMathOperator\Map{Map}
\DeclareMathOperator\SO{SO}

\begin{document}

\newtheorem{theorem}{Theorem}[section]
\newtheorem{lemma}[theorem]{Lemma}
\newtheorem{proposition}[theorem]{Proposition}
\newtheorem{corollary}[theorem]{Corollary}
\newtheorem{conjecture}[theorem]{Conjecture}
\newtheorem{question}[theorem]{Question}
\newtheorem{problem}[theorem]{Problem}
\newtheorem*{claim}{Claim}
\newtheorem*{criterion}{Criterion}
\newtheorem*{G_equivalence_theorem}{$G$-Equivalence Theorem~\ref{theorem:G_equivalence}}
\newtheorem*{equivalence_theorem}{Equivalence Theorem~\ref{theorem:equivalence}}

\theoremstyle{definition}
\newtheorem{definition}[theorem]{Definition}
\newtheorem{provisional_definition}[theorem]{Provisional Definition}
\newtheorem{construction}[theorem]{Construction}
\newtheorem{notation}[theorem]{Notation}
\newtheorem{object}[theorem]{Object}
\newtheorem{operation}[theorem]{Operation}

\theoremstyle{remark}
\newtheorem{remark}[theorem]{Remark}
\newtheorem{example}[theorem]{Example}

\numberwithin{equation}{subsection}

\newcommand\id{\textnormal{id}}

\newcommand\CW{CaTherine wheel}
\newcommand\HH{\mathbb H}
\newcommand\N{\mathbb N}
\newcommand\Z{\mathbb Z}
\newcommand\Q{\mathbb Q}
\newcommand\R{\mathbb R}
\newcommand\C{\mathbb C}
\newcommand\DD{\mathbb D}
\newcommand\CP{\mathbb{CP}}
\newcommand\RP{\mathbb{RP}}
\newcommand\CC{\mathcal C}
\newcommand\F{\mathcal F}
\newcommand\I{\mathcal I}
\newcommand\D{\mathcal D}
\newcommand\M{\mathcal M}
\newcommand\EE{\mathcal E}
\newcommand\EEE{\mathfrak E}
\newcommand\PEEE{\mathfrak P}
\newcommand\dashEEE{\dashover{\mathfrak E}}
\newcommand\LL{\mathcal L}
\newcommand\KK{\mathcal K}
\newcommand\QQ{\mathcal Q}
\newcommand\SLE{\textnormal{SLE}}
\newcommand\LQG{\textnormal{LQG}}
\newcommand\rel{\textnormal{rel}}
\newcommand\Rel{\textnormal{Rel}}
\newcommand\Lam{\textnormal{Lam}}
\newcommand\area{\textnormal{area}}
\newcommand\length{\textnormal{length}}
\newcommand\vol{\textnormal{vol}}
\newcommand\inte{\textnormal{int}}
\newcommand\un{\textnormal{univ}}
\newcommand\fun{f}
\newcommand\fix{\textnormal{fix}}

\title{CaTherine Wheels}
\author{Danny Calegari}
\address{University of Chicago \\ Chicago, Ill 60637 USA}
\email{dannyc@uchicago.edu}
\author{Ino Loukidou}
\address{University of Chicago \\ Chicago, Ill 60637 USA}
\email{thelouk@uchicago.edu}
\date{\today}
\dedicatory{for Curt McMullen, our teacher and friend}

\begin{abstract}
A {\em CaTherine wheel} is a surjective continuous map $f:S^1 \to S^2$ such that for
every closed interval $I\subset S^1$ the image $f(I)$ is homeomorphic to a disk, and
$f(\partial I)$ is contained in the boundary of this disk.
CaTherine wheels arise in many areas of low-dimensional geometry and topology,
including conformal dynamics (expanding Thurston maps, expanding origamis), probability
theory (whole plane $\SLE_\kappa$ for $\kappa \ge 8$, LQG metric trees) and
elsewhere. We develop their theory in generality, and explain how CaTherine wheels
and their associated structures can serve as a dictionary between these
various fields.

Our most substantial applications are to the theory of hyperbolic 3-manifolds.
If $M$ is a closed hyperbolic 3-manifold and $G=\pi_1(M)$, we show that there is a
canonical bijection between four kinds of structures associated to $M$:
\begin{enumerate}
\item{orbit-equivalence classes of pseudo-Anosov flows on $M$ without perfect fits;}
\item{$G$-equivariant CaTherine wheels up to conjugacy;}
\item{minimal $G$-zippers; and}
\item{connected components of the space of uniform quasimorphisms on $G$.}
\end{enumerate}
This generalizes and amplifies the theory of fiberings of hyperbolic 3-manifolds
over the circle and the Thurston norm.
\end{abstract}

\maketitle
\setcounter{tocdepth}{1}
\tableofcontents

\section*{Introduction}

A {\em CaTherine wheel} is a surjective continuous map $f:S^1 \to S^2$ satisfying two
axioms: for every closed interval $I \subset S^1$ the image $f(I)$ is homeomorphic to
the closed unit disk; and the image $f(\partial I)$ is contained in the boundary
circle $\partial f(I)$. 

A picture of (the image of) a CaTherine wheel would not be very interesting, but every
CaTherine wheel may be approximated by an embedding $S^1 \to S^2$, and these can be
visually fascinating; see Figure~\ref{CaTherine_wheel}.

\begin{figure}[htpb]
\centering
\includegraphics[scale=0.2]{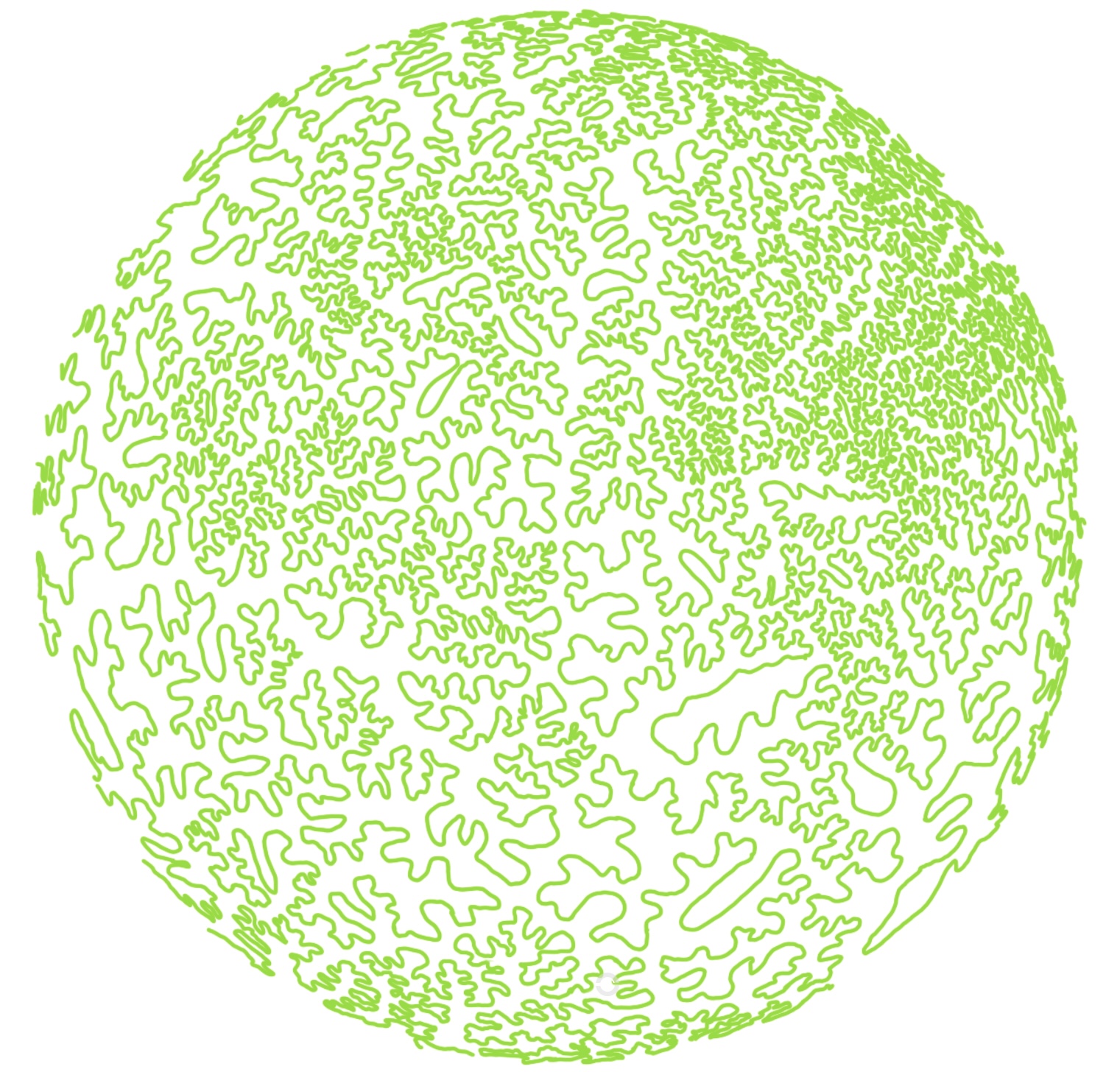}
\caption{A CaTherine wheel; not to be confused with Figure~\ref{Saint_Catherine}.}
\label{CaTherine_wheel}
\end{figure}

\subsection*{Endomorphisms}

If $G$ is some abstract semigroup acting by endomorphisms on $S^2$ then a
$G$-CaTherine wheel is a CaTherine wheel $f:S^1 \to S^2$ together with an action of $G$
on $S^1$ so that $f$ intertwines the given $G$ actions. Some familiar and well-studied
examples (but by no means the only ones; see e.g.\/ \S~\ref{subsection:expanding_origamis}) are 
\begin{enumerate}
\item{Cannon--Thurston \cite{Cannon_Thurston}: when $M$ is a closed hyperbolic 3-manifold 
that fibers over the circle, and $G=\pi_1(M)$ acting on $\CP^1$ there is a 
so-called Cannon--Thurston map $f:S^1_\infty \to \CP^1$
where $S^1_\infty$ is the ideal boundary of the universal cover of a fiber; and}
\item{Tan Lei, Meyer \cite{Tan_Lei_mate, Meyer_unmate}: when $G$ is the semigroup generated by a
degree 2 rational map $H:\CP^1\to \CP^1$ obtained by mating two quadratic polynomials
associated to suitable pairs of Misiurewicz points in the Mandelbrot set (those that do not
lie on conjugate limbs, and for which the associated {\em ray-equivalence classes} --- 
see \S~\ref{subsubsection:mating} --- have diameter zero)
there is a surjective map $f:S^1 \to \CP^1$ intertwining the dynamics of 
$h:z \to z^2$ on $S^1$ with $H$.}
\end{enumerate}
The name `CaTherine wheel' (and its nonstandard capitalization) is derived from the
names of the authors in the first example, and their wider contribution to related topics 
throughout low-dimensional geometry and topology.

In fact it is possible to give a complete characterization of $G$-CaTherine wheels when
$G = \pi_1(M)$ for $M$ a closed hyperbolic 3-manifold. First let us give some informal
definitions (precise definitions are given in the main body of the paper).

We assume the reader is familiar with the notion of a {\em pseudo-Anosov flow} $X$ on $M$;
see e.g.\/ \cite{Thurston_dynamics} or \cite{Barthelme_Frankel_Mann}.
A pseudo-Anosov flow has {\em no perfect fits} if no pair of stable and unstable leaves
are asymptotic at infinity in the orbit space of $\tilde{X}$, the lifted flow on the
universal cover of $M$. 

A (minimal) {\em $G$-zipper} is a pair $Z^\pm$ of disjoint, nonempty, path-connected,
$G$-invariant subsets of $\CP^1$ which is minimal with respect to these properties; 
see \cite{Calegari_Loukidou_Zippers}. 

We briefly recall the notion of a {\em quasimorphism} $\phi:G \to \R$
on a group. This is a function for which there is some least non-negative real
number $D(\phi)$ called the defect so that for all $g,h\in G$ there in an inequality
$$|\phi(gh)-\phi(g)-\phi(h)| \le D(\phi)$$
see e.g.\/ \cite{Calegari_scl}. 

Any (absolutely) bounded function on a group is a quasimorphism; the set of quasimorphisms
modulo bounded functions on $G$ is a vector space $Q(G)$, and every quasimorphism has
a canonical representative in $Q(G)$, a {\em homogeneous} quasimorphism, i.e.\/ one
for which $\phi(g^n)=n\phi(g)$ for all $g\in G$ and $n\in \Z$. When $G$ is
finitely generated, $H^1(G)$ is a finite dimensional subspace of $Q(G)$, and the
defect $D(\cdot)$ descends to a norm on $Q(G)/H^1(G)$ making it into a (typically
non-separable) Banach space. Thus $Q(G)$ has a well-defined topology.

An unbounded quasimorphism on a finitely generated group is Lipschitz
in any word metric on the group, and it is said to be {\em uniform} if the coarse
level sets are coarsely connected in the sense of Gromov \cite{Gromov_asymptotic}; see
also \cite{Calegari_Loukidou_Zippers}.

Combining the main results of \S~\ref{subsection:pseudo_Anosov} and 
\S~\ref{subsection:G_zippers} (Theorem~\ref{theorem:wheels_and_flows},
Theorem~\ref{theorem:G_zipper_to_G_wheel} and Theorem~\ref{theorem:G_zippers_are_hairy}),
we obtain the following:

\begin{G_equivalence_theorem}
For $M$ a fixed closed hyperbolic 3-manifold with fundamental group $G$ acting on $\CP^1$
there are canonical bijections between the following structures:
\begin{enumerate}
\item{pseudo-Anosov flows without perfect fits on $M$ up to orbit equivalence;}
\item{$G$-CaTherine wheels up to conjugacy;}
\item{(minimal) $G$-zippers; and }
\item{connected components in $Q(G)$ of the space of uniform quasimorphisms $\phi:G\to \R$.}
\end{enumerate}
\end{G_equivalence_theorem}

The proof of this theorem uses results of KyeongRo Kim \cite{Kim} and
Calegari--Zung \cite{Calegari_Zung} as well as the content of 
\cite{Calegari_Loukidou_Zippers}; furthermore, it gives a logically new proof of
a well-known theorem of Fenley \cite{Fenley_no_perfect_fits} (in our language:
that every pseudo-Anosov flow without perfect fits  
gives rise to a $G$-CaTherine wheel), albeit under 
the stronger hypothesis that $M$ is already known to be hyperbolic.
Fenley's proof of this theorem is very long and 
technical, so we think having a really conceptually different 
proof of his result is a good thing.

It is a consequence of this theorem that for $G=\pi_1(M)$ as above,
each connected component of the space of
uniform quasimorphisms is an open convex cone in the space $Q(G)$ of all
quasimorphisms (modulo bounded functions) on $G$ (this depends on \cite{Calegari_Zung},
Theorem~2.10). Let us elaborate.
If $X$ is a pseudo-Anosov flow and $\tilde{X}$ is its lift to the universal
cover, let $Q_X(G)\subset Q(G)$ denote the space of
quasimorphisms whose restriction to each (coarse) flowline of
$\tilde{X}$ is a coarse quasi-isometry to $\R$; we say that the $\phi \in Q_X(G)$ are 
{\em adapted} to $X$. It turns out that every quasimorphism adapted to some $X$
is uniform, and conversely. Furthermore, each $Q_X(G)$ is an open convex cone in $Q(G)$, and every
connected component of the space of uniform quasimorphisms on $G$ is 
$Q_X(G)$ for some $X$, unique up to orbit equivalence.

Here is how to think about what it means for a quasimorphism to be uniform.
If $G$ is a finitely generated group, a surjective 
homomorphism $\phi:G \to \Z$
is uniform (as a quasimorphism) if and only if the kernel is finitely generated, and
more generally, a non-zero homomorphism $\phi:G \to \R$ is
uniform if and only if its projective class is contained in the BNS-invariant of $G$
\cite{BNS}. For $M$ a closed 3-manifold, Stallings showed that a surjective homomorphism
$\phi:\pi_1(M) \to \Z$ has finitely generated kernel if and only if it arises from
a fibration of $M$ over $S^1$. Such homomorphisms are precisely the primitive elements of $H^1(M)$
(or, by Poincar\'e duality, $H_2(M)$) that projectively intersect the interiors of the 
fibered faces of the unit ball of the Thurston norm.
Thus our theorem may be seen as generalizing the classical correspondence, 
between orbit equivalence classes of 
{\em suspension} pseudo-Anosov flows on $M$, and cones on (open) fibered faces of 
the Thurston norm ball.

One important remark to make is that the existence of a uniform quasimorphism on a group
is a structure that exists purely at the level of coarse geometry in the sense of Gromov.
In particular, we are able to give a characterization of a 3-manifold
with a pseudo-Anosov flow without perfect fits purely in terms of group theory. The main
advantage of our perspective is that we work directly with objects that exist on
the Gromov boundary of $G$; as an application, we are able to show 
(Theorem~\ref{theorem:wheel_Cannon_conjecture}) that if $G$ is a
hyperbolic group with $\partial_\infty G$ homeomorphic to $S^2$, and if there is a
$G$-CaTherine wheel (for the natural $G$ action on $\partial_\infty G$) then $G$ is
virtually a 3-manifold group; i.e.\/ Cannon's Conjecture is true for $G$. It is not
clear to what extent this represents a viable approach to the full conjecture.

\subsection*{Point set topology}

One of the most attractive and useful aspects of the theory of CaTherine wheels is that
so much of their structure may be derived directly from the axioms without any
assumption of symmetry or analytic regularity whatsoever. 

A surjective map $f:S^1 \to S^2$ induces an equivalence relation on $S^1$, whose
equivalence classes are the point preimages. If $f$ is a CaTherine wheel, it turns out 
that the set of nontrivial equivalence classes $\LL$ admits a {\em canonical}
partition into two subsets $\LL = \LL^+ \sqcup \LL^-$ 
in such a way that each of $\LL^\pm$ is {\em laminar}:
i.e.\/ the equivalence relation on $S^1$ induced by each $\LL^\pm$ separately is closed,
and distinct equivalence classes in $\LL^+$ are unlinked in $S^1$ (and similarly for
$\LL^-$).

The laminar relations $\LL^\pm$ let us build a canonical 2-sphere $S^2_f$ with $S^1$ as the
equator, and a surjective map $F:S^2_f \to S^2$ whose restriction to $S^1$ agrees with $f$.
The images $Z^\pm$ of the upper and lower (open) hemispheres of $S^2_f$ under $F$ are 
a zipper: they are disjoint, dense, path-connected subsets of $S^2$, each of which
is a countable increasing union of finite trees, and in which every point is a cut point.

Part~\ref{part:point_set_topology} describes the relationship between $f:S^1 \to S^2$
and $\LL^\pm$ and $Z^\pm$, and gives necessary and sufficient conditions for 
a pair of laminar relations or a zipper to arise from a CaTherine wheel. 

A pair of laminar relations $\LL^\pm$ is said to have {\em no perfect fits} if the
equivalence classes in $\LL^+$ and in $\LL^-$ are disjoint. A laminar relation $\LL$ 
gives rise to an associated {\em boundary lamination} $\Lambda:=\Lam(\LL)$ (see 
Definition~\ref{definition:boundary_lamination}). The laminar relation is said
to have {\em no isolated sides} if each leaf $\lambda$ of $\Lambda$ is either a
nontrivial limit of leaves of $\Lambda$ on either side, or it is a nontrivial limit
of leaves on one side, and bounds a nontrivial polyhedron 
associated to an equivalence class $\mu$ of $\LL$ with $|\mu|\ge 3$ on the other side.

A zipper $Z^\pm\subset S^2$ has the {\em strong landing property} if, roughly speaking,
every proper ray in $Z^\pm$ extends continuously to a unique endpoint in $S^2$, and
every point in $S^2$ arises as an endpoint of some ray in $Z^+$ and in $Z^-$; 
see Definition~\ref{definition:strong_landing_property} for details.
A zipper is {\em hairy} if every embedded arc in $Z^+$ has arcs branching off it on either side
and similarly for $Z^-$ (see Definition~\ref{definition:hairy}).

The following amalgamates Theorem~\ref{theorem:laminar_decomposition}, 
Theorem~\ref{theorem:zippers_from_wheels}, Theorem~\ref{theorem:wheels_from_laminations}, 
and Theorem~\ref{theorem:wheels_from_zippers}:
\begin{equivalence_theorem}
There is a canonical bijection between isomorphism classes of the following objects:
\begin{enumerate}
\item{CaTherine wheels $f:S^1 \to S^2$;}
\item{pairs of laminar relations $\LL^\pm$ on $S^1$ with no perfect fits and no isolated sides; and}
\item{hairy zippers $Z^\pm \subset S^2$ with the strong landing property.}
\end{enumerate}
\end{equivalence_theorem}

Thus each of the three $f,\LL^\pm,Z^\pm$ are interchangable and faithful avatars
of the same abstract mathematical object. The laminar relations $\LL^\pm$ live on $S^1$,
the zipper $Z^\pm$ lives on $S^2$, and $f$ relates the two.

Here is an application. Let $\EEE$ denote the space of embeddings $S^1\to S^2$, thought
of as a subspace of the space of all maps from $S^1$ to $S^2$ in the compact-open topology.
Let $f:S^1 \to S^2$ be a nowhere locally
constant map which is not an embedding, but admits a {\em pseudo-isotopy}; i.e.\/
there is a 1-parameter family of maps $f_t:S^1 \to S^2$ for $t\in [0,1]$ which are
embeddings for $t<1$ and for which $f=f_1$. Let $\PEEE(f)$ denote the space of
pseudo-isotopies for $f$. Theorem~\ref{theorem:CaTherine_unique_pseudoisotopy} says
that for $f$ a CaTherine wheel, $\PEEE(f)$ is path-connected. The
alternative to this is a {\em self-bumping point} in the frontier of
$\EEE$: an $f:S^1 \to S^2$ that admits two pseudo-isotopies that are not homotopic 
through a family of pseudo-isotopies. Such self-bumping points occur at a dense
subset of points in the frontier of $\EEE$; their analysis
and classification is almost entirely unknown.

\subsection*{Sullivan's dictionary}

One advantage of developing the theory of CaTherine wheels in generality is that once
we identify the intrinsic `meaning' of a CaTherine wheel in a specialized
context, we may transport this meaning to different contexts where it will 
motivate new definitions and constructions, and give rise to interesting questions. 
Let us give three examples.

\medskip

{\bf 1:} $G$-zippers were defined and investigated in \cite{Calegari_Loukidou_Zippers} for $G$
the fundamental group of a closed hyperbolic 3-manifold. Yan He and Daniel Meyer
realized that such objects and suitable generalizations should exist for $G$ the
semigroup generated by an {\em expanding Thurston map} $H:S^2 \to S^2$. In the case
that there is a $G$-CaTherine wheel $f:S^1 \to S^2$, the zippers $Z^\pm$ are the
full inverse images under iterates of $H^{-1}$ of the 
{\em trimmed Hubbard trees} $T^\pm_0$ --- ordinary Hubbard trees with their 
1-valent vertices cut off; see Proposition~\ref{proposition:zipper_preimage}. 
For general expanding Thurston maps the definition and construction of $Z^\pm$ 
is more subtle; see \cite{AHLZ}.

{\bf 2:} For $G$ the fundamental group of a closed hyperbolic 3-manifold $M$, a
$G$-CaTherine wheel $f:S^1 \to \CP^1$ determines a canonical properly embedded
3-manifold $\tilde{N} \subset UT\HH^3$ which is homeomorphic to $\R^3$, which is
foliated by flowlines of the geodesic flow, and which admits a free properly
discontinuous action of $G$. Such a manifold $\tilde{N}$ can be constructed for
{\em any} CaTherine wheel $f:S^1 \to \CP^1$. CaTherine wheels $f$ for which the 
target 2-sphere admits a natural conformal structure include $f$ coming from 
(holomorphic) expanding Thurston maps, and whole plane 
Schramm--Loewner evolution (i.e.\/ $\SLE$; see \S~\ref{section:SLE}) 
from $\infty$ to $\infty$ with driving parameter $\kappa$ for $\kappa\ge 8$.
Can one define the dynamics of $H$ on $\tilde{N}$ for $H$ an expanding Thurston map?
What does the geometry of a random $\tilde{N}$ look like? What can one say about the
(countably many) singular orbits --- how close do they come to each other? 

{\bf 3:} For $G$ the fundamental group of a closed hyperbolic 3-manifold $M$,
there is an equivalence between $G$-CaTherine wheels and pseudo-Anosov flows without
perfect fits on $M$. What objects correspond to general pseudo-Anosov flows?
Many (most?) pseudo-Anosov flows are quasigeodesic, and conversely
Frankel--Landry \cite{Frankel_Landry_flows} showed that every quasigeodesic flow
on a closed hyperbolic 3 manifold is homotopic to a (quasigeodesic) pseudo-Anosov flow, 
unique up to orbit equivalence.
For such quasigeodesic flows, the associated laminar relations $\LL^\pm$ are an
example of what Frankel--Landry call an {\em especial pair} (see \cite{Frankel_Landry_flows}
Definition~3.1); but not every 
especial pair gives rise to a flow, or even to a surjective map $f:S^1 \to S^2$. 
Suppose we try to define a P-CaTherine wheel
to be a surjective map $f:S^1 \to S^2$ satisfying certain desirable properties,
so that in the $G$-invariant case such objects up to conjugacy correspond to
orbit-equivalence classes of (general) pseudo-Anosov flows on $M$.
How can one characterize such P-CaTherine wheels intrinsically? What are the corresponding
P-zippers? Can such $f$ be characterized in terms of topological properties of
a pair of laminar relations $\LL^\pm$? A tentative investigation of such matters 
is begun in \S~\ref{section:P_wheels}.

\medskip

We should emphasize that in this paper, we barely discuss CaTherine wheels as they
pertain to expanding Thurston maps, or to $\SLE_\kappa$ except as examples. 
This is partly from our ignorance, but mostly because these theories are 
already highly developed; for the first, see e.g.\/ Bonk--Meyer \cite{Bonk_Meyer} 
and for the second, see e.g.\/ Gwynne--Holden--Sun \cite{Gwynne_survey} or
Lawler \cite{Lawler_book}. 

\vfill
\pagebreak

\part{Point set topology}\label{part:point_set_topology}

\section{CaTherine wheels}\label{section:CaTherine_wheels}

\subsection{Definition}

In what follows let $S^1$ and $S^2$ denote the standard circle and the standard 2-sphere
respectively. These are closed orientable manifolds, and we fix for all time a choice 
of orientation on each of them. 

A {\em closed interval} $I\subset S^1$ is a closed subset 
homeomorphic to $[0,1]$. We can and do think of such an $I$ as a (codimension 0) submanifold 
of $S^1$; the boundary $\partial I$ is both the boundary as a submanifold, and the frontier
as a closed subset of a topological space. Such an interval inherits an orientation
from $S^1$. If we give $[0,1]$ the standard orientation, then we may choose an 
orientation-preserving homeomorphism $[0,1] \to I$, and denote the images of $0$ and $1$
under such a homeomorphism by $I^-$ and $I^+$ respectively. Thus $\partial I$ is the
disjoint union of $I^-$ and $I^+$.

\begin{definition}[CaTherine wheel]\label{definition:CaTherine_wheel}
A {\em CaTherine wheel} is a continuous map $f:S^1 \to S^2$ so that
\begin{enumerate}
\item{$f$ is surjective; and}
\item{for any closed interval $I \subset S^1$
\begin{enumerate}
\item{the image $f(I)$ is homeomorphic to the closed unit disk $D^2$; and}
\item{the image $f(\partial I)$ is contained in the boundary circle $\partial f(I)$.}
\end{enumerate}
}
\end{enumerate}
\end{definition}

\begin{figure}[htpb]
\centering
\includegraphics[scale=0.5]{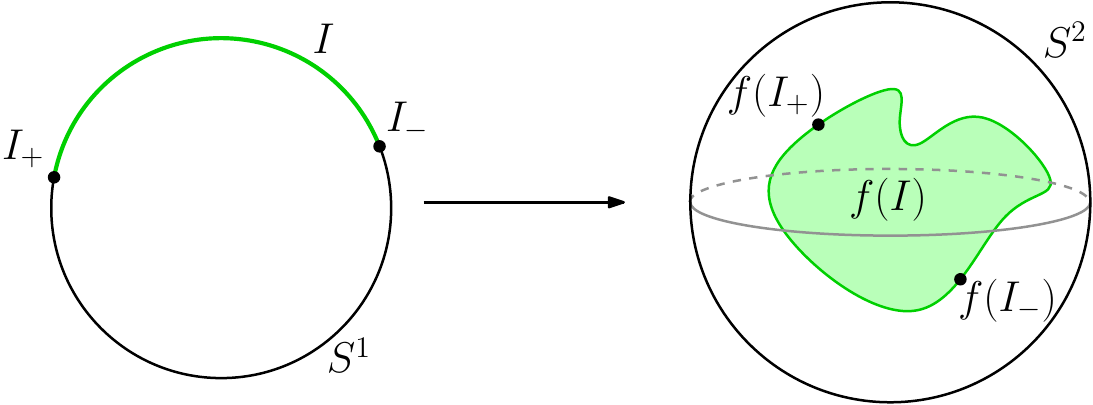}
\caption{The defining property of a CaTherine wheel}
\label{wheel_definition}
\end{figure}

\begin{remark}
By the Jordan Curve Theorem every topological embedding of $D^2$ in $S^2$ is standard, 
i.e.\/ there is a self-homeomorphism of $S^2$ taking the boundary to a great circle and
the interior to an (open) half-space in the usual round metric.
\end{remark}

Here is a suggestive image to keep in mind. If we fix $f$ and `grow' 
the interval $I$ inside $S^1$ then the image $f(I)$ also grows, like an accumulating blob
of foam extruded from a nozzle with its tip always at $f(\partial I)$, which
lies in the boundary of what has been extruded so far. 

\begin{lemma}[Basic properties]\label{lemma:basic_properties}
Let $f:S^1 \to S^2$ be a CaTherine wheel. Then:
\begin{enumerate}
\item{for any nonempty open subset $U \subset S^2$ there is a closed interval $I\subset S^1$
with $f(I) \subset U$;}
\item{$f$ is nowhere locally constant; and}
\item{if $I$ and $J$ are closed intervals in $S^1$ with disjoint interiors then 
$f(I)$ and $f(J)$ have disjoint interiors.}
\end{enumerate}
\end{lemma}
\begin{proof}
Property (1) is true for any surjective continuous map, whereas property (2) follows
immediately from the fact that for $f$ a CaTherine wheel,
the image of any nontrivial interval in $S^1$ has nonempty interior. 

To prove (3), choose an arbitrary point $p\in J$ for which $f(p)$ is in the interior
of $f(J)$. By the definition of a CaTherine wheel, $p$ is in the interior of $J$ and is
therefore in the complement of $I$. Let $K$ denote the closed interval with $K^- = p$ and 
$K^+ =  I^+$; thus in particular $I \subset K$ and therefore $f(I) \subset f(K)$.
But $p\in \partial K$ so $f(p) \in \partial f(K)$ and is therefore not in the interior of $f(I)$.
Interchanging the roles of $I$ and $J$ finishes the proof.
\end{proof}

Two closed intervals $I,J \subset S^1$ are {\em adjacent} if they have disjoint interiors and
exactly one (end)point in common. They are {\em complementary} if they have disjoint interiors
and both endpoints in common. The unique complementary interval to $I$ is denoted $I^c$; note
that $I^c$ is not the set-theoretic complement of $I$, rather it is the set-theoretic complement
of the interior of $I$, or equivalently it is the closure of the set-theoretic complement of $I$.
Evidently $I^{cc}=I$.

\begin{lemma}[Adjacent intervals]\label{lemma:adjacent_intervals}
Let $f:S^1 \to S^2$ be a CaTherine wheel, and let $I,J\subset S^1$ be two closed intervals
with disjoint interiors.
\begin{enumerate}
\item{if $I$ and $J$ are adjacent, then
$f(I) \cap f(J)$ is a closed interval which is contained in $\partial f(I) \cap \partial f(J)$; and}
\item{if $I$ and $J$ are complementary, then $f(I) \cap f(J) = \partial f(I) = \partial f(J)$.}
\end{enumerate}
\end{lemma}
\begin{proof}
We have already seen by Lemma~\ref{lemma:basic_properties} that $f(I)$ and $f(J)$ have disjoint
interiors so $f(I) \cap f(J) = \partial f(I) \cap \partial f(J)$. 
On the other hand, $f(I) \cup f(J) = f(I\cup J)$ is a closed disk if $I$ and $J$ are adjacent,
or all of $S^2$ if $J=I^c$. Since $f(I) \cup f(J)$ is connected and simply-connected, 
the intersection $f(I) \cap f(J)$ is (by e.g.\/ Mayer-Vietoris) nonempty, closed, and connected
and therefore consists of a point, a closed interval, or a circle. The first case cannot occur
or $f(I\cup J)$ would have a cut point. The other two cases correspond to the cases that
$I,J$ are adjacent and complementary respectively.
\begin{figure}[htpb]
\centering
\includegraphics[scale=0.5]{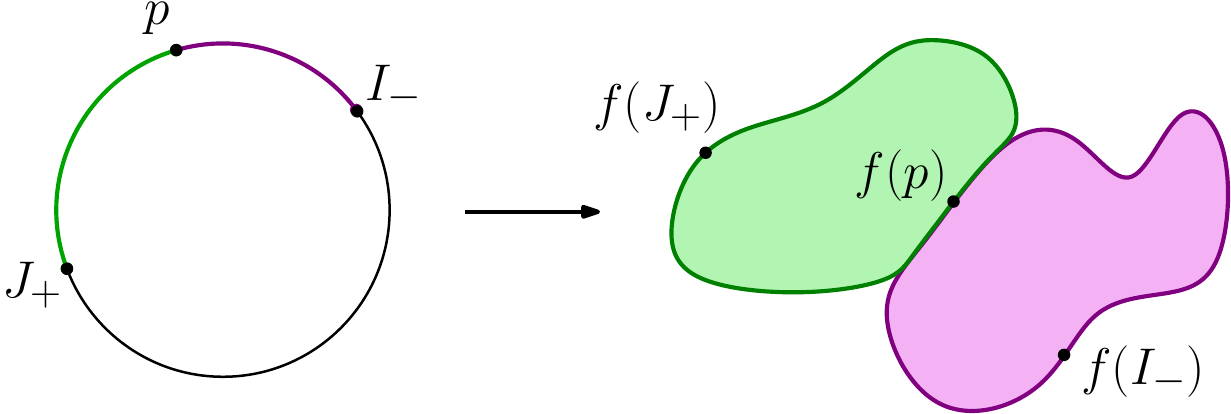}
\caption{If $I$ and $J$ are adjacent, $f(I)$ and $f(J)$ share a closed interval in their
respective boundaries.}
\label{adjacent_intervals}
\end{figure}
\end{proof}

Let $\tilde{\M}$ denote the set of closed intervals in $S^1$. If $\Delta \subset S^1\times S^1$
denotes the diagonal, we may identify $\tilde{\M}$ with the
space $S^1 \times S^1 - \Delta$ of distinct ordered pairs of points in $S^1$ as follows: given
an ordered pair of points $p,q \in S^1$ with $p \ne q$ there is a unique (oriented) 
closed interval $I \subset S^1$ with $I^- = p$ and $I^+ = q$ (sometimes in the sequel we will
denote this oriented interval $I$ by $[p,q]$). Thus $\tilde{\M}$ may be naturally topologized
as an open annulus. The involution that exchanges the factors of $S^1$ acts freely on
$\tilde{\M}$, and the quotient $\M$ is homeomorphic to an open M\"obius band. Observe
that $\M$ may be identified with the space of distinct {\em unordered} pairs of points in $S^1$. 
One may also think of this involution as the operation that takes a closed interval $I$ to $I^c$.

\begin{definition}[Degenerate intervals]\label{definition:degenerate_intervals}
Let $f:S^1 \to S^2$ be a CaTherine wheel. A closed interval $I\subset S^1$ is {\em degenerate}
for $f$ if $f(\partial I)$ consists of a single point, and is {\em nondegenerate} for $f$
otherwise (we also just use the terms {\em degenerate} and {\em nondegenerate} if $f$ is understood). 

Denote the space of intervals degenerate for $f$ by $\tilde{\LL}(f) \subset \tilde{\M}$ or just 
$\tilde{\LL}$ if $f$ is understood, and let $\LL(f)$ denote the image of $\tilde{\LL}(f)$
in $\M$.
\end{definition}

An interval $I$ is degenerate for $f$ if any only if $I^c$ is; thus $\tilde{\LL}$ is the
preimage of $\LL \subset \M$ under the double covering $\tilde{\M} \to \M$. 
Note that $\LL$ is a closed subset of $\M$ (and likewise
$\tilde{\LL}$ is closed in $\tilde{\M}$). We shall see in the sequel that $\LL$ is 
totally disconnected and perfect.

\begin{lemma}[Boundary continuous]\label{lemma:boundary_continuous}
The map that sends $I$ to $\partial f(I)$ is a continuous map from $\tilde{\M}$ to the
space of closed subsets of $S^2$ with the Hausdorff topology.
\end{lemma}
\begin{proof}
Fix an interval $I$ and a compact subset $K \subset \partial f(I) - f(\partial I)$, and let
the Hausdorff distance from $K$ to $\partial f(I)$ be $\delta$.
Choose $\delta > \epsilon > 0$ so that that the $\epsilon$-neighborhood $U$ of $f(\partial I)$ is
disjoint from $K$. By continuity of $f$ there is an open neighborhood $V$ of $I \in \tilde{\M}$
so that for any $J \in V$ we have $f(\partial J) \in U$ and therefore also 
$K \subset \partial f(J)$ (because $K$ is in the frontier of $f(J)$, and the frontier equals
the boundary). Observe that $\partial f(J) \subset \partial f(I) \cup U$ and contains $K$;
thus the Hausdorff distance from $\partial f(I)$ to $\partial f(J)$ is at most $\delta$,
and the lemma is proved.
\end{proof}

\subsection{Monotonicity}

For any closed interval $I\subset S^1$ the subset $\partial f(I) - f(\partial I)$ consists
of either one or two open intervals, depending on whether $I$ is degenerate for $f$ or not.
In either case, let $Y$ be a component of $\partial f(I) - f(\partial I)$ and let
$C: = f^{-1}(Y) \subset I$. Since $f$ is a CaTherine wheel and the 
image of $C$ has no interior, $C$ is totally
disconnected. The orientation on $I$ determines a total order on $C$; likewise, any choice
of orientation on the interval $Y$ determines a total order on $Y$ (as a set). The next
proposition relates these orders:

\begin{proposition}[Monotonicity]\label{proposition:monotonicity}
For any closed interval $I\subset S^1$ and any component $Y$ of $\partial f(I) - f(\partial I)$ 
with preimage $C:=f^{-1}(Y) \subset I$ there is a canonical orientation on $Y$ for
which the map $f:C \to Y$ is monotone increasing with respect to the associated orders.

If $I$ is nondegenerate for $f$, the canonical orientation on $Y$ is the one for which the
associated orientation on the closure $\overline{Y}$ runs from $f(I^-)$ to $f(I^+)$.
\end{proposition}
\begin{proof}
We first prove this proposition for nondegenerate intervals. 
Thus, let $I$ be nondegenerate for $f$ so that $f(I^-)$ and $f(I^+)$ are
distinct points in $\partial f(I)$. Let $Y$ be one of the open complementary intervals 
of $\partial f(I) - f(\partial I)$, and orient the closed interval
$\overline Y$ so that the orientation runs from $f(I^-)$ to $f(I^+)$.

Suppose there are points $x$ and $y$ in $C$ so that $I^- < x < y < I^+$ in
the orientation on $I$, but $f(I^-) < f(y) < f(x) < f(I^+)$ in the orientation on $\overline{Y}$.
The image of the interval $[I^-,x]$ is a closed disk contained entirely in the disk $f(I)$, and
therefore the interior of $f([I^-,x])$ contains an arc $\alpha$ properly embedded in the 
interior of $f(I)$ and running from $f(I^-)$ to $f(x)$. Likewise, the interior of $f([y,I^+])$ 
contains an arc $\beta$ properly embedded in the interior of $f(I)$ and running from
$f(y)$ to $f(I^+)$. The pairs $\lbrace f(I^-),f(x)\rbrace$ and $\lbrace f(y),f(I^+)\rbrace$
are linked in $\partial f(I)$ by hypothesis, so $\alpha$ and $\beta$ must intersect. But
this means that the disjoint intervals $[I^-,x]$ and $[y,I^+]$ in $S^1$, have images
whose interiors intersect, contrary to Lemma~\ref{lemma:basic_properties}, bullet (3). This
contradiction proves the proposition for nondegenerate intervals.

\begin{figure}[htpb]
\centering
\includegraphics[scale=0.5]{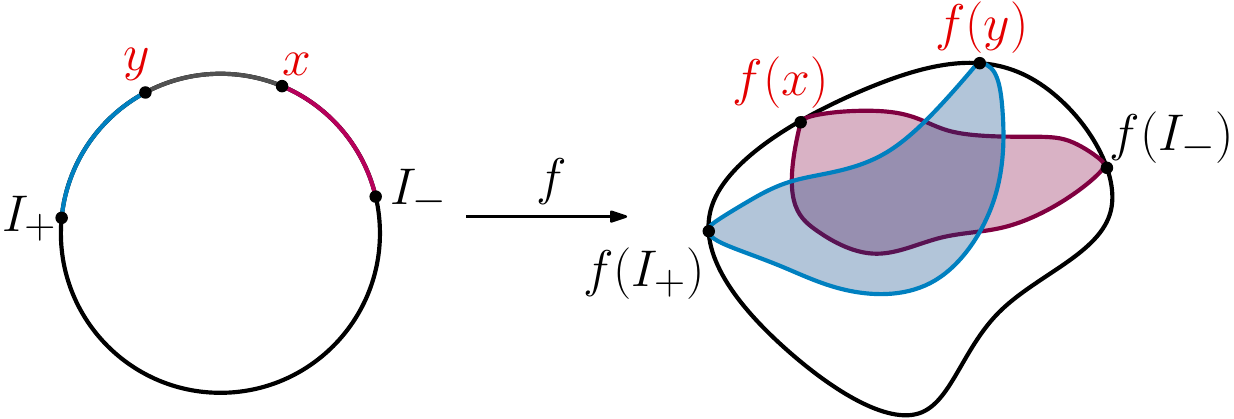}
\caption{A failure of monotonicity forces the interiors of $f([I^-,x])$ and
$f([y,I^+])$ to intersect.}
\label{monotonicity}
\end{figure}

Now let's suppose $I$ is degenerate for $f$. If we choose two points $x,z$ in $C$ with
disjoint image, then we may choose an orientation on $Y$ for which $f(x) < f(z)$ in $Y$.
Now suppose there is $y \in C$ with $x < y < z$ but $f(x) < f(z) < f(y)$ in $Y$ (the case
$f(y) < f(x) < f(z)$ is similar). Then we have $x < y < z < I^+$ in $I$ but 
$\lbrace f(x), f(y)\rbrace$ links $\lbrace f(z), f(I^+)\rbrace$ in $\partial f(I)$.
Repeating the argument of the previous paragraph with $x,y,z,I^+$ in place of $I^-,x,y,I^+$
gives a contradiction. This proves the general case and concludes the argument.
\end{proof}

Proposition~\ref{proposition:monotonicity} determines an orientation
on each component $Y$ of each $\partial f(I) - f(\partial I)$ that we call 
the {\em canonical orientation} associated to $I$. Define $Z^+(I)$ to be the
unique component of $\partial f(I) - f(\partial I)$ (if one exists) for which the
canonical orientation agrees with the orientation it inherits from $\partial f(I)$
as the boundary of a submanifold of the oriented manifold $S^2$, and to be
empty otherwise, and define $Z^-(I)$ to be the component for which the canonical
orientation {\em disagrees} with the orientation on $\partial f(I)$.
We say that a degenerate interval $I$ is {\em positively degenerate} for $f$ if $Z^+(I)$ is 
{\em empty} (note the choice of convention!) and {\em negatively degenerate} otherwise.

\begin{figure}[htpb]
\centering
\includegraphics[scale=0.5]{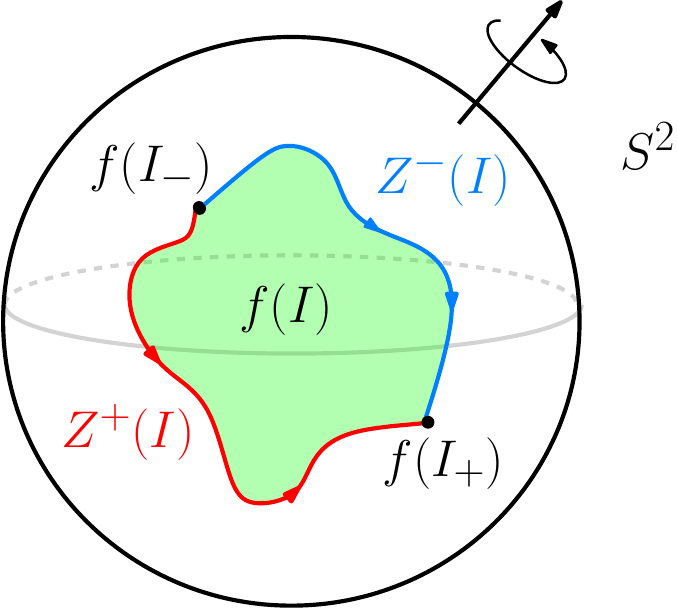}
\caption{The oriented intervals $Z^\pm(I)$ in $\partial f(I) - f(\partial I)$.}
\label{definition_of_Z}
\end{figure}

\begin{remark}
If we want to emphasize the dependence on $f$ we denote $Z^\pm(I)$ by $Z^\pm_f(I)$.
\end{remark}

\begin{lemma}[Canonical orientation continuous]\label{lemma:orientation_continuous}
The family of closed sets $Z^+(I) \cup f(\partial I)$ varies continuously as a 
function of $I \in \tilde{\M}$, and similarly for $Z^-(I) \cup f(\partial I)$.
\end{lemma}
\begin{proof}
By Lemma~\ref{lemma:boundary_continuous} the sets $\partial f(I)$ and $f(\partial I)$ vary 
continuously with $I$, so all that needs to be shown is that the canonical orientations 
vary continuously. 
If $K$ is a compact subset of $Z^+(I)$ then $K$ is bounded away from $f(\partial I)$ and
therefore there is an open neighborhood of $I$ in $\tilde{\M}$ so that for any $J\in \I$ we have
$K \subset \partial f(J) - f(\partial J)$. But then the preimage of $K$ in $I\cap J$ monotonely
parameterizes $K$ in the same way both in $I$ and in $J$.
\end{proof}

\subsection{Image Relations}

A CaTherine wheel $f$ induces an equivalence relation on $S^1$ where $p \sim q$ if and only if 
$f(p) = f(q)$; we refer to this as the {\em image relation}. Any symmetric reflexive
relation on a set is a subset of the set of unordered pairs on that set containing the
diagonal, and is therefore determined by its restriction to the set of unordered distinct pairs.
We may therefore encode the image relation by a (closed) subset of the space $\M$ of 
unordered pairs of distinct points in $S^1$; evidently
this subset is precisely the set $\LL(f)$ introduced in \S~\ref{definition:degenerate_intervals}.

A subset $K$ of $\M$ corresponds in this way to an equivalence relation on $S^1$ if and only
if it is transitive --- i.e.\/ if $\lbrace p,q\rbrace$ and $\lbrace q,r\rbrace$ are in $K$
where $p,q,r$ are distinct, then $\lbrace p,r\rbrace$ is in $K$.

We shall see that $\LL(f)$ has a canonical decomposition into a
disjoint union $\LL(f) = \LL^+(f) \sqcup \LL^-(f)$ 
where each of $\LL^+(f)$ and $\LL^-(f)$ separately is closed, and comes from an
equivalence relation on $S^1$ in the sense above. 
It will turn out that an unordered pair 
$\lbrace p,q\rbrace$ is in $\LL^+(f)$ resp. $\LL^-(f)$ if and only if $[p,q]$ is 
positively resp. negatively degenerate for $f$. Note that because $\LL(f)$ is a set of 
{\em unordered} pairs it will follow that $[p,q]$ is positively resp. negatively
degenerate for $f$ if and only if $[q,p]=[p,q]^c$ is positively resp. negatively
degenerate for $f$; this is a special case of 
Lemma~\ref{lemma:complementary_orientations} to be proved in the sequel.

The crux of the matter is the fact that positively (resp. negatively) degenerate pairs mapping 
to distinct points in $S^2$ are unlinked in $S^1$:

\begin{proposition}[Classes unlinked]\label{proposition:classes_unlinked}
Suppose $[p,q]$ and $[r,s]$ are both positively degenerate, or both
negatively degenerate for $f$, and $f(p) = f(q) \ne f(r) = f(s)$. Then
the pairs $\lbrace p,q\rbrace$ and $\lbrace r,s\rbrace$ are unlinked in $S^1$.
\end{proposition}
\begin{proof}
It suffices to prove the proposition under the assumption that  $[p,q]$ and $[r,s]$ are 
both positively degenerate for $f$.

Suppose not, so that without loss of generality $[p,s]$ is a closed oriented interval
in which $p < r < q < s$. By hypothesis $[p,s]$ is nondegenerate for $f$. The
interval $[p,s]$ is the union of two adjacent intervals $[p,q]$ and $[q,s]$, the
first positively degenerate for $f$ and the second nondegenerate. Let $x = f(p)=f(q)$
and $y = f(r)=f(s)$; these are both in $\partial f([p,s])$ and the complement
$\partial f([p,s]) - (x\cup y) = Z^+([p,s]) \sqcup Z^-([p,s])$, each of which is nonempty
and is canonically oriented to run from $x$ to $y$. By Lemma~\ref{lemma:adjacent_intervals}
the intersection $\alpha:=f([p,q])\cap f([q,s])$ is a proper interval in the interior of
$f([p,q])$ which contains $x$ and $y$ but can't contain them in the interior, and therefore
$\partial \alpha = x\cup y$. 

\begin{figure}[htpb]
\centering
\includegraphics[scale=0.5]{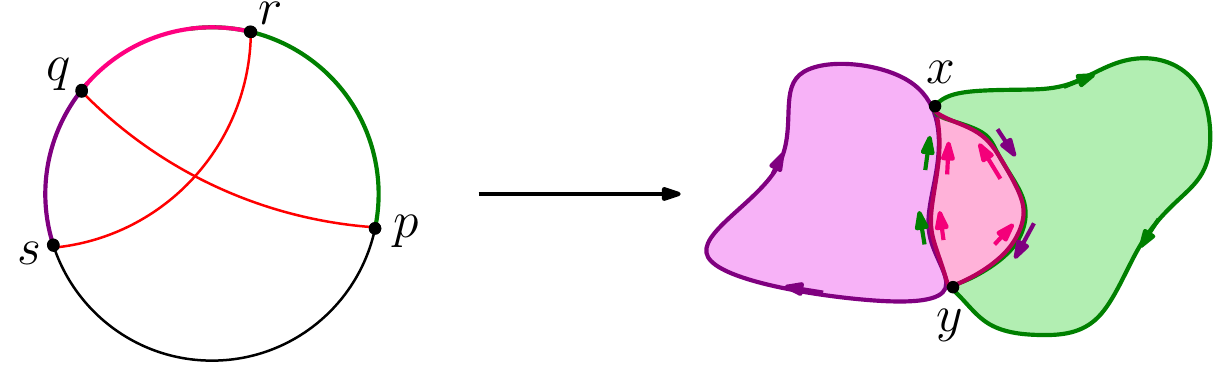}
\caption{The arc $\alpha$ is the common boundary of $f([p,q])\cap f([q,s])$ and
is proper in $f([p,s])$.}
\label{non_crossing}
\end{figure}

By construction each of the (nonempty!) intervals $Z^+([p,s])$ and $Z^-([p,s])$ is entirely contained
in $\partial f([p,q])$ or $\partial f([q,s])$, and since the parameterization of
$\partial f([p,q])$ by its preimage in $[p,q]$ is monotone non-increasing, it follows that
$Z^-([p,s]) \subset \partial f([p,q])$ and $Z^+([p,s]) \subset \partial f([q,s])$.

But $f([r,q])$ is entirely contained in $f([p,q])$ and therefore $Z^+([p,s])$ is contained
in $\partial f([r,s])$. Since $[r,s]$ by hypothesis is positively degenerate, $Z^+([r,s])$
is empty; it follows that the parameterization of $Z^+([p,s]) \subset Z^-([r,s])$ by
its preimage in $[p,s]$ is monotone positive, but by its preimage in $[r,s]$ is monotone negative
which is a contradiction. This proves the lemma.
\end{proof}

\begin{lemma}[Complementary orientations]\label{lemma:complementary_orientations}
Let $I \subset S^1$ be a closed interval. Then $Z^+(I) = Z^+(I^c)$ and $Z^-(I) = Z^-(I^c)$;
equivalently these intervals are equal as subsets, and their canonical orientations 
{\em disagree}.
\end{lemma}
\begin{proof}
We have already seen by Lemma~\ref{lemma:adjacent_intervals} that 
$\partial f(I^c) = \partial f(I)$ and it is also true that $f(\partial I^c)=f(\partial I)$.
Let $Y$ be a component of $\partial f(I) - f(\partial I)$.
It suffices to show that the canonical orientations on $Y$ (coming from $I$ and from $I^c$)
disagree.

The interval $Y$ is monotonely parameterized by subsets $C$ and $C'$ of $I$ and $I^c$ respectively.
For each point $x$ in $Y$ there are $p(x) \in C$ and $p'(x) \in C'$ in the preimage, and therefore
each interval $[p(x), p'(x)]$ is either positively or negatively degenerate. In particular, we
may find at least two distinct points $x,y \in Y$ for which $[p(x),p'(x)]$ and $[p(y),p'(y)]$ are
both positively degenerate or both negatively degenerate, and therefore by 
Proposition~\ref{proposition:classes_unlinked} the pairs $\lbrace p(x),p'(x)\rbrace$ and
$\lbrace p(y),p'(y)\rbrace$ are unlinked in $S^1$. But this implies that the canonical orientations
on $Y$ disagree and the lemma is proved.
\end{proof}

As a special case, if $[p,q]$ is positively degenerate for $f$, then $[q,p] = [p,q]^c$ is
also positively degenerate for $f$. We may therefore unambiguously make the following
definition:

\begin{definition}[Positive and negative relations]\label{definition:signed_relations}
An unordered pair $\lbrace p,q\rbrace$ is in $\LL^+(f)$ resp. $\LL^-(f)$ if and only if 
$[p,q]$ is positively resp. negatively degenerate for $f$.
\end{definition}
Observe that $\LL(f) = \LL^+(f) \sqcup \LL^-(f)$.

\begin{lemma}[Closed subsets]\label{lemma:L_closed_subsets}
Both $\LL^+(f)$ and $\LL^-(f)$ are closed as subsets of $\M$.
\end{lemma}
\begin{proof}
This follows immediately from Lemma~\ref{lemma:orientation_continuous}.
\end{proof}

By abuse of notation, and in anticipation of things to come, we refer to an unordered pair
$\lbrace p,q\rbrace \in \LL(f)$ as a {\em leaf}.

\begin{proposition}[Relations transitive]\label{proposition:relations_transitive}
With notation as above,
\begin{enumerate}
\item{each of the subsets $\LL^+(f)$ and $\LL^-(f)$ are transitive;}
\item{no leaf of $\LL^+(f)$ can have a point in common with a leaf of $\LL^-(f)$;}
\item{two adjacent degenerate intervals in $S^1$ are either both positively
degenerate or both negatively degenerate.}
\end{enumerate}
In particular, each of the relations $\LL^+(f)$ and $\LL^-(f)$ is separately an 
equivalence relation, and the equivalence classes are closed as subsets of $S^1$.
\end{proposition}
\begin{proof}
First observe that all three conditions are equivalent: since $\LL(f)$ is transitive,
and $\LL(f)$ is the disjoint union of $\LL^+(f)$ and $\LL^-(f)$ it follows that each
of these subsets is individually transitive unless they have a pair of leaves with
a point in common; thus (1) and (2) are equivalent. To see that (2) is equivalent to (3),
replace an interval by its complement if necessary and apply
Lemma~\ref{lemma:complementary_orientations}. 

Thus it suffices to show that if adjacent intervals $[p,q]$ and $[q,r]$ are both
degenerate for $f$ then they are either both positively degenerate or both negatively
degenerate. If $[p,q]$ and $[q,r]$ are degenerate then $[p,r]$ is also degenerate for $f$; 
let's suppose for simplicity that it is positively degenerate, so that $Z^+([p,r])$ is
empty. Then $Z^-([p,r])$ contains open subintervals both of $\partial f([p,q]) - f(\partial [p,q])$
and of $\partial f([q,r]) - f(\partial [q,r])$ and therefore $Z^-([p,q])$ and
$Z^-([q,r])$ are nonempty. But these intervals are both degenerate, so they are both positively
degenerate. The case that $[p,r]$ is negatively degenerate is similar.

Since equivalence classes of $\LL^+(f)$ or $\LL^-(f)$ are
equivalence classes of the equivalence relation $\LL(f)$, 
they are point preimages under $f:S^1 \to S^2$ and are therefore closed as subsets of $S^1$.
\end{proof}

Complementary to Proposition~\ref{proposition:classes_unlinked} we have the following:
\begin{lemma}[Linking endpoints]\label{lemma:linking_endpoints}
Let $\lbrace p,q\rbrace$ be an arbitrary unordered pair of distinct points in $S^1$. 
Then either $\lbrace p,q\rbrace$ is a leaf of $\LL^+$,
or there is a disjoint leaf $\lbrace r,s\rbrace$ of $\LL^+$ so that 
$\lbrace p,q\rbrace$ and $\lbrace r,s\rbrace$ link in $S^1$,
and similarly with $\LL^+$ replaced by $\LL^-$.
In particular, a pair $\lbrace p,q\rbrace$ that does not belong to $\LL$ 
is linked by leaves of both $\LL^+$ and $\LL^-$.
\end{lemma}
\begin{proof}
If $\lbrace p,q\rbrace$ is a leaf of $\LL^+$ there is nothing to prove, so suppose not.
Let $I = [p,q]$ and $I^c = [q,p]$. Then by hypothesis $Z^+(I) = Z^+(I^c)$ are nonempty subsets of
$\partial f(I) = \partial f(I^c)$. Choose $s \in (p,q)$ and $r \in (q,p)$ with the
same image $x \in Z^+(I)$. Then $[r,s]$ is degenerate, and we claim it is positively 
degenerate, equivalently that $Z^-([r,s])$ is nonempty. 

Let $J = [p,s]$ and $J' = [r,p]$ so that $[r,s] = J\cup J'$. Then $f(J)$ is a subset
of $f(I)$ with $x \in \partial f(J) \cap \partial f(I)$. Let $\alpha$ be the open subinterval
of $Z^+(I)$ running from $f(p)$ to $x$ and let $C \subset J$ be the preimage of 
$\alpha$. By monotonicity $C$ is contained in $J$, and therefore
$\alpha \subset \partial f(J)$ and in fact $\alpha = Z^+(J)$. On the other hand,
$\partial f(J)$ contains a proper open arc $\beta$ in the interior of $f(I)$ by 
Lemma~\ref{lemma:adjacent_intervals} and $\beta$ is therefore necessarily contained in
$Z^-(J)$. But $\beta$ is also part of $\partial f(J\cup J') = \partial f([r,s])$ and
therefore $Z^-([r,s])$ is nonempty. Since $[r,s]$ is degenerate this implies it is
positively degenerate.
\end{proof}

\subsection{Laminar decomposition}\label{subsection:laminar_decomposition}

Let us summarize where we are. We have demonstrated that the image relation
$\LL(f)$ has a canonical decomposition into disjoint subsets 
$$\LL(f) = \LL^+(f) \sqcup \LL^-(f)$$
which are themselves associated to closed (Lemma~\ref{lemma:L_closed_subsets}) 
unlinked (Lemma~\ref{proposition:classes_unlinked}) equivalence relations
on $S^1$ (Proposition~\ref{proposition:relations_transitive}). Furthermore, 
since each equivalence class in $\LL^\pm(f)$ is an equivalence class in $\LL(f)$,
it is closed as a subset of $S^1$.

Since the concept will arise again in the sequel we formalize these properties as
follows:
\begin{definition}[Laminar equivalence relation]\label{definition:laminar_relation}
An equivalence relation on $S^1$ is {\em laminar} if the equivalence classes are closed
as subsets of $S^1$, if the relation is closed in the space $\M$ of unordered distinct pairs,
and if distinct equivalence classes are unlinked.
\end{definition}
In particular, $\LL^+$ and $\LL^-$ are (associated to) laminar equivalence relations.
We refer to these equivalence relations as the {\em positive image relation} and the 
{\em negative image relation} respectively.

We now use these equivalence relations to build a decomposition of a 2-sphere.
Let's take two copies $P^\pm$ of the hyperbolic plane, each thought of as an open
unit disk in the Poincar\'e model, and identify the ideal boundary of each copy with $S^1$.
Under this identification the union $S^1 \cup P^\pm$ may be naturally topologized as a sphere
that we denote $S^2_\fun$ to distinguish it from `the' $S^2$.

For each equivalence class $\nu$ in the image relation we define a closed subset $C(\nu)$ 
of $S^2_\fun$ as follows:
\begin{enumerate}
\item{if $\nu$ consists of a single point $p\in S^1$ then $C(\nu)=p$;}
\item{if $\nu$ is a nontrivial element of the positive image relation,
then $C(\nu) = \nu \cup H^+(\nu)$ where $H^+(\nu)$ is the convex hull of $\nu$
in $P^+$; and}
\item{if $\nu$ is a nontrivial element of the negative image relation,
then $C(\nu) = \nu \cup H^-(\nu)$ where $H^-(\nu)$ is the convex hull of $\nu$
in $P^-$.}
\end{enumerate}

Note that each $C(\nu)$ is closed as a subset of $S^2_\fun$. Let $\CC$ denote the
collection of subsets $C(\nu)$ as $\nu$ ranges over the equivalence classes of the image
relation.

Each convex hull $H^+(\nu)$ is bounded by a countable family of bi-infinite geodesics in $P^+$
limiting to pairs of points of $\nu$ that bound a maximal open interval in $S^1$ in the 
complement of $\nu$. We refer to a leaf of $\LL^+$ as a {\em boundary leaf} if it corresponds 
to such a boundary component of $H^+(\nu)$ (and similarly for $\LL^-$). Note that we
include the case that $\nu$ consists of only two points, in which case $H^\pm(\nu)$ consists
of a single bi-infinite geodesic whose endpoints should be considered a boundary leaf.
It is immediate from Lemma~\ref{lemma:L_closed_subsets} that the set of boundary leaves is closed.

\begin{figure}[htpb]
\centering
\includegraphics[scale=0.5]{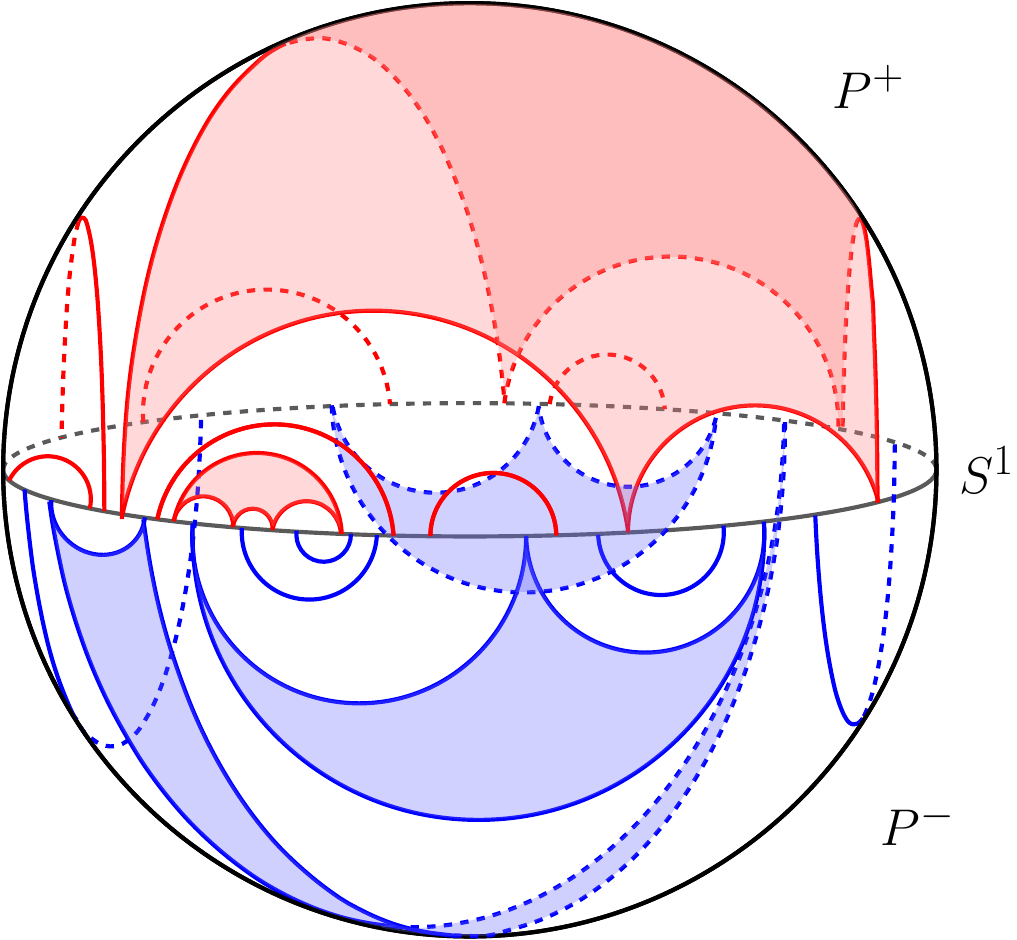}
\caption{Nontrivial subsets $C(\nu)$ in $S^2_f$.}
\label{up_down_decomposition}
\end{figure}

\begin{lemma}[Boundary leaves accumulate]\label{lemma:boundary_leaves_accumulate}
Let $\ell$ be a boundary leaf of $\LL^+$ corresponding to some boundary component of
$H^+(\nu)$. Then $\ell$ is a limit of a sequence of boundary leaves of $\LL^+$ 
on each side (in $P^+$) that does not intersect $H^+(\nu)$. 

Likewise, if $p \in S^1$ is an arbitrary point not in
a leaf of $\LL^+$ then $p$ is a limit of nested boundary leaves of $\LL^+$, i.e.\/ 
leaves whose endpoints approach $p$ from both sides in $S^1$.
\end{lemma}
\begin{proof}
Let $\ell = \lbrace p,q\rbrace$ in such a way that the oriented interval $[p,q] \subset S^1$
does not intersect $\nu$ in its interior. Then by Lemma~\ref{lemma:linking_endpoints},
for every $r \in (p,q)$ there is a positive
leaf $\lbrace s,s'\rbrace$ that links $\lbrace p,r\rbrace$. The endpoints $s,s'$ are themselves
contained in a positive equivalence class $\mu$ that is entirely contained in $(p,q)$
by Proposition~\ref{proposition:classes_unlinked} and therefore without loss of generality we may
assume $\mu$ is disjoint from the intervals $(p,s)$ and $(s',q)$ and therefore that
$\lbrace s,s'\rbrace$ is a boundary leaf of $H^+(\mu)$ separating $H^+(\nu)$ from
$H^+(\mu)$.

We may repeat the argument with $s$ in place of $r$ and obtain a new boundary leaf
$\lbrace t,t'\rbrace$ with $t \in (p,s)$ and $t' \in (s,q)$. Actually, by 
Proposition~\ref{proposition:classes_unlinked} we must have $t' \in (s',q)$. Thus we
obtain a sequence of boundary leaves whose endpoints move monotonely towards $p$ and $q$
respectively. A limit of such a monotone sequence is itself a positive leaf by
Lemma~\ref{lemma:L_closed_subsets} and is either a boundary leaf or is contained in 
an equivalence class with a boundary leaf with both endpoints still closer to $\ell$;
thus by transfinite induction we obtain a (countable) sequence of boundary leaves 
$\lbrace s_i,s_i'\rbrace$ with endpoints moving monotonely $s_i \to u$, $s_i' \to u'$
and such that either $u= p$ or $u'= q$ or both. In fact, by Lemma~\ref{lemma:L_closed_subsets} 
$\lbrace u,u'\rbrace$ is a positive leaf, and since it has at least one endpoint in
common with $\ell$ it is contained in the class of $\nu$. But $\nu$ does not intersect
$(p,q)$ and therefore $u=p$ and $u'=q$ as desired.

The second assertion of the lemma is proved almost exactly the same as the first
assertion, with $p$ in place of the pair $p,q$; i.e.\/ for any $r \ne p$ in $S^1$ 
we may find a positive leaf $\lbrace s,s'\rbrace$ that links $\lbrace p,r\rbrace$ and 
therefore a boundary leaf of some $H^+(\mu)$ separating $p$ from $H^+(\mu)$, and so on.
\end{proof}

Note that if an equivalence class in the positive image relation $\nu$ consists of 
only 2 points so that $H^+(\nu)$ consists of a single geodesic corresponding to a boundary
leaf $\ell$, then $\ell$ is a limit of (positive) boundary leaves on both sides.

Since the concept will be used repeatedly in the sequel, we introduce terminology for
a nested sequence as in the statement of Lemma~\ref{lemma:boundary_leaves_accumulate}:

\begin{definition}[Rainbow]\label{definition:rainbow}
A family of positive resp. negative nested boundary leaves of $\LL^+$ or $\LL^-$
limiting down to a single point of $S^1$ is called a positive resp. negative {\em rainbow}.
\end{definition}

\begin{figure}[htpb]
\centering
\includegraphics[scale=0.5]{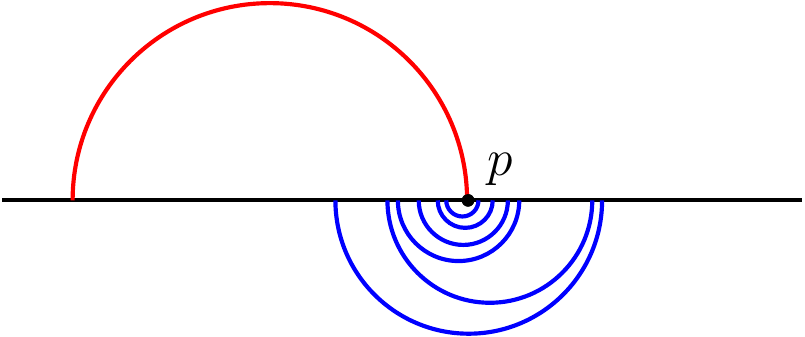}
\caption{A negative rainbow at $p$}
\label{rainbow}
\end{figure}

Thus part of Lemma~\ref{lemma:boundary_leaves_accumulate} may be restated as 
the assertion that an arbitrary $p \in S^1$ not in a leaf of $\LL^+$ has a positive rainbow.

In order to state the next theorem we must recall some elements from decomposition
theory. A good reference is Daverman \cite{Daverman} or the notes of 
Boldizs\'ar Kalm\'ar \cite{Kalmar}. 
We do not state definitions in the greatest possible generality, since our focus is
on decompositions of compact Hausdorff spaces.

\begin{definition}[Upper semi-continuous decomposition]\label{definition:usc_decomposition}
Let $X$ be a compact Hausdorff space. A collection of subsets $\CC$ of $X$ is an
{\em upper semi-continuous decomposition} of $X$ if it satisfies the following properties:
\begin{enumerate}
\item{elements of $\CC$ are closed subsets of $X$;}
\item{distinct elements of $\CC$ are disjoint as subsets of $X$;}
\item{the union of all elements of $\CC$ is $X$;}
\item{(upper semi-continuity): for every element $K$ of $\CC$ and every open neighborhood
$V$ of $K$ in $X$ there is an open neighborhood $U$ of $K$ with $K \subset U \subset V$
such that $U$ is a union of elements of $\CC$.}
\end{enumerate}
\end{definition}

If $\CC$ is an upper semi-continuous decomposition of $X$ we may form the quotient space
$X \to X_\CC$ in which the points of $X_\CC$ correspond to the elements of $\CC$, and
give $X_\CC$ the quotient topology.

In our context where $X$ is a compact Hausdorff space, Proposition~2.9 in \cite{Kalmar} says
that when $\CC$ is an upper semi-continuous decomposition of $X$, the quotient space
$X_\CC$ is also a compact Hausdorff space.

\begin{theorem}[Laminar decomposition]\label{theorem:laminar_decomposition}
Let $f:S^1 \to S^2$ be a CaTherine wheel. Let $\CC$ be the collection of subsets of
$S^2_\fun$ associated to the image relation of $f$ as above. Then $\CC$ 
forms an upper semi-continuous decomposition of $S^2_\fun$. 

Furthermore, the quotient space by this decomposition is homeomorphic to $S^2$, and
there is a unique choice of this homeomorphism such that the restriction of the quotient map 
$F:S^2_\fun \to S^2$ to $S^1$ is equal to $f$.
\end{theorem}
\begin{proof}
The fact that elements of $\CC$ are closed follows from 
Proposition~\ref{proposition:relations_transitive} and the fact that they are 
disjoint follows from Lemma~\ref{proposition:classes_unlinked}.

Let's denote by $|\CC| \subset S^2_\fun$ the union of the elements of $\CC$. We will
show that $|\CC| = S^2_\fun$. Suppose not, so that there is some $x$ in the complement.
Since all of $S^1$ is in $|\CC|$ we may assume without loss of generality that $x \in P^+$.
For every nontrivial positive image relation class $\nu$ the convex hull $H^+(\nu)$ has a unique
boundary leaf in $P^+$ that separates $x$ from the rest of $H^+(\nu)$ (if any). 
We may define a partial order on these boundary leaves where $\ell > \ell'$ if $\ell$ 
separates $\ell'$ from $x$. 

Maximal elements for this partial order exist, since if
$\ell_i \subset H^+(\nu_i)$ is an ascending sequence, some subsequence has a limiting leaf
$\ell_\infty$ in $P^+$ which is necessarily contained in some $H^+(\nu)$ because
$\LL^+$ is closed.  But then $H^+(\nu)$ has some boundary leaf $\ell'$ with
$\ell_i < \ell'$ for all $i$; this proves the claim.

On the other hand, every boundary leaf of every $H^+(\nu)$ is a limit of boundary leaves
from the outside by Lemma~\ref{lemma:boundary_leaves_accumulate}, and therefore maximal
elements can't exist after all. This contradiction proves that $|\CC| = S^2_\fun$.

\begin{figure}[htpb]
\centering
\includegraphics[scale=0.5]{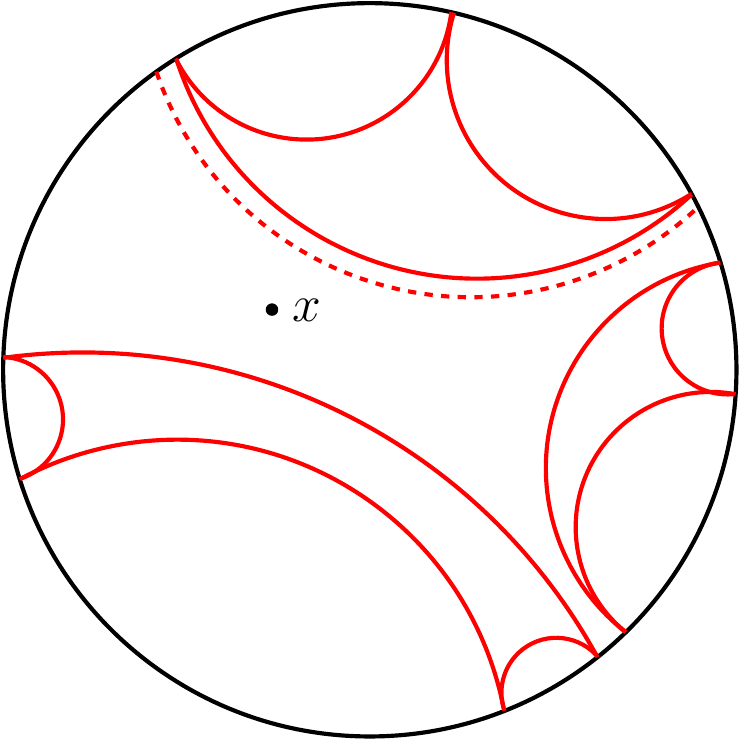}
\caption{The union $|\CC|$ is all of $S^2_\fun$.}
\label{decomposition_fills}
\end{figure}

Finally we must prove upper semi-continuity. Let's fix an element $C(\nu)$. First let's
consider the case that $C(\nu)$ consists of a single point $p\in S^1$. This point is a limit
of rainbows in both $\LL^+$ and $\LL^-$. Choose a boundary leaf of $\LL^+$ that cuts off
an open half-space $A$ in $P^+$. The closure of $A$ meets $S^1$ in a closed interval
$I$ with $p$ in its interior. For every point $q$ in the interior of $I$ that is
contained in a nontrivial leaf of $\LL^+$ there is a rainbow in $\LL^-$ that limits
to $q$, and we choose a boundary leaf of $\LL^-$ in this rainbow that cuts off an 
open half-space $B_q$ in $P^-$ and whose closure meets $S^1$ in an interval $J_q$ in
the interior of $I$. Likewise, we choose a boundary leaf of $\LL^-$ in a rainbow for
$p$ that cuts off $B_p$ in $P^-$ whose closure meets $S^1$ in an interval $J_p$ in the interior
of $I$. Then we may form a set 
$$U: = A \cup B_p \cup \inte(J_p) \bigcup_q B_q \bigcup_q \inte(J_q)$$
where $\inte(\cdot)$ of an interval denotes interior in $S^1$.
One may check that this is open in $S^2_\fun$ and may be chosen to lie inside any open
neighborhood $V$ of $p$. Furthermore by construction any element of $\CC$ that meets
$U$ must be entirely contained in $U$, and therefore $U$ is a union of elements
of $\CC$.

\begin{figure}[htpb]
\centering
\includegraphics[scale=0.5]{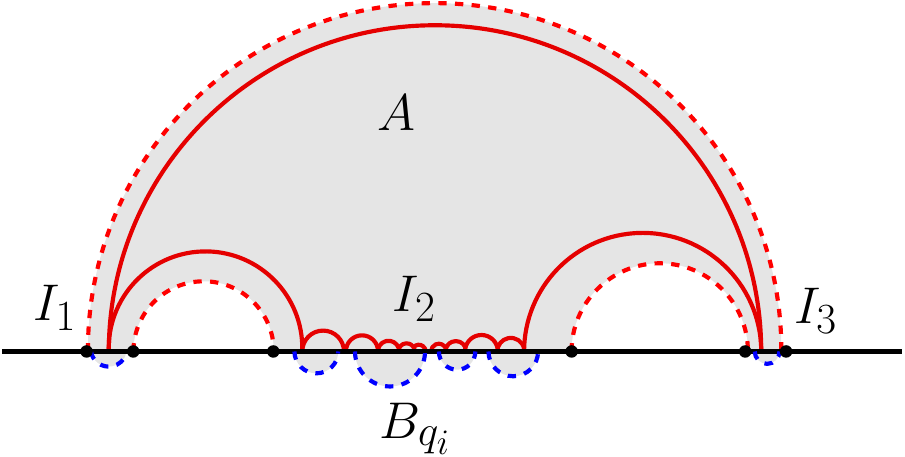}
\caption{The construction of an open neighborhood $U$ which is a union of elements of $\CC$.}
\label{upper_semi_leaf}
\end{figure}

If $\nu$ is a nontrivial class (without loss of generality positive) the boundary
leaves of $C(\nu)$ in $P^+$ are limits of positive boundary leaves, and therefore we
may cut $P^+$ along a finite family of such leaves, one for each boundary leaf of $C(\nu)$
with diameter $>\epsilon$ in a round metric on $S^2_\fun$, to produce an open subspace $A$
containing $H^+(\nu)$, and such that the closure of $A$ meets $S^1$ in finitely many
intervals $I_1,\cdots,I_n$ whose interiors contain $\nu$. As in the previous case,
for every point $q$ in the interior of any $I_i$ that is contained in a nontrivial leaf
of $\LL^+$ there is a rainbow in $\LL^-$ that limits to $q$, and we may choose a
boundary leaf of $\LL^-$ in this rainbow that cuts off an open half-space $B_q$ in
$P^-$ and whose closure meets $S^1$ in an interval $J_q$ in the interior of $I_i$.
Then we may form as above the set
$$U: = A \bigcup_q B_q \bigcup_q \inte(J_q)$$
and check that this is open in $S^2_\fun$ and consists of a union of elements of
$\CC$; furthermore, by choosing $\epsilon$ sufficiently small we may ensure that
$U$ is contained inside any open neighborhood $V$ of $C(\nu)$. This completes the
proof of upper semi-continuity. 

Since each decomposition element of $\CC$ meets $S^1$ in an equivalence class of
the image relation, and since all equivalence classes arise this way, the map
$f:S^1 \to S^2$ factors through the inclusion $i:S^1\to S^2_\fun$, and the
induced map $F:S^2_\fun \to S^2$ for which $Fi=f$ is exactly the map from $S^2_\fun$
to its quotient space under the decomposition $\CC$.
\end{proof}

\begin{remark}
In fact, a theorem of R. L. Moore \cite{Moore} says that if $\CC$ is any upper semi-continuous
decomposition of $S^2$ whose elements are non-separating, the quotient space is
homeomorphic to $S^2$. This consequence may be seen directly in our context without
invoking Moore's theorem; however, we shall use Moore's theorem in the sequel in order
to prove a converse to Theorem~\ref{theorem:laminar_decomposition}.
\end{remark}

\subsection{Zippers}

We may now reinterpret the subsets $Z^\pm(I)$ as the images of bi-infinite geodesics
in the hyperbolic disks $P^\pm$.

For a closed interval $I$ in $S^1$ let's introduce the notation $Z_e^+(I)$ (the
subscript $e$ stands for `extended') for the union of 
$Z^+(I)$ together with certain points in $f(\partial I)$ as follows:
\begin{enumerate}
\item{if $I$ is positive degenerate, then $Z_e^+(I) = f(\partial I)$; and}
\item{if $I$ is not positive degenerate, but $x\in \partial I$ is contained
in some equivalence class $\nu$ of $\LL^+$ that nontrivially intersects the
interiors of both $I$ and $I^c$, then $f(x) \in Z_e^+(I)$.}
\end{enumerate}

Notice that any point of $Z_e^+(I) - Z^+(I)$ is contained in the image of a nontrivial
class of $\LL^+$, and therefore $Z_e^+(I)$ and $Z_e^-(I)$ are still disjoint, 
and furthermore $Z_e^\pm(I) = Z_e^\pm(I^c)$ for any $I$.
For example, if $I$ is positive degenerate, $Z_e^+(I)$ consists of the single point
$f(\partial I)$. Notice that $Z_e^+(I)$ might be open, closed, or neither.

\begin{lemma}[$Z_e$ is image]\label{lemma:Z_is_image}
For any closed interval $I \subset S^1$ let $\gamma^+$ be the oriented bi-infinite geodesic
in $P^+$ resp. $P^-$ joining $I^-$ to $I^+$. Then
$F(\gamma^+) = Z_e^+(I)$ resp. $Z_e^-(I)$ where $F:S^2_\fun \to S^2$ is the quotient map defined in 
Theorem~\ref{theorem:laminar_decomposition},
and furthermore the map $F:\gamma \to Z_e^+(I)$ resp $Z_e^-(I)$ is monotone increasing in the
canonical orientation. 
\end{lemma}
\begin{proof}
Each point $x \in \gamma^+$ lies in exactly one equivalence class $C(\mu)$. Since
$x \in P^+$ this means $x$ is contained in the convex hull $H^+(\mu)$. If $I$ is positive
degenerate, $\gamma^+$ lies entirely in $H^+(\mu)$ and $F(\gamma^+) = f(\mu) = Z_e^+(I)$
by definition. Otherwise $\mu$ contains points on either
side of $\gamma$. In other words, $F(x) = f(\mu) \in f(I) \cap f(I^c) = \partial f(I)$.
If $\mu$ contains a point of $\partial I$ then $F(x)\in Z_e^+(I)$ by definition.
Otherwise $F(x)$ is in one of $Z^+(I)$ or $Z^-(I)$. We shall show it lies in $Z^+(I)$.

Pick a point $p \in \mu \cap \inte(I)$. Then $p$ has a negative rainbow, so there
is a closed negatively degenerate interval $J$ with $p \in \inte(J)$ and $J \subset \inte(I)$.
In particular, $F(x)=f(\mu)=f(p) \in Z^+(J)$ because $\mu$ is a nontrivial positive
equivalence class, while $J^\pm$ are contained in a nontrivial negative equivalence
class. Let $K$ be the interval adjacent to $J$
defined by $K=[I^-,J^-]$. The disk $f(J\cup K)$ is the union of the two disks $f(J)$ and
$f(K)$ glued together along a mutual interval in their common boundaries. 
It is also true that $f(p)\in \partial f(J\cup K)$ because $f(p)\in \partial f(I)$ and 
$J\cup K\subset I$, and we have already seen that $f(p)$ is not equal to $f(I^-)$
or to $f(J^-)$ and therefore $f(p)$ is contained in a 
component $Y$ of $\partial f(J \cup K) - f(\partial (J\cup K))$ and this component 
necessarily contains an interval in common
with $Z^+(J)$. Since the preimage of this interval in $J$ maps monotonely under $f$
in both $J$ and in $J\cup K$, it follows that $f(p) \in Z^+(J\cup K)$.

But now we can repeat the same argument. Let $L$ be the interval adjacent to
$J\cup K$ defined by $L = [J^+,I^+]$. Then $f(I)$ is the union of the disks $f(L)$ and
$f(J\cup K)$ glued together along a mutual interval in their common boundaries, 
and $f(p) \in \partial f(I)$ in the complement of $f(\partial I)$. Thus 
$f(p)$ is contained in a component $Y'$ of $\partial f(I) - f(\partial I)$, and this
component contains an interval in common with $Z^+(J\cup K)$. Since the preimage
of this interval in $J$ maps monotonely under $f$ in both $I$ and in $J\cup K$
it follows that $f(p) \in Z^+(J\cup K)$.

It follows that the image $F(\gamma)$ is contained in $Z_e^+(I)$. By continuity of $F$
its image is path connected, and its closure contains $f(\partial I)$. Thus
$F(\gamma) = Z_e^+(I)$.

Finally, if $x_i \in \gamma$ is a monotone increasing sequence then each $x_i$
is in some $H^+(\mu_i)$ where $\mu_i$ contains a point $p_i \in \inte(I)$ in such a way that
$p_i$ is monotone increasing in $I$. Since the $p_i$ are all in the preimage of
$Z_e^+(I)$ their image in $Z_e^+(I)$ are also monotone increasing for the canonical orientation,
and we are done.
\end{proof}

\begin{theorem}[Zippers]\label{theorem:zippers_from_wheels}
Let $f:S^1 \to S^2$ be a CaTherine wheel. Define $Z^+:=\cup_I Z^+(I)$ and 
$Z^-:=\cup_I Z^-(I)$ as subsets of $S^2$. Then $Z^\pm = F(P^\pm)$ and furthermore,
\begin{enumerate}
\item{$Z^+$ and $Z^-$ are disjoint, dense, consist entirely of cut points,
and each of them is a countable increasing union of finite trees 
(in particular, they are path connected, and in the path topology 
they are each homeomorphic to a $\sigma$-compact topological $\R$-tree); }
\item{each of $Z^+$ and $Z^-$ satisfy the {\em strong landing property} for proper rays;
in other words: every embedded half-open interval $r:[0,1) \to Z^+$ extends
to an embedded closed interval $\bar{r}:[0,1] \to S^2$ in one of three ways:
\begin{enumerate}
\item{(type 0): $\bar{r}(1) \in Z^+$;}
\item{(type 1): $\bar{r}(1) \in S^2 - (Z^+\cup Z^-)$ and there is another embedded
half-open interval $s:[0,1) \to Z^-$ extending to $\bar{s}:[0,1]\to S^2$ with $r(1)=s(1)$; or}
\item{(type 2): $\bar{r}(1) \in Z^-$;}
\end{enumerate}
and conversely every point in $S^2$ is obtained as $\bar{r}(1)$ for some ray
$r$; and similarly for $Z^-$; and}
\item{each of $Z^\pm$ is {\em hairy}: every embedded interval in $Z^+$ has arcs 
branching off it on either side, and similarly for $Z^-$.}
\end{enumerate}
In fact, every $r:[0,1) \to Z^+$ actually has image contained in $Z_e^+(I)$ for some 
closed interval $I\subset S^1$ and similarly for $Z^-$.
\end{theorem}
\begin{proof}
We have already seen by Lemma~\ref{lemma:Z_is_image} that $Z_e^+(I) = F(\gamma^+)$ 
where $\gamma^+ \subset P^+$ is the positively oriented geodesic running from $I^-$ to
$I^+$. Thus $F(P^+) = \cup_I Z_e^+(I)$. To prove the stronger statement that
$F(P^+) = \cup_I Z^+(I)$ it suffices to find, for each $x\in H^+(\mu) \subset P^+$, 
a geodesic $\gamma^+$ containing $x$ whose endpoints are not in $\mu$, so that
$F(x) \in Z_e^+(I) - f(\partial I) = Z^+(I)$.
Since $\mu \subset S^1$ is closed and totally disconnected, this is easy to arrange. This proves
$Z^+ = F(P^+)$ and similarly for $Z^-$.

Since $F$ is the map to the quotient space defined by $\CC$, and since the elements
of $\CC$ that intersect $P^+$ are disjoint from the elements that intersect $P^-$, it
follows that $Z^+$ and $Z^-$ are disjoint. Since the closure of each of $P^\pm$ in
$S^2_\fun$ contains $S^1$, they are dense. Since every nontrivial positive image
equivalence class $\nu \subset S^1$ has at least two points but is nowhere dense, the
complement of $H^+(\nu)$ in $P^+$ has at least two components, and thus every point of
$Z^+$ is a cut point (and similarly for $Z^-$).

Since $P^\pm$ are path-connected, so are $Z^\pm$. It follows, because
the $Z^\pm$ are disjoint, dense and path-connected, that every two
distinct points in $Z^+$ are joined by a unique embedded path in $Z^+$ (and likewise
for $Z^-$) or else an innermost pair of distinct embedded paths between two points of
$Z^+$ would form a Jordan curve which would necessarily intersect $Z^-$.

Here is an explicit way to construct these unique paths. If $p,q \in Z^+$ are any two distinct points, 
they are in the image of equivalence classes $C(\mu)$, $C(\nu)$ that intersect $P^+$
in nonempty convex hulls $H^+(\mu)$, $H^+(\nu)$. If we choose any two points 
$x\in H^+(\mu)$ and $y\in H^+(\nu)$ and let $\delta$ be the (oriented) geodesic segment
joining them, then we claim $F(\delta)$ is an embedded path in $Z^+$ running from $p$
to $q$. To see this, extend $\delta$ to a bi-infinite oriented geodesic $\gamma^+$
running between the endpoints $I^-,I^+$ of a closed interval $I\subset S^1$.
By Lemma~\ref{lemma:Z_is_image} the map $F:\gamma^+ \to Z_e^+(I)$ is a monotone surjection,
and therefore its restriction to $\delta$ is an embedded subpath of $Z_e^+(I)$ running
from $p$ to $q$. Notice by the way that this path is also a subpath of $Z_e^+(J)$ 
for every interval $J$ that intersects both $\mu$ and $\nu$ in $S^1$.

To see that each of $Z^\pm$ is a countable increasing union of finite trees, take a countable
increasing union of finite sided geodesic polygons $P_i \subset P_{i+1} \subset$ whose
union is $P^+$. Any equivalence class $C(\mu)$ that intersects $P_i$ must intersect its
boundary, and therefore $F(P_i)=F(\partial P_i)$ and $\cup F(P_i) = F(P^+) = Z^+$. 
But $F(\partial P_i)$ is a finite union of embedded closed paths, each contained in 
some path $Z^+(I)$. Since there is a unique path in $Z^+$ joining any two points, 
any finite union of embedded closed paths is a finite tree. 

Now let's prove the strong landing property. Let $r:[0,1) \to Z^+$ be an
embedding. For each positive integer $n$ we may choose a point $x_n \in P^+$ with $F(x_n)=r(1-1/n)$
and let $\gamma_n$ be the oriented geodesic in $P^+$ from $x_1$ to $x_n$. Then 
$F(\gamma_n) = r([0,1-1/n])$. There is some subsequence of $x_n$ that converges to some
$x_\infty \in P^+\cup S^1$, and we let $\gamma^+$ be the bi-infinite geodesic in $P^+$
whose closure contains $x_1$ and $x_\infty$. Let $I \subset S^1$ be the closed interval 
whose endpoints are the endpoints of the closure of $\gamma^+$. Then by 
Lemma~\ref{lemma:Z_is_image} it follows that the image of $r$ lies in $Z_e^+(I)$.
From this the strong landing property follows, since we may 
evidently find $\bar{r}$ and $s$ if necessary with image in $\partial f(I)$.

Finally we show that $Z^\pm$ are hairy. Suppose not, so that without loss of generality
there is some oriented segment
$\delta$ in $Z^+$ (say) which has no branching on the right hand side. Let $\gamma$ be
an oriented geodesic in $P^+$ containing a segment $\gamma'\subset \gamma$ mapping to $\delta$.
Consider the intersection of $\gamma'$ with decomposition elements $H^+(\mu)$, and 
for each of these consider the intersection of $H^+(\mu)$ with the interval of $S^1$ 
on the right of $\gamma$. The hypothesis of no branching means that each such $H^+(\mu)$ 
intersects this interval in $S^1$ in exactly one point, and therefore there is a nontrivial
interval $J$ in $S^1$ consisting entirely of endpoints of distinct nontrivial decomposition
elements of $\LL^+$. But then any element of $\LL^-$ with an endpoint in $J$ would create
a perfect fit. This contradiction completes the proof.
\end{proof}

The objects $Z^\pm \subset S^2$ are an example of a {\em zipper}, whose precise
definition we shall give in \S~\ref{subsection:zipper_definitions}.

\section{Laminar Relations}

The main goal of this section is to prove Theorem~\ref{theorem:wheels_from_laminations} 
which is a converse to 
Theorem~\ref{theorem:laminar_decomposition}. Namely, we shall characterize exactly
which pairs of laminar relations $\LL^+$, $\LL^-$ can arise from CaTherine wheels.
This will be useful in the sequel for constructing examples with desirable properties.

\subsection{Definitions}

Although it is not strictly necessary, it is convenient to express things in the 
language of laminations, since this makes it easier to connect the contents of this paper
to the literature. See for example Baik--Kim \cite{Baik_Kim} for an introduction to the theory
and significance of laminations in low-dimensional topology, or 
Thurston \cite{Thurston_degree_d}.

\begin{definition}[Lamination]\label{definition:lamination}
A {\em lamination} is a closed subset of the space $\M$ of unordered pairs of points
in $S^1$ such that no two elements of this subset are linked in $S^1$. The elements
of a lamination are called {\em leaves}. 
\end{definition}

It is convenient to alternate between thinking of a leaf of a lamination as an unordered
pair of points in $S^1$, or as a geodesic in $\HH^2$; likewise, we will think of a nontrivial
equivalence class $\nu$ of a laminar relation either as a closed subset of $S^1$, or
as its convex hull $H(\nu)$ in $\HH^2$.

Thinking of leaves as geodesics in $\HH^2$, it makes sense to think of each leaf as having
two `sides'. If $\Lambda$ is a lamination, a leaf $\lambda$ of $\Lambda$ may be 
a limit of leaves of $\Lambda$ on both sides, on one side, or on neither side. 
A leaf which is a limit of leaves of $\Lambda$ on neither side is an {\em isolated leaf}.

There is a close relationship between laminations and laminar relations, although this
relationship has some subtlety. This is explained carefully in \cite{Thurston_degree_d}.
We do not pursue this matter in the greatest possible generality, restricting ourselves
to the laminations and laminar relations that are important for CaTherine wheels.

First of all, a laminar relation determines a lamination as follows:

\begin{definition}[Boundary lamination]\label{definition:boundary_lamination}
A laminar relation $\LL$ determines a lamination $\Lam(\LL)$ whose elements
are the boundary leaves of $\LL$, i.e.\/ the boundaries of the convex hulls of the
equivalence classes of $\LL$. 

To be precise, if $\nu$ is a nontrivial equivalence class of $\LL$, we may take the
convex hull $H(\nu)$. If $|\nu|=2$ then $H(\nu)$ consists of a single geodesic,
and this geodesic is a leaf of $\Lam(\LL)$. If $|\nu|>2$ then $H(\nu)$ is a convex
polygon, and the sides of $H(\nu)$ are leaves of $\Lam(\LL)$.

A lamination that arises this way is called 
a {\em boundary lamination}.
\end{definition}

A lamination $\Lambda$ generates a laminar relation $\LL$ in several different ways,
which we now distinguish:
\begin{definition}[Two relations]\label{definition:two_relations}
Let $\Lambda$ be a lamination. The {\em little relation}, denoted $\rel(\Lambda)$, 
is the smallest closed equivalence relation on unordered pairs that contains $\Lambda$.

The {\em big relation}, denoted $\Rel(\Lambda)$ is the smallest closed equivalence
relation on unordered pairs that contains $\Lambda$, and contains 
every unordered pair that crosses only countably many leaves of $\Lambda$.
\end{definition}
It is straightforward to check that both $\rel(\Lambda)$ and $\Rel(\Lambda)$ are
laminar relations. Notice that every two distinct elements of $\Rel(\Lambda)$
are separated from each other in the disk 
by uncountably many distinct elements of $\Rel(\Lambda)$; this is an essential 
property if we want the quotient of the disk by this relation to be Hausdorff
and path-connected. 

\begin{remark}
Thurston defines only the little relation in \cite{Thurston_degree_d}, and denotes
it $\Rel$, which is an unfortunate clash of notation. For the relations that arise
from CaTherine wheels it is the big relation that is important, as we shall see, 
and it seems to make sense for us to denote it by $\Rel$.
\end{remark}

In general, neither $\rel$ nor $\Rel$ are inverse to $\Lam$. However for the laminar
relations that arise from CaTherine wheels the situation is cleaner.

\begin{proposition}[Back and forth]\label{proposition:back_and_forth}
Suppose $\Lambda$ is a boundary lamination. Then no point $p\in S^1$ is contained in
more than two leaves of $\Lambda$.

Conversely, suppose $\Lambda$ is a lamination such that
\begin{enumerate}
\item{no point $p \in S^1$ is contained in more than two leaves of $\Lambda$; and}
\item{$\Lambda$ has no isolated leaves.}
\end{enumerate}
Let $\LL:=\Rel(\Lambda)$. Then $\Lam(\LL) = \Lambda$; in particular, 
$\Lambda$ is a boundary lamination.
\end{proposition}
\begin{proof}
Suppose $\LL$ is a laminar relation, and $\Lambda = \Lam(\LL)$. Let $p \in S^1$ 
and let $\nu \subset S^1$ be the equivalence class of $\LL$ containing $p$. 
Then $\lbrace p,q\rbrace$ is a leaf of $\Lambda$ if and only if $(p,q) \subset S^1$ 
or $(q,p) \subset S^1$ is a maximal open interval in the complement of $\nu$; 
in particular, at most two leaves of $\Lambda$ may contain $p$.

Conversely, suppose $\Lambda$ is a lamination satisfying the two properties above.
Let's realize $\Lambda$ as a geodesic lamination in $\HH^2$. Since $\Lambda$ has no
isolated leaves, each leaf $\lambda$ is a limit of leaves $\lambda_i$ from at least
one side; and furthermore, the approximating leaves $\lambda_i$ can have no endpoints
in common with $\lambda$ or we would violate bullet (1). A countable closed set
of leaves must have some isolated leaves, by considering Cantor--Bendixson rank;
thus the nontrivial equivalence classes of $\Rel(\Lambda)$ are of exactly two kinds:
the endpoints of leaves $\lambda$ of $\Lambda$ that are not isolated on either side,
and the closures in $S^1$ of components $U$ of $\HH^2 -\Lambda$. From this it
follows that $\Lam(\LL) = \Lambda$ as claimed.
\end{proof}

\begin{example}[Two relations]\label{example:two_relations}
One must be careful interpreting Proposition~\ref{proposition:back_and_forth}. 
Suppose $\Lambda$ is a boundary lamination with no isolated leaves, 
and let $\LL:=\Rel(\Lambda)$ so that
$\Lam(\LL)=\Lambda$. Suppose that $\LL$ contains a nontrivial equivalence class
$\nu \subset S^1$ which is a Cantor set. Let $\LL'$ be the laminar relation whose
equivalence classes are of two sorts: the classes of $\LL$ other than $\nu$, and
the pairs of points $p,q \in \nu$ for which the interval $[p,q]$ intersects $\nu$
only at its endpoints. Then $\Lam(\LL')=\Lambda$.
\end{example}

In order to distinguish between the laminar relations $\LL$ and $\LL'$ from
Example~\ref{example:two_relations} we make the following definition:

\begin{definition}[No isolated sides]\label{definition:no_isolated_side}
A laminar relation $\LL$ has {\em no isolated sides} if, for $\Lambda:=\Lam(\LL)$,
each leaf $\lambda$ of $\Lambda$ is either a nontrivial limit of leaves on both sides,
or it is a nontrivial limit of leaves on one side, and is a boundary leaf of $H(\mu)$
for some equivalence class $\mu$ of $\LL$ with $|\mu|\ge 3$ on the other side.
\end{definition}

\begin{lemma}[No isolated side]\label{lemma:no_isolated_side}
Suppose $\Lambda$ is a boundary lamination with no isolated leaves. Then 
$\Rel(\Lambda)$ has no isolated sides.
\end{lemma}
\begin{proof}
The proof is immediate.
\end{proof}

Before we state the main theorem of this section, 
we must state some elements of decomposition theory,
and a theorem of R. L. Moore \cite{Moore} about 
upper semi-continuous decompositions of 2-manifolds. 
We state the theorem only in the case that the 2-manifold
in question is either a sphere or a disk, since these are the only cases we need to use,
but the reader should know that Moore's theorem holds more generally. An elegant modern
proof is given by Timorin \cite{Timorin}.

First a couple of definitions.
\begin{definition}[Cellular]\label{definition:cellular}
A closed subset $K$ of a manifold is {\em cellular} if it is the intersection of
a nested sequence of closed balls.
\end{definition}

\begin{definition}[Shrinkable decomposition]\label{definition:shrinkable}
An upper semi-continuous decomposition $\CC$ of $X$ is {\em shrinkable} if for every
open cover $\mathcal{V}$ and every $\CC$-saturated open cover $\mathcal{U}$ (i.e.\/ one
for which each $U\in \mathcal{U}$ is a union of elements of $\CC$) there is a
homeomorphism $h:X \to X$ such that for every $K\in \CC$ the set $h(K)$ is contained
in $V$ for some $V\in \mathcal{V}$, and for every $x\in X$ there is $U\in \mathcal{U}$ 
so that $x,h(x) \in U$.
\end{definition}

The following theorem is essentially due to Bing \cite{Bing}; for a proof see e.g.\/
Theorem~4.4 from \cite{Kalmar}:
\begin{theorem}[Shrinkable is approximable]\label{theorem:shrinkable}
If $X$ is a complete metric space and $\CC$ is a shrinkable upper semi-continuous
decomposition, then the map $f:X \to X_\CC$ from $X$ to its quotient space is a
near-homeomorphism; i.e.\/ for every open covering $\mathcal{U}$ of $X_\CC$ there is a homeomorphism
$h: X \to X_\CC$ such that for every $x\in X$ the points $f(x)$ and $h(x)$ are in some
$U \in \mathcal{U}$.
\end{theorem}

Notice in particular that if $\CC$ is shrinkable, $X_\CC$ is homeomorphic to $X$.
Finally we state Moore's theorem:

\begin{theorem}[Moore; decompositions of 2-manifolds]\label{theorem:Moore}
Let $X$ be one of $S^2$ or $D^2$ and let $\CC$ be an upper semi-continuous decomposition 
of $X$. Suppose every element of $\CC$ is cellular, and if $X$ is $D^2$ suppose further
that every element is non-separating. Then $\CC$ is shrinkable.
\end{theorem}

Together with Theorem~\ref{theorem:shrinkable} this implies that $X_\CC$ is
homeomorphic to $X$, and the map $X \to X_\CC$ is approximable by homeomorphisms.

We are now ready to state and prove the main theorem of this section:

\begin{theorem}[CaTherine wheels from laminations]\label{theorem:wheels_from_laminations}
Let $\LL^\pm$ be a pair of laminar relations. Then $\LL^\pm$ arise from a CaTherine wheel
$f:S^1 \to S^2$ if and only if 
\begin{enumerate}
\item{(no isolated sides): the laminar relations have no isolated sides; and}
\item{(no perfect fits): no pair of leaves of the laminar relations $\lambda^\pm \in \LL^\pm$ 
have an endpoint in common.}
\end{enumerate}
\end{theorem}
\begin{proof}
First we show that the conditions are necessary. The fact that $\LL^\pm$ have no
isolated sides follows from Lemma~\ref{lemma:boundary_leaves_accumulate}, and
the fact that no pair of leaves of $\LL^\pm$ have an endpoint in common follows from
Proposition~\ref{proposition:relations_transitive}.

Now we show that the conditions are sufficient. 
We repeat the constructions of an
upper semi-continuous decomposition $\CC$ of a 2-sphere as in 
\S~\ref{subsection:laminar_decomposition} as follows. First, build a 2-sphere 
$S^2_\Lambda$ from the union of $S^1$ with two hyperbolic planes $P^\pm$. 
Associated to $\LL^+$ we form a decomposition $\CC$ of $S^2_\Lambda$ whose elements
are single points $p\in S^1$ not contained in either of $\LL^\pm$, and
convex hulls $C(\nu):=\nu\cup H^+(\nu)$ for $\nu$ a nontrivial equivalence class
associated to $\LL^+$, and likewise for $\LL^-$.

We claim that $\CC$ forms an upper semi-continuous decomposition of $S^2_\Lambda$. The
proof very closely mirrors the proofs of Lemma~\ref{lemma:boundary_leaves_accumulate} and
Theorem~\ref{theorem:laminar_decomposition} and we are consequently brief.
First we show that an arbitrary $p\in S^1$ not in a leaf of $\LL^+$ resp. $\LL^-$ has a 
positive resp. negative rainbow.

\begin{lemma}[Rainbow]\label{lemma:rainbow}
Suppose $\LL^\pm$ has no isolated sides, and no perfect fits. Then
every $p \in S^1$ not contained in a leaf of $\LL^+$ resp. $\LL^-$ has a positive
resp. negative rainbow.
\end{lemma}
\begin{proof}
Define $\Lambda^\pm:=\Lam(\LL^\pm)$. We define a partial order on 
$\Lambda^+$ by $\mu > \lambda$ if the leaf $\mu$ separates $p$ from $\lambda$. Choose some maximal 
ascending chain. Any such chain must have a 
cofinal countable increasing subchain; call it $\lbrace s_i,s_i'\rbrace$ where
$s_i$ and $s_i'$ approach $p$ monotonely from either side. None of $s_i$ or $s_i'$ can
be equal to $p$ since by hypothesis $p$ is not contained in a leaf of $\LL^+$. Furthermore
we may assume that no two succesive $s_i$ or $s_i'$ are equal, since otherwise 
$\lbrace s_i,s_i'\rbrace$ and $\lbrace s_{i+1},s_{i+1}'\rbrace$ are both boundary leaves
of the same element of $\LL^+$ and we may choose the greater one and discard the other.
We claim this sequence forms a rainbow. For if not, we may extract a limit
$\lbrace s_\infty,s_\infty'\rbrace$ which is a leaf of $\Lambda^+$ because $\Lambda^+$
is closed. But $\Lambda^+$ is a boundary lamination, and therefore either
$\lbrace s_\infty,s_\infty'\rbrace$ is a limit of an infinite {\em decreasing} 
sequence of leaves from above, or it is the boundary of $H(\mu)$ for some nontrivial equivalence
class $\mu$ of $\LL^+$, and some other leaf of $\partial H(\mu)$ separates $p$ from
$\lbrace s_\infty,s_\infty'\rbrace$; in particular, the original sequence was not cofinal after
all, contrary to hypothesis. Thus this sequence forms a rainbow as claimed.
\end{proof}

We are now in a position to prove that $\CC$ is an upper semi-continuous decomposition.
The classes of $\CC$ are closed and disjoint by construction. The argument in
Theorem~\ref{theorem:laminar_decomposition} that the union of $\CC$ is all of $S^2_f$
only used the fact that $\LL^\pm$ are closed, and that they have no isolated sides.
Likewise, the proof that the decomposition is upper semi-continuous only used the
fact that the $\LL^\pm$ have no isolated side, and the existence of rainbows as 
proved in Lemma~\ref{lemma:boundary_leaves_accumulate}. Thus we may repeat the proof,
using Lemma~\ref{lemma:rainbow} and the hypothesis that $\LL^\pm$ have no isolated sides.
In conclusion, $\CC$ is an upper semi-continuous decomposition of $S^2_\Lambda$.

At this point we appeal to Moore's Theorem~\ref{theorem:Moore} and Bing's
Theorem~\ref{theorem:shrinkable}. Let $F$ be the map from $S^2_\Lambda$ to its quotient
space, and let $f$ be the restriction to $S^1$. The decomposition $\CC$ is evidently
cellular, so the quotient space is homeomorphic to a sphere, and we may choose an
identification of it with $S^2$ so that $F:S^2_\Lambda \to S^2$ and $f:S^1 \to S^2$.
To see that $f$ is a CaTherine wheel, choose a closed interval $I$. This bounds closed
half spaces $H^\pm(I)$ in $P^\pm$, and the union $D(I):=H^\pm(I) \cup I$ is a disk. The
restriction of the elements of $\CC$ to this disk is an upper semi-continuous 
decomposition, and the elements are still cellular and non-separating, since each
decomposition element is either entirely contained in $H^+(I) \cup I$ or $H^-(I)\cup I$.
Thus $F(D(I))$ is a disk. Since each decomposition element of $\CC|D(I)$ meets $I$,
we also have that $f(I)$ is a disk. Since $\partial I \in \partial D(I)$ it follows
that $f(\partial I) \subset \partial f(I)$ and the theorem is proved.

\begin{figure}[htpb]
\centering
\includegraphics[scale=0.5]{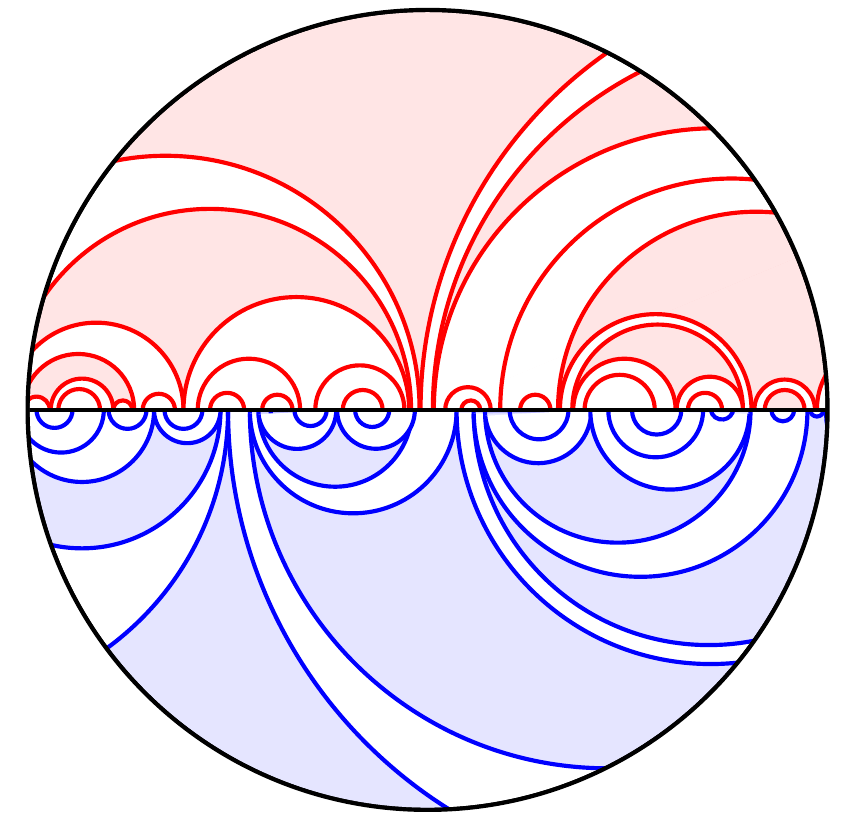}
\caption{The restriction of the decomposition of $S^2_f$ to the disk $D(I)$.}
\label{laminations_interval_restrict}
\end{figure}
\end{proof}

\begin{remark}
If each equivalence class of $\LL^\pm$ is finite, then $\LL^\pm$ have no perfect fits
if and only if $\Lambda^\pm$ have no perfect fits.
\end{remark}

\begin{remark}
A pair of laminar relations $\LL^\pm$ with no perfect fits and no isolated sides are
a special case of what Frankel--Landry call {\em especial pairs} in
\cite{Frankel_Landry_flows}, the only difference being that each pair of equivalence classes
$\nu^\pm$ in the equivalence relations associated to $\LL^\pm$ is allowed to have at most
one point in $S^1$ in common. Especial pairs give rise to a pair of upper
semi-continuous decompositions of
the disk, but the existence of perfect fits causes complications when one tries to 
extend these to upper semi-continuous decompositions of $S^2$, and it is by no
means clear that the decompositions one ultimately obtains satisfy the hypotheses of
Moore's theorem. In particular, some especial pairs give rise to surjective maps 
$f:S^1 \to S^2$ and some don't, and it seems like a challenging question to determine
those that do. One approach to this question is developed in \S~\ref{section:P_wheels}.
\end{remark}

By the way, Moore's Theorem~\ref{theorem:Moore} and Bing's Theorem~\ref{theorem:shrinkable}
together prove that every CaTherine wheel can be approximated by embeddings. In fact
we shall prove a sharper theorem in \S~\ref{section:embeddings}, 
building on some results of Kerbs \cite{Kerbs}.

We end this subsection by stating a simple criterion for a pair of laminations to
satisfy the hypothesis of Theorem~\ref{theorem:wheels_from_laminations}.

\begin{proposition}
Suppose $\LL^\pm$ are a pair of laminar relations without perfect fits, and
let $\Lambda^\pm = \Lam(\LL^\pm)$ be the associated boundary laminations. 
Suppose further that $\LL^\pm$ 
satisfy the conclusion of Lemma~\ref{lemma:linking_endpoints}, i.e.\/ 
that every distinct pair of points in $S^1$ is either in $\LL^+$ or links an
element of $\LL^+$ (and similarly for $\LL^-$). Then $\Lambda^\pm$
has no isolated leaves, and $\LL^\pm$ have no isolated sides.
\end{proposition}
\begin{proof}
We may simply repeat the proof of the first part of 
Lemma~\ref{lemma:boundary_leaves_accumulate}, using the conclusion of 
Lemma~\ref{lemma:linking_endpoints} and the fact that $\LL^\pm$ are closed and
equivalence classes are unlinked (by definition).
\end{proof}

\section{Zippers}\label{section:zippers}

The main goal of this section is to prove Theorem~\ref{theorem:wheels_from_zippers}
which is a converse to Theorem~\ref{theorem:zippers_from_wheels}. Namely, we shall 
characterize exactly which zippers $Z^\pm$ can arise from CaTherine wheels. 

\subsection{Definitions}\label{subsection:zipper_definitions}

In this subsection we give the definition of a zipper, and recall some basic properties
and constructions from \cite{Calegari_Loukidou_Zippers}. We are concerned with
the path connectivity of certain subsets of $S^2$; thus throughout this section `component'
means path component, and `cut point' means path cut point (i.e.\/ a point in a
path connected set whose removal creates at least two distinct path components).

\begin{definition}[Zipper]\label{definition:zipper}
A pair of nonempty subsets $Z^\pm$ of $S^2$ are a {\em zipper} if they satisfy the following
conditions:
\begin{enumerate}
\item{$Z^+$ and $Z^-$ are disjoint;}
\item{each of $Z^\pm$ is dense in $S^2$;}
\item{every point of $Z^\pm$ is a cut point; and}
\item{each of $Z^\pm$ is a countable increasing union of finite trees.}
\end{enumerate}
\end{definition}

\begin{remark}\label{remark:two_zippers}
In fact this definition diverges very slightly from the definition of zippers
in \cite{Calegari_Loukidou_Zippers}. The context of that paper is 
that one has a cocompact Kleinian group $G$ acting on $S^2_\infty$,
and in that context, what is called a zipper is a 
pair of nonempty $G$-invariant 
subsets $Z^\pm \subset S^2_\infty$ that are disjoint and path-connected.
Since $G$ acts minimally on $S^2_\infty$, both of $Z^\pm$ are necessarily dense. 
Furthermore, it is shown (Proposition~2.10 in that paper) that
any zipper in this sense contains a minimal subzipper, for which every point of 
$Z^\pm$ is a cut point, and each of $Z^\pm$ is 
necessarily a countable increasing union of finite trees. In particular, a minimal zipper 
in the sense of \cite{Calegari_Loukidou_Zippers} is a zipper in the sense of
Definition~\ref{definition:zipper}, and conversely a zipper in the sense of 
Definition~\ref{definition:zipper} invariant under a cocompact Kleinian group $G$ 
contains a minimal subzipper in the sense of \cite{Calegari_Loukidou_Zippers}. 
In this paper by abuse of notation we will reserve the word zipper for a zipper
in the sense of Definition~\ref{definition:zipper}, and will refer to a minimal
zipper in the sense of \cite{Calegari_Loukidou_Zippers} as a $G$-zipper.

In the sequel we will appeal to certain constructions and facts proved about
$G$-zippers in \cite{Calegari_Loukidou_Zippers} that do not use $G$-invariance, and
therefore hold also for zippers.
\end{remark}

\begin{lemma}[Unique path]\label{lemma:unique_path}
Let $Z^\pm$ be a zipper. Then any two distinct points in $Z^+$ may be joined by a
unique embedded path in $Z^+$, and similarly for $Z^-$.
\end{lemma}
\begin{proof}
Bullet (4) from the definition implies that each of $Z^\pm$ is path-connected. If
$p,q\in Z^+$ are joined by distinct embedded paths $\alpha,\beta$ then by restricting
to smaller subpaths if necessary, we may assume $\alpha \cup \beta$ is a Jordan curve
in $Z^+$, so that by disjointness, $Z^-$ is forced to lie in one of the complementary 
components. But this violates density.
\end{proof}

The set of {\em ends} of a zipper $Z^\pm$, denoted $\EE^\pm$, is the set of topological ends 
in the sense of Freudenthal (although we will give them a different topology); 
by bullet (4) of Definition~\ref{definition:zipper}
we can take $\EE^+ = \varprojlim \pi_0(Z^+ - T)$ where the
limit is taken over all (compact) finite subtrees $T \subset Z^+$, and similarly for
$Z^-$. Thus one may easily see that the ends of $Z^\pm$
are in bijection with equivalence classes of (unparameterized) proper rays, 
where $r \sim r'$ if they are equal away from some pair of compact initial segments. 

It is convenient to define the convex hull of a subset $S \subset Z^+\sqcup \EE^+$.
This is exactly what the name suggests: it is the union of all embedded paths (finite
or infinite) `between' pairs of distinct elements of $S$ (with some exceptional 
cases when $|S|<2$) but we spell it out:

\begin{definition}[Convex hull]\label{definition:convex_hull}
Let $S$ be a subset of $Z^+ \sqcup \EE^+$. If $S$ is empty or consists of a single point of
$\EE^+$, the convex hull of $S$ is empty. Otherwise the convex hull of $S$ is the smallest 
path-connected subset of $Z^+$ containing $S \cap Z^+$ and containing rays representing
every element of $S \cap \EE^+$.
\end{definition}

The key property of zippers that we shall use is that the sets $\EE^+$ have a natural
{\em circular order}, which can be completed by adding certain objects called {\em ideal
ends}. 

\begin{lemma}[Circular order]\label{lemma:circular_order}
The end spaces $\EE^\pm$ each admit a canonical circular order. 
\end{lemma}
This is proved in \cite{Calegari_Loukidou_Zippers}, Lemma~2.13. Since we need to 
understand the relationship of this order to the topology of $Z^\pm$ and the way they
are embedded in $S^2$ we describe the circular order explicitly. To give a circular
order on a set is to give the assignment clockwise or anticlockwise to each distinct
ordered triple of elements in a set, in a way which is invariant under cyclic permutation,
and is compatible on distinct 4-tuples of elements.
If $e_1,e_2,e_3\in \EE^+$ are three distinct ends, we may find disjoint representative
rays $r_1,r_2,r_3$ that begin on some compact subtree $K$. We may then find a closed
disk $D$ in $S^2$ that contains $K$ in the interior but does not contain all of any
of the $r_i$. Let $p_i$ be the first point of intersection of $r_i$ with $\partial D$.
The Jordan curve $\partial D$ inherits an orientation from $D$ which in turn inherits it from
$S^2$, and $p_1,p_2,p_3$ thereby inherit a circular order from $\partial D$. This is
easily seen to be independent of choices, and defines the circular order on $\EE^+$
and likewise on $\EE^-$.

Each of the subsets $\EE^\pm$ gets a natural {\em order topology} from its circular
order (this is quite different from the usual Freudenthal topology on the end space).
Any circularly ordered set is Hausdorff in its order topology. The spaces $\EE^\pm$ furthermore are
separable, by bullet (4) of Definition~\ref{definition:zipper}. 

A set with a separable (circular) order topology has a natural separable order completion. We
denote these completed spaces by $\bar{\EE}^\pm$. The points of $\bar{\EE}^\pm - \EE^\pm$
are represented by geometric objects called {\em ideal gaps} that we now define,
following \cite{Calegari_Loukidou_Zippers}, \S~2.5.

If $e^L,e^R \in \EE^+$, we define $[e^L,e^R] \subset \EE^+$ to be the set
$$[e^L,e^R]:= \lbrace e^L,e^R\rbrace \cup \lbrace e \in \EE^+ \text{ such that } 
e^L,e,e^R \text{ is anticlockwise}\rbrace$$

\begin{lemma}[Nonempty intersection]\label{lemma:nonempty_intersection}
Suppose there is an infinite nested sequence
$$ \cdots [e_{i+1}^L,e_{i+1}^R] \subset [e_i^L,e_i^R] \subset \cdots$$
whose intersection is empty. For each $i$ let $\gamma_i\subset Z^+$ denote the convex hull of
$e_i^L$ and $e_i^R$ (this is a properly embedded line), oriented to run from $e_i^L$ to $e_i^R$. 
Then there is some $i$ so that the intersection $X: = \cap_{j\ge i} \gamma_j$ is either
a point, or a compact, nonempty interval.
\end{lemma}
This is parallel to \cite{Calegari_Loukidou_Zippers}, Lemma~2.15 and the proof is identical.
Note that we do not assume the $e_i^L$ or the $e_i^R$ are all distinct, though of course
the hypothesis of the lemma implies that at least one of the sequences may be chosen to
be so.

We are now in a position to define ideal gaps.

\begin{definition}[End nested; ideal gap]\label{definition:ideal_gap}
A sequence of proper lines $\gamma_i$ in $Z^\pm$ which are the convex hulls of a
sequence of nested intervals $[e_i^L,e_i^R]$ in $\EE^\pm$ is said to be {\em end nested}.

A point $p \in Z^\pm$ together with an end nested sequence of proper lines
$\gamma_i$ associated to a nested sequence of closed intervals in $\EE^\pm$ with empty
intersection for which $p \in \cap_i \gamma_i$ (as in the statement of 
Lemma~\ref{lemma:nonempty_intersection}) is called an {\em ideal gap}. Two ideal
gaps are equivalent if they contain cofinal end nested subsequences of some end 
nested sequence.
\end{definition}

\begin{figure}[htpb]
\centering
\includegraphics[scale=0.5]{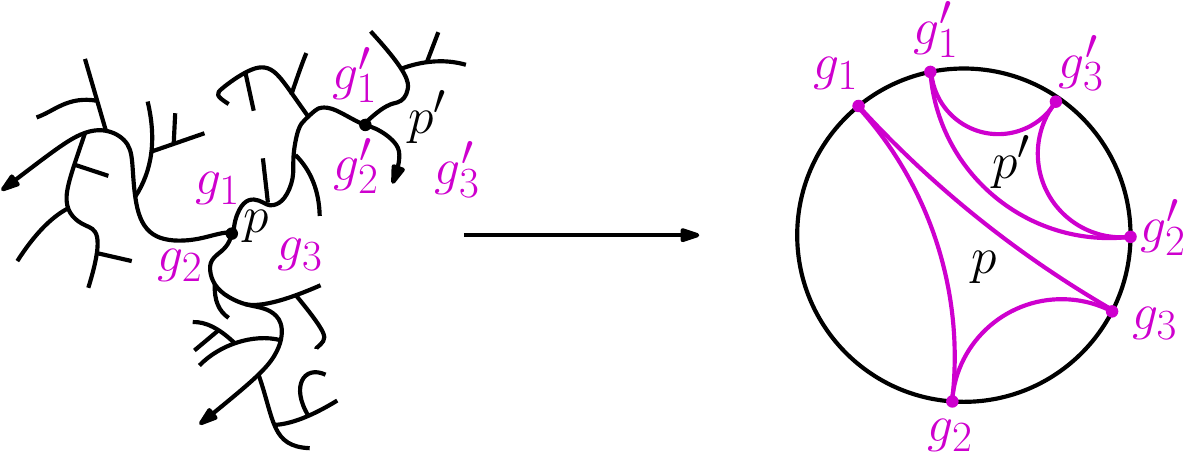}
\caption{Branch points in $Z^+$ are in the support of $>2$ ideal gaps.}
\label{ideal_gap}
\end{figure}

Informally, an ideal gap is an equivalence class of pairs consisting of a
point $p \in Z^+$ together with
a `direction of approaching' $p$ in the complement $S^2 - Z^+$ (and similarly for $Z^-$).
A point $p$ as above is said to be contained in the {\em support} of the ideal gap (and by abuse of notation,
the support of an equivalence class). From the proof of Lemma~\ref{lemma:nonempty_intersection} 
(see \cite{Calegari_Loukidou_Zippers}, Lemma~2.15) one may see that each equivalence 
class of ideal gap has support of the form $\cup_i \cap_{j\ge i} \gamma_j$ for some
end-nested sequence, and is therefore either a single point or a closed subset of $Z^\pm$ homeomorphic 
to an interval. If an interval, this interval is necessarily closed, since any
ray it contains would determine an end that lies in the intersection of all the $[e_i^L,e_i^R]$
contrary to hypothesis. For a zipper that is {\em hairy} (see Definition~\ref{definition:hairy})
the support of an equivalence class of ideal gap necessarily consists of a single point
(see Lemma~\ref{lemma:hairy_ideal_gap_point}).

\begin{lemma}[Order completion]\label{lemma:order_completion}
The union of $\EE^+$ together with the equivalence classes of ideal gaps of $Z^+$ 
admit a natural circular ordering which is equal to the order completion $\bar{\EE}^+$,
and similarly for $Z^-$.
\end{lemma}
This is immediate from the definition of ideal gaps, and from 
Lemma~\ref{lemma:nonempty_intersection}.

\subsection{Proof of the main theorem}

The zippers that arise from CaTherine wheels are not arbitrary, but satisfy two
important properties. First, they satisfy the strong landing property:
\begin{definition}[Strong landing property]\label{definition:strong_landing_property}
Let $Z^\pm$ be a zipper. Then $Z^\pm$ have the {\em strong landing property} if 
every ray in $Z^+$, i.e.\/ every embedded half-open interval $r:[0,1) \to Z^+$ extends
to an embedded closed interval $\bar{r}:[0,1] \to S^2$ in one of three ways:
\begin{enumerate}
\item{(type 0): $\bar{r}(1) \in Z^+$;}
\item{(type 1): $\bar{r}(1) \in S^2 - (Z^+\cup Z^-)$ and there is another embedded
half-open interval $s:[0,1) \to Z^-$ extending to $\bar{s}:[0,1]\to S^2$ with 
$\bar{r}(1)=\bar{s}(1)$; or}
\item{(type 2): $\bar{r}(1) \in Z^-$;}
\end{enumerate}
and conversely every point in $S^2$ is obtained as $\bar{r}(1)$ for some ray
$r$; and similarly for $Z^-$.
\end{definition}

Second, they are hairy:
\begin{definition}[Hairy]\label{definition:hairy}
A zipper $Z^\pm$ is {\em hairy} if for every closed oriented interval $I$ embedded in $Z^+$
there are nontrivial rays $r,l$ in $Z^+$ that begin on $I$ and intersect
$I$ only at these initial points, and which lie respectively to the right and to the left
of the oriented interval $I$ in $S^2$, and similarly for $Z^-$.
\end{definition}
We refer to $r$ and $l$ as above as `right and left hairs' on $I$.
Let us now state and prove some important consequences of hairiness.

\begin{lemma}[Ideal gaps of hairy zippers]\label{lemma:hairy_ideal_gap_point}
Suppose $Z^\pm$ is a hairy zipper. Then the support of each equivalence class of ideal gap 
is a single point.
\end{lemma}
\begin{proof}
The support of an equivalence class is the intersection $J:=\cup_i \cap_{j\ge i} \gamma_j$
of end nested proper lines associated to a nested sequence of closed
intervals $[e_i^L,e_i^R]$ in $\EE^\pm$ with empty intersection.
If $J$ contains more than one point, it cannot be hairy, since otherwise there would
be a ray $r$ starting at an interior point of $J$ on a side that forces the end $e$
associated to $r$ to lie in each of the intervals $[e_i^L,e_i^R]$, contrary to the
hypothesis that they have an empty intersection. Thus $J$ consists of a single point.
\end{proof}

\begin{lemma}[Order completions are circles]\label{lemma:order_circles}
Suppose $Z^\pm$ is a hairy zipper. Then the order completions $\bar{\EE}^\pm$ are 
homeomorphic to circles $S^1_\pm$.
\end{lemma}
\begin{proof}
We have already seen that $\EE^\pm$, and therefore also $\bar{\EE}^\pm$, are separable. 
Furthermore, we claim $\bar{\EE}^\pm$ have no {\em gaps}, i.e.\/ nontrivial 
closed intervals $[e^L,e^R]$ in $\bar{\EE}^\pm$ with no interior points. To see this,
suppose to the contrary that $[e^L,e^R]$ is an interval with no interior points. Then 
each of $e^L$ and $e^R$ is either an ordinary end or an ideal gap, and the convex hull
of these two elements or supports as appropriate is a nontrivial interval
$J \subset Z^\pm$ which is closed as a subspace of $Z^\pm$. But exactly the same
argument as the proof of Lemma~\ref{lemma:hairy_ideal_gap_point} shows that $J$
must fail to be hairy on at least one side, or we would be able to find a point
in the interior of $[e^L,e^R]$. This contradiction proves the claim.

A separable, complete circularly ordered set with no gaps is homeomorphic to a
circle; see e.g.\/ \cite{Frankel_Mobius}, Construction~5.7.
\end{proof}

\begin{lemma}[Laminar relations from hairy zippers]\label{lemma:zipper_relations}
Suppose $Z^\pm$ is a hairy zipper.
Define equivalence relations $\LL^\pm$ on $S^1_\pm$ whose trivial equivalence classes
are precisely the points of $\EE^\pm$ and whose nontrivial classes
are generated by pairs of ideal gaps of $Z^\pm$ with the same support. Then $\LL^\pm$ are
laminar relations with no isolated sides.
\end{lemma}
\begin{proof}
It is evident that these really are equivalence classes, since the set of ideal gaps
with support $p$ is disjoint from the set of ideal gaps with support $q$, by the
definition of support.

The first thing to see is that the equivalence classes of $\LL^\pm$ are unlinked.
This is because if $p,q\in Z^+$ are distinct points, and $\gamma_i$ is an end nested
sequence associated to $[e_i^L,e_i^R]$ with $p:=\cup_i \cap_{j\ge i} \gamma_j$ then
the $\gamma_j$ are eventually contained in the component of $Z^+-q$ containing $p$, and
conversely with the roles of $p$ and $q$ interchanged.
It follows that any pair of ideal gaps supported by $p$ or $q$ are approximated by
pairs of ideal points, so that the convex hull of the pair converging to $p$ is disjoint
from the convex hull of the pair converging to $q$. The claim follows.

The second thing to see is that the equivalence classes of $\LL^\pm$ are closed.
This is immediate for the trivial equivalence classes, so let $p \in Z^+$ be a point
which is the support of a set $\nu \in S^1_+$ of ideal gaps. We shall show, for every
$x\in S^1_+ - \nu$ that $x$ is contained in a nontrivial interval in $S^1_+$ disjoint
from $\nu$. If $x$ corresponds to a ray $r$ then we may truncate $r$ so that it lies
in a component of $Z^+-p$ and then by hairiness, find rays on either side of $r$ separating
it from every ideal gap of $p$. If $x$ corresponds to an ideal point with support
$q \ne p$ then an argument as in the previous paragraph, together with hairiness, gives
the desired separation. Thus $\LL^\pm$ are laminar relations.

What are the boundary leaves? They all arise in the following way: for each 
$p \in Z^+$ and each component $K$ of $Z^+-p$ the set of rays of $Z^+$ contained in $K$
correspond to ends in $S^1_+$ that are all contained in a minimal closed interval,
and the endpoints of that closed interval are (by construction) ideal points in the class
of $p$. Conversely, suppose $\lbrace x,y\rbrace \subset S^+_1$ is a boundary leaf of
the class with support $p$. Pick a point in $S^1_+$ which is represented by a ray $r$.
This ray can be truncated to a ray contained in some component of $Z^+-p$ and therefore
this boundary leaf is of the form above.

Notice by the way that if $p \in Z^+$ is 2-valent (i.e.\/ $Z^+-p$ has exactly 2 components)
then the equivalence class with support $p$ has exactly 2 points in $S^1_+$ (and likewise
for points of any finite valence $k$). Since $Z^+$ is a countable union of finite trees,
it has only countably many points of valence $>2$. Now, if $p$ is the support of an
ideal class $\nu$, and $\lbrace x,y\rbrace$ is a boundary leaf of $\nu$, it corresponds to
a component $K$ of $Z^+-p$. Take a sequence of 2-valent points $p_i$ in this component 
that converge to $p$ and let $\lbrace x_i,y_i\rbrace$ be the (boundary) leaves
associated to the $p_i$. Then $x_i \to x$ and $y_i \to y$, both nontrivially. This
proves that $\LL^\pm$ have no isolated sides, and we are done.

\begin{figure}[htpb]
\centering
\includegraphics[scale=0.5]{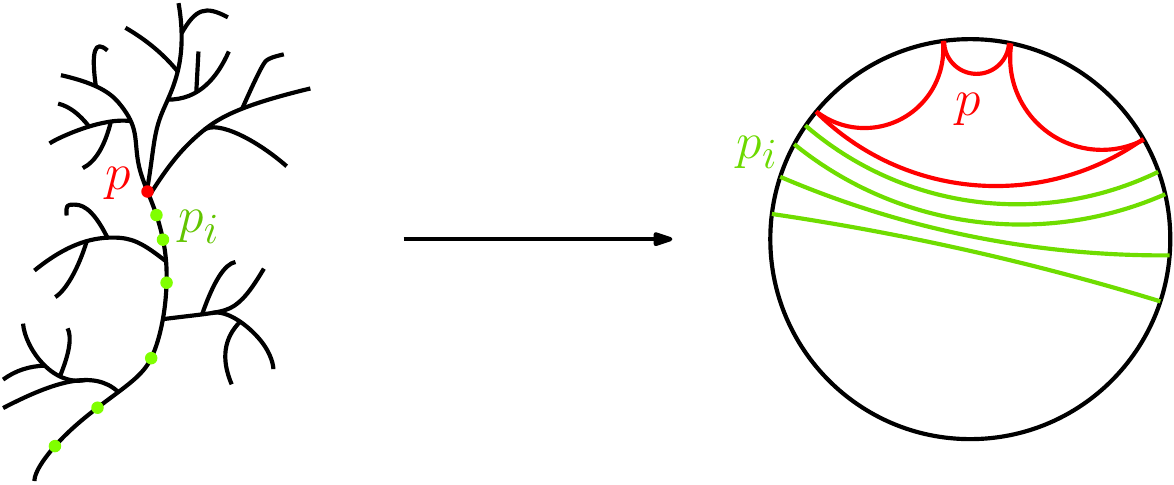}
\caption{A sequence of 2-valent points $p_i$ converging to $p$ gives a sequence of
leaves converging to a boundary leaf of $\nu$.}
\label{zippers_limit_leaves}
\end{figure}
\end{proof}

With these definitions, we are ready to state the main theorem:

\begin{theorem}[CaTherine wheels from zippers]\label{theorem:wheels_from_zippers}
Let $Z^\pm$ be a zipper. Then $Z^\pm$ arises from a CaTherine wheel $f:S^1 \to S^2$
if and only if $Z^\pm$ is hairy, and has the strong landing property.
\end{theorem}

\begin{proof}
A CaTherine wheel gives rise to a hairy zipper with the strong landing property, by
Theorem~\ref{theorem:zippers_from_wheels}. Thus we must prove the converse. 

Since $Z^\pm$ is hairy Lemma~\ref{lemma:order_circles} lets us obtain circles
$S^1_\pm$ homeomorphic to the order completions $\bar{\EE}^\pm$. 
Lemma~\ref{lemma:zipper_relations} says that $\LL^\pm$ have no isolated sides.
We must now show that there is a canonical identification of $S^1_+$ and
$S^1_-$ for which $\LL^\pm$ have no perfect fits. This is the next proposition:

\begin{proposition}[Canonical identification]\label{proposition:canonical_identification}
Suppose that $Z^\pm$ is a hairy zipper with the strong landing property. Then there is a
canonical (orientation reversing) identification of $S^1_+$ and $S^1_-$ for which
the associated laminar relations $\LL^+$ and $\LL^-$ have no perfect fits.
\end{proposition}
\begin{proof}
The correspondence between $S^1_+$ and $S^1_-$ is defined and constructed
in \cite{Calegari_Loukidou_Zippers}, \S~2.7, and in
fact only depends on the existence of a dense set of landing rays in either zipper. 
We summarize the construction. 

Let $r$ be a proper ray in $Z^+$. If it is type 1 it extends to a landing ray $\bar{r}$
that lands at $\bar{r}(1) \in S^2 - (Z^+ \cup Z^-)$, and there is another proper ray
$s$ in $Z^-$ that also lands at $\bar{r}(1)$. The equivalence classes of rays in $Z^+$
or $Z^-$ landing at any point $p \in (Z^+ \cup Z^-)$ are unique, or else we could find
a pair of proper rays $r,r'$ in $Z^+$ both starting at the same point, both landing at
$p$, and both disjoint except for their initial point. But then $p \cup r \cup r'$ would
be a Jordan curve disjoint from $Z^-$, contradicting the fact that $Z^-$ is connected
and dense in $S^2$. Thus, type 1 landing rays determine a correspondence between certain 
pairs of points in $\EE^\pm$. One may check that this correspondence is 
circular order-reversing (see e.g.\/ \cite{Calegari_Loukidou_Zippers}, Lemma~2.25) 
and evidently injective where defined.

If $r$ is type 2 then it lands at a point $p \in Z^-$. It evidently determines a cut
in the circular order of the components of $Z^- - p$ and therefore determines an
ideal gap supported at $p$. Again, one may check that this correspondence is circular
order-reversing and compatible with the type 1 correspondences. 
It is injective where defined, since if $r,r'$ are two type 2 rays that
land at $p\in Z^-$ then after modifying $r,r'$ in an initial segment so they start at
the same point and are disjoint except for their initial point, the Jordan curve
$p \cup r \cup r'$ intersects $Z^-$ only at $p$, and there is one side of it that is
disjoint from $Z^-$ or else there would be a component of $Z^- - p$ witnessing that
$r$ and $r'$ determine different cuts, and therefore different ideal points of $p$.

Conversely, if $p \in Z^-$ and $x \in S^1_-$ is an ideal gap with support equal to $p$,
we shall show there is a type 2 ray in $Z^+$ that lands at $p$ and represents $x$. 
This argument is rather involved, and uses the full force of the strong landing property.

Let $\gamma_i \subset Z^-$ 
be an end nested sequence with $\cup_i \cap_{j\ge i} \gamma_j$ associated to a nested
sequence of intervals $[e_i^L,e_i^R]$ converging to $x$. 
For each $i$ we may find a path $\delta_i$
in $Z^+$ whose endpoints are either type 1 landing points for rays 
representing $e_i^L,e_i^R$ or are ordinary points of $Z^+$ representing type 2
landing points for these rays. The unions $\beta_i$ obtained from
$\gamma_i \cup \delta_i$ by adding landing points (if any) form a sequence of Jordan
curves bounding nested closed disks $D_i \subset S^2$ whose intersection $K$ 
contains $p$ (because every $\gamma_i$ does).

We build a ray $r$ in $Z^+$ that lands at $p$
and determines the ideal gap $x$. The start of this argument parallels the proof of 
\cite{Calegari_Loukidou_Zippers}, Lemma~2.15 very closely.

Choose a point $q_0 \in \delta_0$ and for each $i$ let $J_i$ be the shortest embedded 
interval in $Z^+$ from $q_0$ to some unique nearest $q_i \in \delta_i$
(it is possible that $q_i = q_0$ and $J_i$ is degenerate). Then because the intervals
$[e_i^L,e_i^R]$ are nested, we have $J_0 \subset J_1 \subset \cdots$ and we may 
therefore form the union $J = \cup J_i$. If this is a proper ray then it lands
at some point $q$ in $K$. Otherwise it is not proper, and may be completed by adding
an endpoint $q$ in $Z^+ \cap K$. But $K$ is connected, and therefore either $K=p=q$ 
and we are done, or $K$ contains a sequence of points converging nontrivially to $p$.
By the strong landing property, these points are landing points of rays in $Z^+$,
and by construction the corresponding points in $S^1_-$ are all contained in every
interval $[e_i^L,e_i^R]$. But the only point contained in all of these intervals is
$x$, and we have already seen that different type 2 rays must land at different
ideal gaps. It follows that $K$ is a point after all, and $J$ is the desired proper
(type 2) ray, landing at $p$ and corresponding to $x$.

\begin{figure}[htpb]
\centering
\includegraphics[scale=0.5]{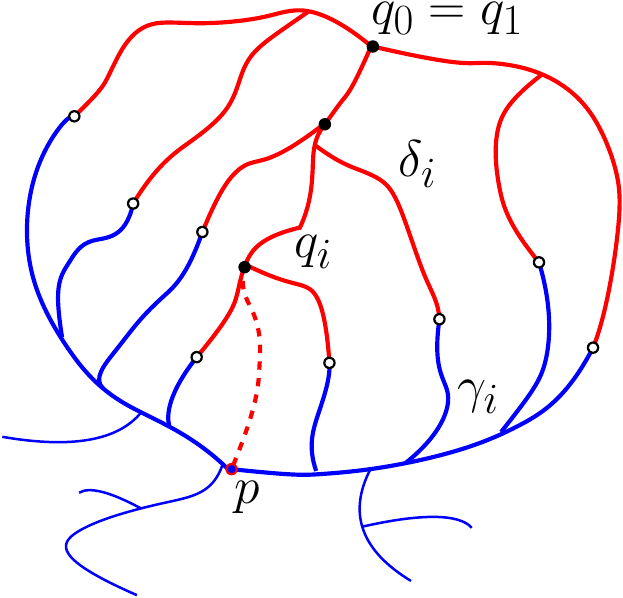}
\caption{Constructing a ray in $Z^+$ that lands at $p$.}
\label{ray_landing_at_point}
\end{figure}

Thus, by considering the correspondences of type 1 and 2 we obtain an order-reversing
bijection between $S^1_+$ and $S^1_-$. Since the topology agrees with the (circular) 
order topology, this extends uniquely to a homeomorphism.

It remains to show that in this correspondence the laminar relations $\LL^\pm$ have
no perfect fits. But this is obvious: we have just shown that the correspondence either
pairs trivial equivalence classes of one of $\LL^\pm$ with trivial equivalence classes 
of the other (type 1) or trivial equivalence classes with elements of non-trivial 
equivalence classes (type 2). Another way to see this is the observation 
that every ideal gap $x \in S^1_+$ corresponds to some landing ray $r \in Z^-$, and
by taking a sequence of 2-valent points on $r$ converging to the support of $x$ we
obtain a rainbow in $\LL^-$ for the point corresponding to $r$.
\end{proof}

But now the proof of Theorem~\ref{theorem:wheels_from_zippers} is immediate:
Proposition~\ref{proposition:canonical_identification} constructs a circle with
a pair of laminar relations with no perfect fits whose associated boundary laminations
have no isolated leaves, and then Theorem~\ref{theorem:wheels_from_laminations} 
constructs the desired CaTherine wheel. One can check that this construction is the 
inverse of the procedure in Theorem~\ref{theorem:zippers_from_wheels} that gives 
a zipper from a CaTherine wheel.
\end{proof}

\begin{remark}
For $G$-zippers there is more structure to work with, and in fact in the sequel we
will show (Theorem~\ref{theorem:G_zipper_to_G_wheel} and Theorem~\ref{theorem:G_zippers_are_hairy}), 
building on \cite{Calegari_Loukidou_Zippers} and a result of
KyeongRo Kim \cite{Kim}, that any (minimal) $G$-zipper $Z^\pm \subset S^2_\infty$ is hairy
and arises from a $G$-CaTherine wheel (for the given $G$ action on $S^2_\infty$) and
consequently a posteriori satisfies the strong landing property.
\end{remark}

Combining Theorem~\ref{theorem:laminar_decomposition}, 
Theorem~\ref{theorem:zippers_from_wheels}, Theorem~\ref{theorem:wheels_from_laminations}, 
and Theorem~\ref{theorem:wheels_from_zippers} gives the first main result of this paper:
 
\begin{theorem}[Equivalence Theorem]\label{theorem:equivalence} 
There is a canonical bijection between isomorphism classes of the following objects:
\begin{enumerate}
\item{CaTherine wheels $f:S^1 \to S^2$;}
\item{pairs of laminar relations $\LL^\pm$ on $S^1$ with no perfect fits and no isolated sides; and}
\item{hairy zippers $Z^\pm \subset S^2$ with the strong landing property.}
\end{enumerate}
\end{theorem}

\subsubsection{Half-zippers}

Suppose you have a zipper $Z^\pm$ and you lose half of it. When can you recover $Z^-$
from $Z^+$ alone?

The following definition is made in \cite{Calegari_Gwynne}:
\begin{definition}[Half-zipper]\label{definition:half_zipper}
A {\em half-zipper} is a subset $Z\subset S^2$ such that
\begin{enumerate}
\item{$Z$ is dense in $S^2$;}
\item{there is a unique embedded path in $Z$ connecting any two distinct points; and}
\item{every point of $Z$ is a (path) cut point.}
\end{enumerate}
A half-zipper is said to have {\em short hair} if, furthermore, for every
$\epsilon> 0$ there is a compact subset $T\subset Z$ homeomorphic to a finite
(simplicial) tree and so that every (path) component of $Z-T$ has diameter $<\epsilon$
in the spherical metric.
\end{definition}

\cite{Calegari_Gwynne} Theorem~1.3 says that every half-zipper with short hair is
actually $Z^+$ for some unique zipper $Z^\pm$ associated to a unique CaTherine wheel
$f:S^1 \to S^2$, and conversely \cite{Calegari_Gwynne} Proposition~1.4 says that if
$f:S^1 \to S^2$ is any CaTherine wheel with zipper $Z^\pm$, then each of
$Z^\pm$ has short hair. The proof builds on the material in this section. 

The short hair property is rather strong, and it is hard to come up with natural
examples of half-zippers (or even zippers) that have this property without already knowing
that they come from CaTherine wheels. However, some such examples do
occur in the theory of probability, as Liouville quantum gravity geodesic trees. See 
Example~\ref{example:LQG}.

\section{Embeddings}\label{section:embeddings}

We have already seen that a CaTherine wheel $f:S^1 \to S^2$ is approximable by 
embeddings. A stronger statement is that it admits a pseudo-isotopy, i.e.\/ that
there is a path of maps $f_t:S^1 \to S^2$ for $t \in [0,1]$ such that 
$f_t$ is an embedding for $t<1$, and $f=f_1$.
This is a standard result in decomposition theory, given 
Theorem~\ref{theorem:wheels_from_laminations} and a strengthening of Bing's
Theorem~\ref{theorem:shrinkable} due to Daverman \cite{Daverman} says 
that any shrinkable decomposition of a manifold can be realized by pseudo-isotopies.

Kerbs \cite{Kerbs} defines an invariant of a pseudo-isotopy of a
map $f:S^1 \to S^2$, and it turns out (Proposition~\ref{proposition:Kerbs_complete}) 
that it is a complete invariant of the homotopy
class of the pseudo-isotopy. 

Using Kerbs theorem, the main result of this section 
(Theorem~\ref{theorem:CaTherine_unique_pseudoisotopy}) is the statement 
that a CaTherine wheel admits a {\em unique} homotopy class of 
pseudo-isotopy.

\subsection{Definitions}

We give $\Map(S^1,S^2)$ the compact-open topology, and then inside this space we identify
$\EEE$, the space of embeddings. We may (and do) metrize $\Map(S^1,S^2)$ by fixing a
round metric on $S^2$, and defining the distance between $f$ and $g$ to be the supremum
of the distance from $f(p)$ to $g(p)$ over $p\in S^1$. The metric topology (often called
the uniform topology) and the compact open topology agree.

Denote by $\dashEEE$ the space obtained from $\EEE$ by
adding points in the closure corresponding to maps $S^1 \to S^2$ that are not injective
but are nowhere locally constant. The space $\EEE$ has
the homotopy type of $\SO(3)$; see 
\cite{Yagasaki}, Theorem~1.3. But the topology of $\dashEEE$ seems extremely complicated.
Nevertheless, both spaces admit an action of $\Homeo^+(S^2)$
by post composition. The action on $\EEE$ is transitive (with contractible point stabilizers),
from which the homotopy type of $\EEE$ can be deduced, but the action on $\dashEEE$ is very far
from transitive, and has uncountably many orbits:

\begin{example}[Uncountably many orbits]\label{example:uncountable}
Let $I \subset S^1$ be an embedded interval. Choose an identification of $I$ with $[0,1]$,
and for each $n\in \N$ let $p_n \in I \subset S^1$ be the point corresponding to $1-1/n\in [0,1]$.
Finally, let $P\subset S^1$ be the union of the $p_i$.

For each infinite sequence $S \in \lbrace 2,3\rbrace^\N$ we denote the coordinates of $S$
by $S_i$, for $i\in \N$. Let 
$$P_i:=\Bigl\lbrace p_j \in P \text{ such that } \sum_{k<i} S_k < j \le \sum_{k\le i} S_k\Bigr\rbrace$$
Observe that $P=\sqcup_i P_i$ and the cardinality of $P_i$ is $S_i$. 
Let $f_S:S^1 \to S^2$ be any continuous map which is injective away from $P$, and whose
nontrivial equivalence classes are exactly the finite subsets $P_i$. Then $f_S$ and
$f_{S'}$ are in different $\Homeo^+(S^2)$ orbits for different $S,S'\in \lbrace 2,3\rbrace^\N$.
\end{example}

\subsection{Some theorems of Kerbs}\label{subsection:Kerbs}

\begin{definition}[Pseudo-isotopy]\label{definition:pseudoisotopy}
A family of maps $F_t:X \to Y$ for $t\in [0,1]$ is a {\em pseudo-isotopy} if
$F_t$ is an embedding for $t<1$. 

A map $F:X \to Y$ admits a pseudo-isotopy if there is a pseudo-isotopy $F_t$ with
$F_1 = F$.
\end{definition}

The following may be found in Daverman \cite{Daverman}, Theorem~13.4.
\begin{theorem}[Shrinkable pseudo-isotopy \cite{Daverman}]\label{theorem:shrinkable_pseudoisotopy}
Let $M$ be a manifold, and let $F:M \to M$ be induced by a shrinkable decomposition. 
Then $F$ admits a pseudo-isotopy $F_t:M \to M$.
\end{theorem}
We would like to apply this theorem to $M=S^2$ and a map $F:S^2 \to S^2$ 
induced by a shrinkable decomposition. If we further include $S^1 \subset S^2$ as the equator, 
the map $f:=F|S^1$ admits a pseudo-isotopy $f_t:=F_t|S^1$.

Let us denote by $\PEEE(f)$ the space of pseudo-isotopies $f_t:S^1 \to S^2$ with $f_1=f$.
We topologize this as a subset of $\Map(I,\Map(S^1,S^2))$ or equivalently we may
define the distance between pseudo-isotopies $f_t$ and $g_t$ to be the maximum over $t_0\in [0,1]$
of the distance (as above) between the {\em maps} $f_{t_0}$ and $g_{t_0}$.

We do not know if every $f\in \partial \dashEEE$ admits a pseudo-isotopy, but for those
that do one may construct an invariant of a pseudo-isotopy following Kerbs \cite{Kerbs}.
Let us first describe the construction; for details, see Kerbs \cite{Kerbs}.
An informal version of this construction is described in \cite{Calegari_foliations},
\S~10.4.1.

\begin{construction}[Kerbs invariant]\label{construction:Kerbs_invariant}
We fix for all time a round metric on $S^2$. Next,
fix $f:S^1 \to S^2$ in the frontier of $\dashEEE$ and a pseudo-isotopy $f_t$ in $\PEEE(f)$. 

For each $t\in [0,1)$ Kerbs builds a pair of bounded metrics $d^\pm_t$ on $S^1$ where 
$d^+_t(p,q)$ resp. $d^-_t(p,q)$ is the minimal diameter of a closed connected set 
$K_t(p,q)$ in $S^2$, intersecting $f_t(S^1)$ only at $f_t(p)$ and $f_t(q)$, and
lying on the positive resp. negative side of $f_t(S^1)$.

Kerbs shows that these metrics have well-defined limits as 
$t \to 1$ and one obtains laminar relations
$\LL^\pm$ on $S^1$ by declaring $\lbrace p,q\rbrace \in \LL^\pm$ if $d^\pm_t(p,q) \to 0$.
We refer to these relations as the {\em Kerbs invariant} of the pseudo-isotopy $f_t$,
and denote them $\LL^\pm(f_t)$.
\end{construction}

\begin{theorem}[Kerbs invariant \cite{Kerbs}]\label{theorem:Kerbs_invariant}
Suppose $f:S^1 \to S^2$ is surjective and nowhere locally constant, and
admits a pseudo-isotopy $f_t\in \PEEE(f)$. Then the Kerbs invariants $\LL^\pm(f_t)$ are
an especial pair, and if $f_t,f_t'\in \PEEE(f)$ satisfy $f_{t_i} = f'_{t_i'}$ for
infinite sequences $t_i \to 1$ and $t_i' \to 1$, then $\LL^\pm(f_t) = \LL^\pm(f'_t)$.
\end{theorem}

The following corollary is a straightforward application of 
Theorem~\ref{theorem:Kerbs_invariant} and Construction~\ref{construction:Kerbs_invariant}:

\begin{corollary}[Homotopy class \cite{Kerbs}]\label{corollary:homotopy_class}
If $f_t,f_t'$ determine the same class in $\pi_0(\PEEE(f))$ then their Kerbs
invariants are equal. Furthermore, if $\LL^\pm$ induce a shrinkable decomposition
of $S^2$ realized by a pseudo-isotopy $f_t \in \PEEE(f)$ 
(via Theorem~\ref{theorem:shrinkable_pseudoisotopy}) then $\LL^\pm(f_t)=\LL^\pm$.
\end{corollary}
In other words, the Kerbs invariant is an invariant of $\pi_0(\PEEE(f))$. 
It is easy to prove a converse to this, i.e.\/ that the Kerbs invariant is a
{\em complete} invariant of $\pi_0(\PEEE(f))$:

\begin{proposition}[Calegari--Kerbs; Kerbs is complete]\label{proposition:Kerbs_complete}
Suppose $f:S^1 \to S^2$ is surjective and nowhere locally constant, and
admit pseudo-isotopies $f_t,g_t \in \PEEE(f)$ with the same Kerbs invariant. Then
$f_t$ and $g_t$ determine the same class in $\pi_0(\PEEE(f))$.
\end{proposition}
\begin{proof}
The idea of the proof is to interpolate between $f_s$ and $g_s$ for each $s\in [0,1]$
for which $f_s$ and $g_s$ are $\delta$-close by a path in $\EEE$ of diameter
$O(\delta)$. We may then fill in small diameter loops in $\EEE$ by small diameter
disks, using the local contractibility of $\EEE$. We remark that this is a local
statement, in other words one must take $\delta$ small enough, since $\EEE$ is not (globally)
simply-connected. 

Fix a small $0 < \delta \ll 1$. Then we may find $s \in [0,1]$ so that each pair
of maps in $f_s, g_s, f$ have distance at most $\delta$. 

Next, we approximate the Kerbs invariant $\LL^\pm$ by finite laminations
$L^\pm$ contained in $\LL^\pm$ with the same set of endpoints $X$ in $S^1$ which 
are a $\delta$-net in either metric $d^\pm_s$ associated either to $f_s$ or $g_s$. 
Embed $S^1$ in $S^2$ as the equator,
and realize leaves of $L^\pm$ by chords in the northern and southern hemisphere, so that
$\Gamma:=S^1 \cup L^\pm$ is a finite connected graph in $S^2$ with disk complementary
components. Add more (but finitely
many) leaves to $L^\pm$ if necessary, still contained in the laminar relations, 
so that complementary polygons to $\Gamma$ have bounded number of sides (6 can be achieved).
 
By the definition of the Kerbs invariant, for each leaf $\lambda$ of
$L^+$ resp. $L^-$ with endpoints $p,q\in X$
we can find a properly embedded arc in $S^2 - f_s(S^1)$ on the 
positive resp. negative side of
$f_s(S^1)$ joining $f_s(p)$ to $f_s(q)$ and with diameter at most $\delta$; by
abuse of notation, call this arc $f_s(\lambda)$, and likewise define $g_s(\lambda)$.
If two such arcs $f_s(\lambda)$ and $f_s(\lambda')$ intersect, we may eliminate innermost
bigons inductively to obtain disjoint arcs with the same diameter bounds.
We may therefore extend $f_s$ and $g_s$ to maps $f_s,g_s: \Gamma \to S^2$ for which the
diameter of each complementary polygonal region is bounded by $3\delta$. Now
extend these $f_s$ and $g_s$ further over complementary polygons of $\Gamma$ to obtain
homeomorphisms $f_s,g_s:S^2 \to S^2$ so that $g_s f_s^{-1}:S^2 \to S^2$ moves points
at most $7\delta$.

But by the Whitney trick, any homeomorphism of $S^2$ that moves points at most $O(\delta)$
can be isotoped to the identity by an isotopy that moves points at most $O(\delta)$
(the exact constants are not important). It follows that we may interpolate between
$f_s$ and $g_s$ by a 1-parameter family of embeddings $F_{s,t}$ with diameter $O(\delta)$
in $\EEE$.

Choose a sequence $s_i \to 1$ and 1-parameter families $F_{s_i,t}$ interpolating
between $f_{s_i}$ and $g_{s_i}$ with diameters going uniformly to $0$. Now
use the fact that the group of homeomorphisms of $S^2$ is locally contractible
to fill in the rest of the homotopy through pseudo-isotopies $F_{s,t}$ from $f_s$ to $g_s$.
\end{proof}

In fact, similar arguments, and the local contractibility of $\Homeo^+(S^2)$ can 
probably be used to prove the following conjecture (though we have not pursued this):

\begin{conjecture}[Local weak contractibility]\label{conjecture:local_contractibility}
Let $f\in \partial \dashEEE$ be any map. Then every path component of $\PEEE(f)$ is
weakly contractible. That is, any $S^k \to \PEEE(f)$ for $k>0$
extends to $D^{k+1} \to \PEEE(f)$.
\end{conjecture}

\subsection{Uniqueness of pseudo-isotopies}

For some $f \in \partial \dashEEE$ the space of pseudo-isotopies $\PEEE(f)$ is not path-connected.
In other words, $\EEE$ is not locally path-connected near $f$. One says that such $f$
are {\em self-bumping} points (since $\EEE$ `bumps into' itself nontrivially at $f$).

\begin{example}[Zigzag]\label{example:zigzag}
We give a simple example of two embeddings of $S^1$ in $S^2$ that are very close
as maps, but can't be joined by a path of embeddings of small diameter. The
example is local, and may be inserted near any $f$ in $\partial \dashEEE$. 
In particular, the set of self-bumping points is dense in $\partial \dashEEE$. 
See Figure~\ref{zigzag}.

\begin{figure}[htpb]
\centering
\includegraphics[scale=0.5]{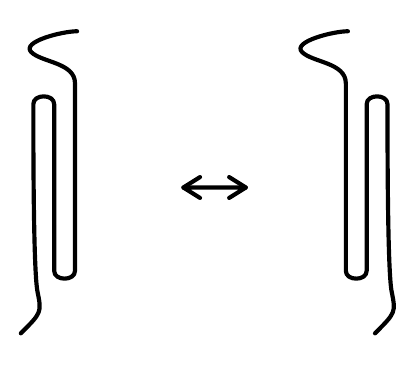}
\caption{Two embedded intervals that are close as maps but not close through embedded
maps.}
\label{zigzag}
\end{figure}
\end{example}

\begin{theorem}[Uniqueness of pseudo-isotopies]\label{theorem:CaTherine_unique_pseudoisotopy}
Let $f:S^1 \to S^2$ be a CaTherine wheel. Then $f$ admits a unique path equivalence
class of pseudo-isotopy. In particular, $f\in \partial \dashEEE$ is not a self-bumping point
of $\EEE$: i.e.\/ there is a unique local path component of $\PEEE(f)$.
\end{theorem}
\begin{proof}
Let $f:S^1 \to S^2$ be a CaTherine wheel with associated laminar relations
$\LL^\pm$. Any CaTherine wheel admits a pseudo-isotopy with associated Kerbs 
invariant $\LL^\pm$. Suppose there is another pseudo-isotopy with Kerbs invariant
$\KK^\pm$. By Proposition~\ref{proposition:Kerbs_complete} 
it suffices to show that $\KK^\pm$ agrees with
$\LL^\pm$. In fact we shall show that any especial pair $\KK^\pm$ on $S^1$ inducing the same
decomposition as $f$ must agree with $\LL^\pm$ (up to sign).

First, each equivalence class of $\KK^\pm$ is an entire equivalence class of one of 
$\LL^\pm$. For, if $\lbrace p,q\rbrace$ is in $\KK^+$ and also (say) in $\LL^+$, then
there is a negative rainbow in $\LL^-$ for both $p$ and $q$; the elements in this rainbow
cannot be in $\KK^+$ so they must be in $\KK^-$. But then $\KK^-$ cannot have a nontrivial 
equivalence class containing $p$ or $q$. This establishes the claim (and incidentally
shows that $\KK^\pm$ have no perfect fits).

Second, by abuse of notation we will identify $\KK^\pm$ and $\LL^\pm$ with the
corresponding hull decompositions of $P^\pm$. We show
that the set of decomposition elements of $\KK^+$ in $P^+$ 
that are also decomposition elements of $\LL^+$ is open and closed. Closed follows
because both relations are closed. Open follows from no isolated sides (both for $\KK^+$ and
for $\LL^+$). This proves the claim.

The conclusion follows.
\end{proof}

\section{P-CaTherine wheels}\label{section:P_wheels}

The definition of CaTherine wheel is simple enough to be useful, and flexible enough
to cover a great many important examples that arise in practice, as we shall see 
(pseudo-Anosov flows without perfect fits; some expanding Thurston maps; 
whole plane $\SLE_\kappa$ for $\kappa \ge 8$). But there are important exceptions:
Cannon--Thurston maps associated to hyperbolic punctured surface bundles; general
pseudo-Anosov flows or expanding Thurson maps; Brownian webs and space-filling
$\SLE_\kappa$ for $\kappa \in (4,8)$ and so on; see e.g.\/ \cite{Gwynne_Miller_Sheffield}
Figure~5. 

It is challenging to give a useful definition of a surjective map $f:S^1 \to S^2$ that covers
all and only the examples we want. Let us give a name to this class of maps, in the
hope that this will bring us closer to a definition: P-CaTherine wheels, where the P
stands of Peano, parabolic, puncture, pseudo-Anosov, pseudo-isotopy, provisional, 
or whatever the reader likes.

One might try to define P-CaTherine wheels in terms of a pair of laminar relations
$\LL^\pm$. We have already seen that a pair of laminar relations $\LL^\pm$ of $S^1$ arises from
a CaTherine wheel if and only if it is an especial pair (in the sense of 
Frankel--Landry) without perfect fits. An arbitrary especial pair gives rise to two
decompositions of $S^1$, but the nontrivial elements in these two decompositions
can overlap nontrivially, and the maximal common quotient is typically not homeomorphic to
a sphere, for instance,
if the especial pair contains a finite cycle of leaves alternating between $\LL^+$ and
$\LL^-$ for which each leaf shares exactly one endpoint with the next one in the cycle.
More generally, an especial pair might not contain any finite alternating cycles, but 
it might contain finite cycles of infinite alternating chains whose closures share an
endpoint in common with the closure of their successor, and so on. 

At a minimum, a P-CaTherine wheel should have the property that for every closed
interval $J\subset S^1$ the image $f(\partial J)$ is contained in the frontier
$\partial f(J)$. This property alone implies the analog of bullet (3) from 
Lemma~\ref{lemma:basic_properties}, and with essentially the same proof:

\begin{lemma}[Interiors disjoint]\label{lemma:interiors_disjoint}
Let $f:S^1 \to S^2$ be any map so that for all closed intervals $J\subset S^1$,
the image $f(\partial J)$ is contained in the frontier $\partial f(J)$. The for
disjoint closed intervals $I,J \subset S^1$ the interiors of $f(I)$ and $f(J)$ are disjoint.
\end{lemma}
\begin{proof}
If $p\in J$ is arbitrary with $f(p)$ in the interior of $f(J)$ then $p$ is in the interior
of $J$ by hypothesis and therefore in the complement of $I$. Let $K$ be the closed
interval $[p,I^+]$. Then $I\subset K$ and therefore the interior of $f(I)$ is contained in
the interior of $f(K)$. But $f(p)$ is contained in the frontier of $f(K)$ by hypothesis and
therefore the interior of $f(I)$ does not contain $f(p)$. This completes the proof.
\end{proof}

Before giving a precise definition,
Let's discuss a couple of examples to get more clarity on the properties that a
P-CaTherine wheel should have.

\begin{example}[Punctured surface bundle]\label{example:punctured_surface}
Let $f:S^1 \to S^2$ be the Cannon--Thurston map associated to a hyperbolic
punctured surface bundle $M$. A parabolic element $t\in \pi_1(M)$ generating the
monodromy has countably infinitely many fixed points $p_i$ on $S^1$ with $i \in \Z$
alternating between stable/unstable fixed points, and accumulating as $i \to \pm \infty$
on an `indifferent' fixed point $p$. The stable fixed points and $p$ are all in a
single class $\nu^+ \in \LL^+$, and the unstable fixed points and $p$ are all in a
single class $\nu^- \in \LL^-$. In particular, $\nu^+ \cap \nu^-$ is nonempty, and
there is a perfect fit.

If $J \subset S^1$ is the interval between two adjacent stable (or unstable) fixed points
then $f(J^+)=f(J^-)$ is a local (but not global) cut point of $f(J)$, which is not 
simply-connected. See Figure~\ref{parabolic_circle}.

\begin{figure}[htpb]
\centering
\includegraphics[scale=0.6]{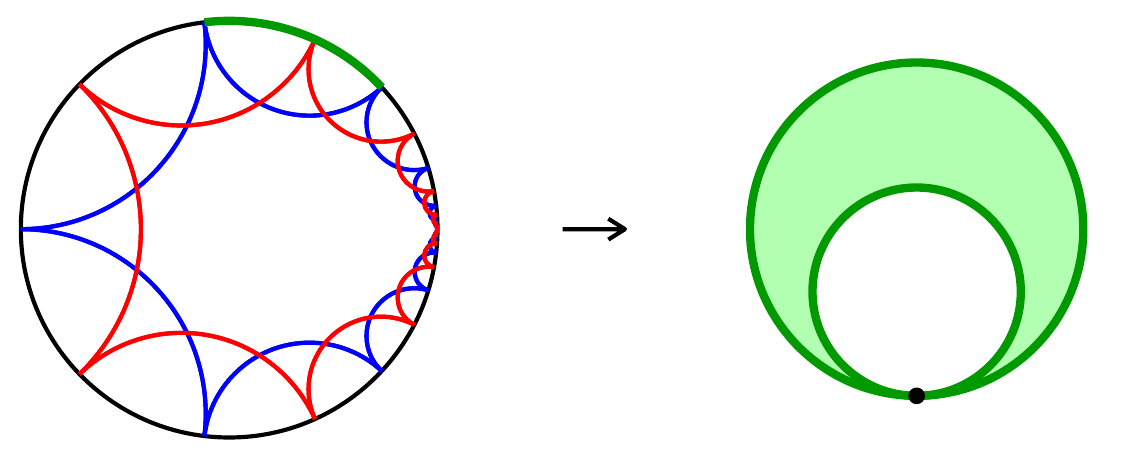}
\caption{The parabolic fixed point in $S^2_\infty$ is a local cut point of $f(J)$}
\label{parabolic_circle}
\end{figure}
\end{example}

\begin{example}[Quasigeodesic flow]\label{example:pA_perfect_fits}
Let $f:S^1_\un \to S^2_\infty$ be the map defined by Frankel \cite{Frankel_Peano} 
associated to a quasigeodesic flow $X$ on a closed
hyperbolic 3-manifold; the domain $S^1_\un$ of $f$ here is the {\em universal circle} 
associated to the orbit space $P$ of $\tilde{X}$ (which is topologically a plane)
defined in \cite{Calegari_quasigeodesic}.
Associated to $\tilde{X}$ are two decompositions of $P$, associated to
the positive and negative endpoint maps $e^\pm:P \to S^2_\infty$, and these in turn
give rise to laminar relations $\LL^\pm$ on $S^1_\un$. It can easily happen that 
these have perfect fits; for example, $X$ could already be a quasigeodesic pseudo-Anosov 
flow with perfect fits.

A typical configuration is a {\em lozenge}, a pair of leaves $\mu^+,\nu^+$ in $\LL^+$
and another pair $\mu^-,\nu^-$ in $\LL^-$ sharing endpoints as in 
Figure~\ref{lozenge}. Let $J \subset S^1$ be the interval indicated. Then 
$f(J)$ has two local cut points. 

\begin{figure}[htpb]
\centering
\includegraphics[scale=0.6]{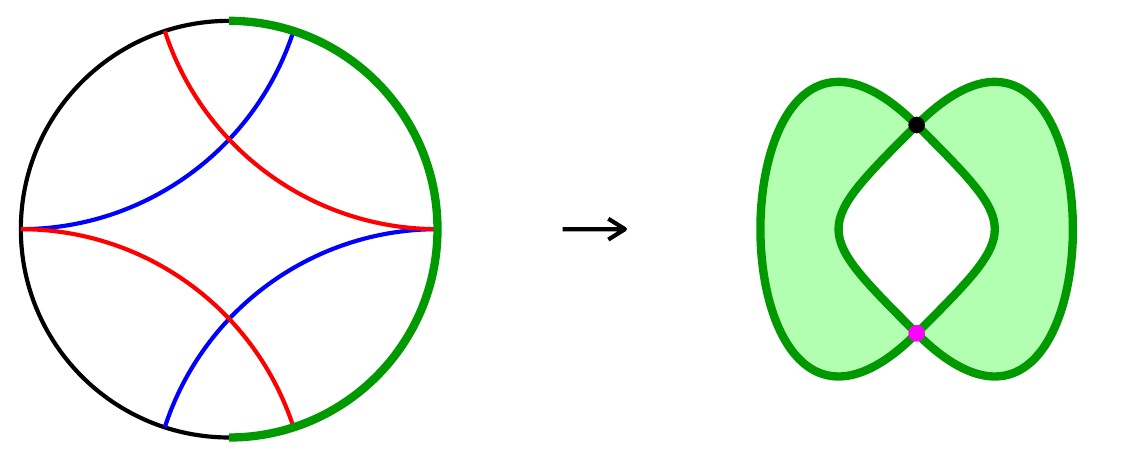}
\caption{The image $f(J)$ has two local cut points. One (in black) is genuine; the
other (in pink) is fake.}
\label{lozenge}
\end{figure}
\end{example}

The local cut points in Example~\ref{example:pA_perfect_fits} 
are different in kind from each other. Roughly speaking, one of the local cut points
locally separates the track of $f$ (for now, call such a cut point {\em genuine}) whereas the
other one does not (call such a cut point {\em fake}). For example,
the local cut point in Example~\ref{example:punctured_surface} is fake. 

Taking complements reverses this: the interior of $f(J^c)$ is the complement of 
$f(J)$ and vice versa, and the fake local cut point in $\partial f(J)$ is a 
genuine local cut point in $\partial f(J^c)$ and vice versa. This is essential if we
want $f$ to admit a pseudo-isotopy: if $I$ and $J$ are closed disjoint intervals
with a common point $x \in f(I)\cap f(J)$ which is a genuine local cut point in both, and
if the local tracks of $f|I$ and $f|J$ cross essentially at $x$, then $f$ cannot
be perturbed to an embedding. With this in mind we give the following provisional
definition:

\begin{provisional_definition}[P-CaTherine wheel]\label{definition:P_CaTherine_wheel}
A {\em P-CaTherine wheel} is a map $f:S^1 \to S^2$ which is surjective, and such 
that for every closed interval $J$ in $S^1$,
\begin{enumerate}
\item{the interior of $f(J)$ is dense;}
\item{every component of the interior of $f(J)$ is an open disk;}
\item{$f(\partial J)$ is contained in the frontier $\partial f(J)$; and}
\item{{\bf (no essential crossing):} if $I$ and $J$ are disjoint and their
images $f(I)$ and $f(J)$ contain a common point $x$ which is a local cut point in both,
then $f|I$ and $f|J$ do not cross essentially at $x$.}
\end{enumerate}
\end{provisional_definition}

Given $f(J)$ we would like to {\em resolve} the fake local cut points. In our (toy)
examples, this means to remove each fake local cut point and add a new point on either
side so that the resolution is still compact and path-connected. If we perform this
resolution, then we may embed the resolutions of $f(J)$ and $f(J^c)$ disjointly from
each other as path-connected and contractible subsets of $S^2$ separated by an
abstract `circle' $S^1(J)=S^1(J^c)$. 

In general we must deal with  the fact that a local cut point may have a high 
(even infinite) valence.

\begin{definition}[Local resolution]\label{definition:local_resolution}
The {\em resolution} $f_R(J)$ of $f(J)$ is a compact, connected, locally
connected subset of $S^2$ together with a monotone decomposition $S^2 \to S^2$
taking $f_R(J)$ to $f(J)$ so that $f_R(J) \to f(J)$ is $1$--$1$
on the complement of the local cut points of $f(J)$, so that $f|J$ factorizes as
$J \to f_R(J) \to f(J)$, and so that $f_R(J)$ is maximal with respect to these properties.
\end{definition}

Here is how to construct $f_R(J)$. We explain the construction first in the case of
a single local cut point. If $x \in f(J)$ is a local cut point we give $f(J)-x$
the path topology, and define $\EE_x$ to be the set of ends of $f(J)-x$ in this 
topology. Give $\EE_x$ the obvious (circular) order topology, and let $\overline{\EE}_x$
be the order completion, and embed $\overline{\EE}_x$ in $S^1$ in an order-reversing
way. Now choose a homeomorphism $S^2 - x$ into $S^2 - D_x$ for $D_x$ a closed unit disk
in such a way that the homeomorphism extends to the given embedding of
$\overline{\EE}_x$ into $S^1_x = \partial D_x$. The embedding $\overline{\EE}_x \to S^1_x$
should be order reversing, because $\overline{\EE}_x$ lives on the `outside' of $D_x$
whereas $S^1_x$ inherits its orientation from the inside, i.e.\/ from $D_x$ itself.

For some pairs of ends $e,e'$ of $\EE_x$ there is a germ of an arc $I\subset J$ so that
$f|I$ nontrivially crosses from $e$ to $e'$; this collection of pairs determines a
relation on $\EE_x$, and we let $\LL_x$ denote the closed equivalence 
relation on $S^1_x$ it generates. 

\begin{lemma}[Cut point relation laminar]\label{lemma:cut_relation_laminar}
With notation as above, the relation $\LL_x$ is laminar.
\end{lemma}
\begin{proof}
For, if not, we could find disjoint embedded intervals $I,I'$ in $J$ for which the local
tracks $f|I$ and $f|I'$ cross essentially over $x$ between pairs of local components
that are linked in $\EE_x$, violating the no essential crossing axiom.
\end{proof}

Define a decomposition of $D_x$ whose elements are the convex hulls of classes in
$\LL_x$, and extend it trivially to $S^2 - D_x$. The quotient of $S^2$ 
by this decomposition is $S^2$, and $f_R(J)$ sits inside it. Further quotienting
all of $D_x$ to a point gives the desired map $f_R(J) \to f(J)$, and by construction
$f$ factors to $J \to f_R(J) \to f(J)$.

When there are more local cut points we perform this operation simultaneously on all
of them. This is a little subtle, since {\em a priori} there might be uncountably many local cut
points. However, there are only countably many local cut points with valence at least
$3$ (this is for essentially the same reason that one can only draw countably many
disjoint copies of the letter T in the plane); and of the valence $2$ local cut
points, only countably many fake ones. Since this discussion is largely informal, 
details are left to the reader. 

\begin{proposition}[Canonical cactus resolution]\label{proposition:cactus_resolution}
For every closed interval $J$ in $S^1$, the canonical resolution $f_R(J)$ is a {\em cactus}:
it is a nonempty, compact, locally connected, simply connected planar set with dense interior.
\end{proposition}
\begin{proof}
The resolution $f_R(J)$ is compact, locally connected and has dense interior
because it is the surjective image of $J$
under the factorization $J \to f_R(J) \to f(J)$. We show it is simply-connected.
Suppose not, so that there are at least two disjoint complementary regions $U$ and $V$
to $f(J)$ contained in disjoint complementary regions of $f_R(J)$.
Pick points $x\in U$ and $y\in V$ with preimages $p,q\in S^1$. Up to exchanging the
order, the interval $[p,q]$ is disjoint from $J$ and therefore the interior of $f([p,q])$ 
is disjoint from the interior of $f(J)$. It follows that there must be some genuine local
cut point of $f(J)$ that $f|[p,q]$ crosses essentially, violating the no essential crossing
axiom.
\end{proof}

For any interval $J$, the boundary of the (planar) cactus $f_R(J)$ 
has a canonical parameterization by an abstract circle $S^1(J)$ that may be identified
with the space of prime ends of the complement.

\begin{proposition}[Monotonicity]\label{proposition:monotone_cactus}
With notation as above, if $C(J) \subset J$ is the preimage of $\partial f(J)$, there is a canonical 
factorization $C(J) \to S^1(J) \to \partial f(J)$, and the 
map from $C(J)$ to $S^1(J)$ is monotone in the complement of the image of $\partial J$,
in the sense of Proposition~\ref{proposition:monotonicity}.
\end{proposition}
\begin{proof}
By continuity of $f$ it suffices to prove this in the generic case that $f(J^-)$ and $f(J^+)$ are
distinct, and are not local cut points of $f(J)$. Let $C^o(J) \subset J$ denote the preimage
of the set of points on $\partial f(J)$ that are not local cut points. Note that
$C^o(J)$ is dense in $C(J)$, and by definition, there is a factorization
$C^o(J) \to S^1(J) \to \partial f(J)$. As in the proof of
Proposition~\ref{proposition:monotonicity} we suppose that the proposition is false, and that there
are $J^- < x < y < J^+$ with $x,y\in C^o(J)$ and such that $f(J^-) < f(y) < f(x) < f(J^+)$ in
the orientation on $S^1(J)$. Consider the images $f([J^-,x])$ and $f([y,J^+])$ in the resolution
$f_R(J)$. Because the extreme points link in $S^1(J)$, the images must intersect. The intersection
cannot be in the interior of either, and therefore it must be a local cut point in both.
But this violates the no essential crossing axiom.
\end{proof}

Monotonicity is the key to the structure theory of (ordinary) CaTherine wheels
developed in \S~\ref{section:CaTherine_wheels}. Our
confidence in Provisional Definition~\ref{definition:P_CaTherine_wheel} 
rests on Proposition~\ref{proposition:monotone_cactus}. 

\vfill
\pagebreak

\part{Dynamics and Conformal Geometry}

\section{Automorphisms}

Some CaTherine wheels enjoy a large family of symmetries. These are in many ways
among the most interesting examples, and it is in this section that we start to
connect up our technology to low-dimensional geometry, especially to the theory of
hyperbolic 3-manifolds.

\subsection{Classification of automorphisms}

\begin{definition}[Automorphism]\label{definition:automorphism}
An {\em automorphism} of a CaTherine wheel $f:S^1 \to S^2$ is a pair of 
{\em orientation preserving} homeomorphisms $h:S^1 \to S^1$ and $H:S^2 \to S^2$ such
that $fh = Hf$.
\end{definition}

\begin{remark}
There is no a priori reason why we can't consider orientation-reversing homeomorphisms
but in this section we restrict to the orientation-preserving kind for simplicity. Note that in 
case one map is orientation-reversing on one of $S^1$ and $S^2$ and the other map is
orientation-preserving on the other, it will interchange $\LL^\pm$. Otherwise
it preserves each of $\LL^\pm$. 
\end{remark}

\begin{theorem}[Automorphisms]\label{theorem:automorphisms}
Let $h:S^1 \to S^1$, $H:S^2 \to S^2$ be an automorphism of a CaTherine wheel $f:S^1 \to S^2$.
Then \begin{enumerate}
\item{some finite power $h^n$ of $h$ is either the identity, or has a finite even 
positive number of fixed points on $S^1$ with alternating source-sink dynamics;}
\item{up to replacing $h^n$ by $h^{-n}$ if necessary, the attracting/repelling fixed 
points are joined by leaves of $\LL^+$ and $\LL^-$ respectively; }
\item{if there are at least 4 fixed points for $h^n$, the image in $S^2$ of 
the attracting/repelling fixed points of $h^n$ in $S^1$ are the unique fixed points 
of $H$ and are attracting/repelling for $H$; and}
\item{if there are exactly 2 fixed points for $h^n$, the image in $S^2$ of 
these fixed points are the unique fixed points of $H$ and either 
\begin{enumerate}
\item{they are distinct and are attracting/repelling for $H$; or}
\item{(exotic): they are equal in $S^2$, 
and the restriction of $H$ to the complement of this
unique fixed point is conjugate to a translation of the plane.}
\end{enumerate}}
\end{enumerate} 
\end{theorem}
\begin{proof}
Let $h,H$ be an automorphism of $f:S^1 \to S^2$. The first step is to show that no
point of $\fix(h)$ is isolated in $\fix(h)$ on exactly one side, and every isolated
point of $\fix(h)$ is either a source or a sink. Equivalently, after possibly replacing $h$ with
its inverse, we must show that if $p \in \fix(h)$ is attracting on one side 
(without loss of generality, the positive side) then it must also be attracting on the 
other side. Suppose not, so that there is $p\in \fix(h)$ attracting on the positive side
and there exist $p_i \to p$ on the negative side which are either fixed or repelled by $h$.
Without loss of generality $p$ has a positive rainbow. In particular, there are positive
boundary leaves $\ell:=\lbrace x,y\rbrace$ and $\ell':=\lbrace x',y'\rbrace$ with
$$p_{i-1} < x < p_i < x' < p < y' < y$$
and so that $y',y$ are in the attracting basin of $p$ on the right. But then there
is a power $h^n$ for which $h^n(x) < x' < p < h^n(y) < y'$ so that $h^n(\ell)$ crosses
$\ell'$, which is a contradiction.

It follows that every nontrivial $h$ with a fixed point has an isolated fixed point $p$,
without loss of generality a source. Again without loss of generality there is a positive
rainbow for $p$, and we choose a positive boundary leaf $\ell$ in this rainbow with
endpoints in the repelling basin of $p$. The
iterates $h^n(\ell)$ must limit somewhere. Either they limit to a single sink point $q\in\fix(h)$
or $h^n(\ell)$ limits to a leaf of $\LL^+$ whose endpoints $p_1,p_{-1}$ are both in $\fix(h)$. By
construction, these endpoints are both isolated in $\fix(h)$ on at least one side, and
therefore they are isolated sinks in $\fix(h)$. The points $p_{\pm 1}$ have
negative rainbows, so we may repeat this construction and obtain a countable
chain of consecutive isolated fixed points $p_i\in \fix(h)$ indexed by $i\in \Z/n\Z$
for some $n\ge 0$, that alternate between sources and sinks; and providing $n>2$,
all the sources resp. sinks are in a single equivalence class of $\LL^-$ resp. $\LL^+$.
But this implies $n>0$, for if $n=0$ the $p_i$ have a limit point in $\fix(h)$, and
this limit point is in a nontrivial equivalence class of both $\LL^-$ and $\LL^+$,
contradicting no perfect fits.

If no power of $h$ has a fixed point, it has irrational rotation number. 
By the Lefschetz fixed point theorem, the map $H: S^2 \to S^2$ must have some fixed point, 
and the preimage of this (call it $C$ in $S^1$) must be a nontrivial equivalence
class (without loss of generality in $\LL^-$) permuted by $h$. Note that $C$ 
can't be countable, or else by Sierpinski induction it would contain an invariant 
finite subset, which would show that the rotation number is rational after all. 
Thus after passing to an invariant subset of $C$ if necessary, we may assume $C$ is 
a Cantor set, and $f$ is transitive on the complementary intervals $J_i$
(i.e.\/ they are indexed by $\Z$ so that $h(J_i) = J_{i+1}$). 
Let $J_0$ be any such complementary interval. Without loss of generality the 
endpoint $J_0^+$ has a positive
rainbow, and therefore there is a boundary leaf $\mu$ of $\LL^+$ with endpoints in $J_0$ and
in some other $J_i$. But complementary intervals are dense and share no endpoints,
so there must be some complementary interval $J_j$ between $J_0$ and $J_i$, and 
then $h^j(\mu)$ must cross $\mu$ transversely, giving a contradiction. 

We are thus reduced to proving the third and fourth bullet. Without loss of generality, 
$h$ has a finite even number of fixed points that alternate between attracting and 
repelling; if there are at least $4$ points, the attracting points are all contained
in a single equivalence class of $\LL^+$ (say) and the repelling points are all 
contained in a single equivalence class of $\LL^-$; 
in particular, these have different images in $S^2$ which are the unique 
attracting/repelling points of $H$ and we are done. 

On the other hand, if there are exactly 2 points it is a priori possible that these two 
points are joined by a leaf $\mu$ of $\LL^+$ (say) and that $H$ has a unique fixed point $x$ 
in $S^2$. Notice that $\mu$ cannot be a boundary leaf of $\LL^+$ since otherwise it would be
accumulated on at least one side by leaves of $\LL^+$ in distinct equivalence classes, which would
intersect themselves under powers of $h$. In any case, for every compact subset
$K \subset S^2 - x$ the preimage of $K$ in $S^1$ is compact and disjoint from the
endpoints of $\mu$, and therefore under positive/negative powers of $H$ the image of
$K$ must converge to $x$. This shows that the action of $\langle H\rangle$ on
$S^2-x$ is a covering space action with Hausdorff quotient. Since the quotient space
is orienable and has fundamental group $\Z$ it is an annulus, 
so $H$ is conjugate to a translation.
\end{proof} 

We call this last kind of automorphism {\em exotic}. Evidently it cannot
arise as an element of the group $G$ of automorphisms of $f:S^1 \to S^2$ when $G$
is a cocompact Kleinian group. 

\begin{example}[Exotic automorphism]\label{example:exotic_automorphism}
We describe a construction in $\HH^2$. First, start with any CaTherine wheel.
Pick a leaf $\lambda$ in $\LL^+$ that is a limit leaf on both sides, and a pair
$\mu,\nu$ in $\LL^-$ that are limit leaves on both sides, and whose geodesic
representatives in $\HH^2$ intersect the geodesic representative of $\lambda$ in two
distinct points $p,q$. Now form the infinite dihedral cover of $\HH^2$ with order 2
branching over both these points. This is topologically a plane that may be thought
of as an infinite rectangular strip running between two parallel horizontal lines, 
and it may be compactified to a closed disk by adding these lines together with 
two points corresponding to the ends of the infinite dihedral group. 

The preimage of (the geodesics associated to) 
each of $\mu$ and $\nu$ alternate in an infinite horizontal sequence of 
letter $X$s, and the preimage of $\lambda$ is an infinite horizontal line crossed 
by infinitely many vertical intervals which bisect each of these $X$s. Together with
the components of the preimages of the other leaves of $\LL^\pm$, these give a
decomposition of the disk that limits to a new pair of laminar relations of $S^1$
invariant under the infinite dihedral group. By construction these relations have no
perfect fits and no isolated sides.
See Figure~\ref{exotic_laminations}.

\begin{figure}[htpb]
\centering
\includegraphics[scale=0.6]{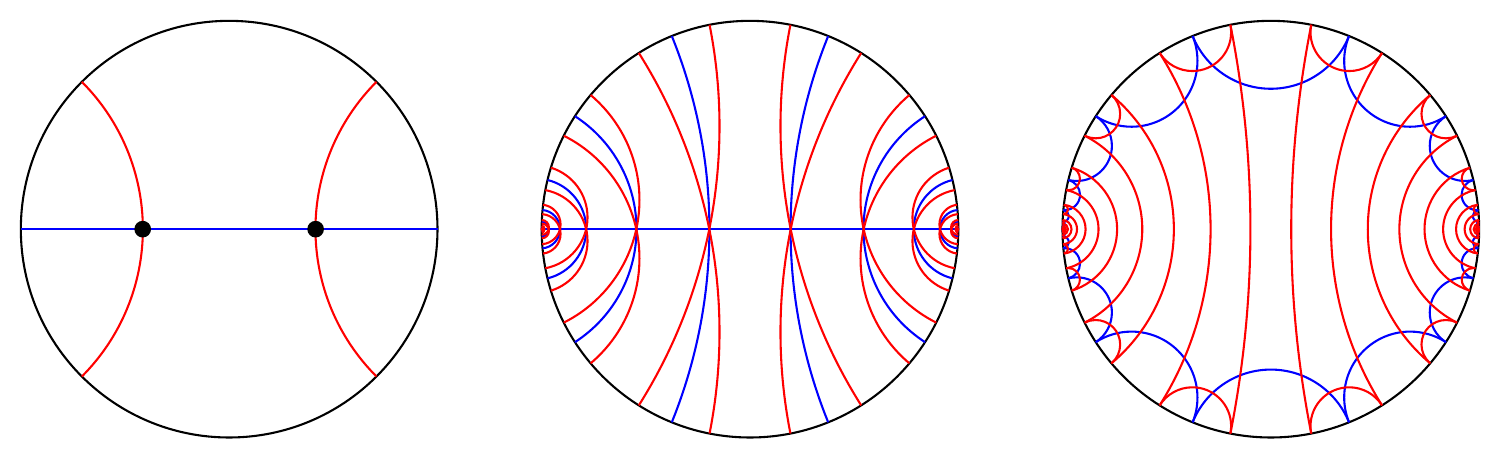}
\caption{The geodesics associated to $\lambda$ (blue) and $\mu$ and $\nu$ (red),
their preimages under the infinite dihedral cover, and boundary leaves in the
associated lamination in the cover}
\label{exotic_laminations}
\end{figure}
\end{example}

In the sequel we will consider CaTherine wheels with large automorphism groups.
Let's formalize the notion.

\begin{definition}[$G$-CaTherine wheel]\label{definition:G_wheel}
If $G$ is a group of orientation-preserving homeomorphisms of $S^2$, a
{\em $G$-CaTherine wheel} is a CaTherine wheel $f:S^1 \to S^2$ together with an
action of $G$ on $S^1$ by automorphisms such that $f$ intertwines the actions of $G$
on $S^1$ and $S^2$. 

If $f:S^1 \to S^2$ is a $G$-CaTherine wheel for some $\rho:G \to \Homeo^+(S^1)$ and
$\alpha\in \Homeo^+(S^1)$ is arbitrary, then $f\circ \alpha:S^1 \to S^2$ is a 
$G$-CaTherine wheel for $\rho^\alpha:G \to \Homeo^+(S^1)$ defined by 
$\rho^\alpha(g) := \alpha^{-1} \rho(g) \alpha$.
We say that the $G$-CaTherine wheels $f$ and $f\circ \alpha$ (with the given actions)
are {\em conjugate}. 
\end{definition}

\subsection{Hyperbolic 3-manifolds that fiber over the circle}

Let $M$ be a closed 3-manifold. If there is a fibration over the
circle $F \to M \to S^1$ then there is an associated short exact sequence of fundamental
groups $\pi_1(F) \to \pi_1(M) \to \Z$ where $\phi:\pi_1(M) \to \Z$ is the surjective homomorphism 
induced by $M \to S^1$.

The kernel of $\phi$ is $\pi_1(F)$ which is finitely generated (in fact, finitely
presented). Stallings \cite{Stallings} showed the converse: if
$\phi:\pi_1(M) \to \Z$ is a surjective homomorphism, and $\ker(\phi)$ is finitely
generated then $M$ fibers over the circle, and the fibering induces $\phi$
(in fact Stallings needed to assume further that $M$ is irreducible, but that hypothesis
is superfluous now that the Poincar\'e Conjecture has been proved).

Thurston \cite{Thurston_fibered} showed that $M$ as above is hyperbolic if and only
if the fiber $F$ has $\chi(F)<0$ and the monodromy is pseudo-Anosov (see
\cite{Thurston_dynamics}), and in this case one obtains a pseudo-Anosov suspension flow $X$. 
This means in particular that the monodromy preserves a
pair of projectively invariant transversely measured singular foliations $\F^\pm_F$ on
$F$ that lift to singular foliations $\tilde{\F}^\pm$ of the universal cover $\tilde{F}$ 
invariant under the natural action of $\pi_1(M)$. The data of $\tilde{\F}^\pm$ is
encoded by a pair of laminar relations $\LL^\pm$ without perfect fits
on the ideal boundary $S^1_\infty$ of $\tilde{F}$ .

The leaf spaces of $\tilde{\F}^\pm$ are topological $\R$-trees, and the endpoint
maps of the lifted suspension pseudo-Anosov flow $\tilde{X}$ on $\tilde{M}=\HH^3$ have
disjoint images $Z^\pm$ that are homeomorphic to these leaf spaces.

Finally, Cannon--Thurston \cite{Cannon_Thurston} used this structure to build a
map $f:S^1_\infty \to \CP^1$ equivariant under the natural action of $\pi_1(M)$ on
$S^1_\infty$ coming from the fibration and the geometric action of $\pi_1(M)$
on $\CP^1$ coming from the hyperbolic structure.

Setting $G=\pi_1(M)$, the four structures $X,f,Z^\pm,\phi$ are respectively 
examples of a pseudo-Anosov flow without perfect fits, a $G$-CaTherine wheel, 
a $G$-zipper, and a uniform quasimorphism on $G$. Looking ahead 
to \S~\ref{subsection:coarse_correspondence} we
shall see that there is an equivalence of these four different
structures in general.

\subsection{Kleinian groups and quasicircles}

\begin{definition}[Quasicircle]\label{definition:quasicircle}
A Jordan curve $\Gamma \subset \CP^1$ is a {\em $K$-quasicircle} if it is the image of a
round circle under a $K$-quasiconformal homeomorphism of $\CP^1$. It is a 
{\em quasicircle} if it is a $K$-quasicircle for some $K$.
\end{definition}

For an introduction to the theory of quasiconformal maps and quasicircles, see
Ahlfors \cite{Ahlfors_book} or Hubbard \cite{Hubbard_Teichmuller_1}. 

\begin{proposition}[Quasicircle]\label{proposition:quasicircle}
A Jordan curve $\Gamma \subset \CP^1$ is a quasicircle if and only if
every Hausdorff limit of translates of $\Gamma$ by conformal maps is 
either a Jordan curve or a single point.
\end{proposition}

This is essentially equivalent to Ahlfors' turning criterion for quasicircles: 
pairs of points in $\Gamma$ which are
close in $\CP^1$ are joined by an arc of $\Gamma$ of small diameter. According to
Curt McMullen, the characterization of quasicircles in Proposition~\ref{proposition:quasicircle}
is ``well known in Finland''.

\begin{definition}[$K$-CaTherine wheel]\label{definition:K_wheel}
For $K\ge 1$ a {\em $K$-CaTherine wheel} is a CaTherine wheel $f:S^1 \to \CP^1$ 
for which $\partial f(I)$ is a $K$-quasicircle for every closed interval $I\subset S^1$.
\end{definition}

As an easy application of Proposition~\ref{proposition:quasicircle}
one has the following elegant observation by McMullen \cite{McMullen}:

\begin{proposition}[McMullen, uniform $K$]\label{proposition:uniform_K}
Let $f:S^1 \to \CP^1$ be any CaTherine wheel invariant under a Kleinian group $G$
which is either 
\begin{enumerate}
\item{cocompact; or}
\item{a doubly degenerate closed surface group 
with the injectivity radius of $\HH^3/G$ bounded below.}
\end{enumerate}
Then $f$ is a $K$-CaTherine wheel for some $K$. Furthermore in the second case, $K$
depends only on the genus of the surface and the injectivity radius of $\HH^3/G$.
\end{proposition}
\begin{proof}
The first case follows immediately from Proposition~\ref{proposition:quasicircle}.
The set of (pointed) manifolds $\HH^3/G$ arising as in the second case is closed
under geometric limits and hence compact, and invariant under the obvious conjugation
action of the M\"obius group. Again, apply Proposition~\ref{proposition:quasicircle}.
\end{proof} 

\begin{example}[Surface bundles over a graph]\label{example:surface_graph}
This example was kindly explained to us by Chris Leininger.

Let $S$ be a closed oriented surface of genus at least $2$, and let $S \to Y \to \Gamma$
be a surface bundle over a finite graph $\Gamma$ (for concreteness, we can let $\Gamma$ be
a wedge of $k>1$ circles) for which $\pi_1(Y)$ is word-hyperbolic. 
Many such examples are known from the seminal work of Farb--Mosher 
\cite{Farb_Mosher}; they are associated to so-called `convex co-compact' free subgroups
of the mapping class group of $S$. Let $\tilde{\Gamma}$ denote the universal cover of $\Gamma$;
this is a regular $2k$-valent tree. For every proper embedded line $\ell$ in $\tilde{\Gamma}$ the
total space of the pulled back $S$ bundle over $\ell$ is uniformly quasi-isometric
to $\HH^3/G$ where $G$ is a doubly degenerate surface group isomorphic to $\pi_1(S)$, 
and furthermore the injectivity radius in this quotient is uniformly bounded 
below over all such $\ell$.
 
Let $\CC$ denote the space of (Freudenthal) ends of $\tilde{\Gamma}$; this is a Cantor set.
It follows that we obtain a parameterized family of CaTherine wheels 
$$S^1 \times (\CC \times \CC - \Delta) \to S^2 \times (\CC \times \CC - \Delta) \to \partial_\infty \pi_1(Y)$$
whose image $S^2$s each surject in a complicated way to $\partial_\infty \pi_1(Y)$; see
Mj \cite{Mitra} for details.

Each one of these $S^2$s has a natural conformal structure for which the associated
CaTherine wheel is a $K$-CaTherine wheel, with $K$ uniform over the entire family, by
Proposition~\ref{proposition:uniform_K}.
Can one use the images of these $K$-quasicircles in $\partial_\infty \pi_1(Y)$ to 
estimate (e.g.\/ by the method of Gromov as developed 
by Bourdon \cite{Bourdon,Gromov_asymptotic}) the conformal dimension of $\partial_\infty \pi_1(Y)$?
\end{example}

Because of Proposition~\ref{proposition:uniform_K} we are emboldened to make a conjecture:

\begin{conjecture}[Injectivity radius]\label{conjecture:injectivity_radius}
For any $R > r > 0$ there is a finite $K(r,R)$ such that
any CaTherine wheel invariant under a Kleinian group $G$
for which $\HH^3/G$ has injectivity radius in the interval $[r,R]$ everywhere
is a $K(r,R)$-CaTherine wheel.
\end{conjecture}

We shall see some considerable evidence for Conjecture~\ref{conjecture:injectivity_radius} 
in \S~\ref{subsubsection:finiteness_and_compactness} after we relate $G$-CaTherine 
wheels to pseudo-Anosov flows without perfect fits.

If $f:S^1 \to \CP^1$ is a $K$-CaTherine wheel, then so is $f \circ g:S^1 \to \CP^1$
for any $g\in\Homeo^+(S^1)$. Up to this flexibility (i.e.\/ choice of `gauge' $g$),
and without any assumptions on dynamics, we have the following:
\begin{proposition}[$K$ is compact]\label{proposition:K_is_compact}
Let $f_n:S^1 \to \CP^1$ be any sequence of $K$-CaTherine wheels. Then after passing
to a subsequence, and precomposing with a sequence $g_n \in \Homeo^+(S^1)$ if
necessary, $f_n \to f_\infty$ in the compact-open topology, and 
$f_\infty:S^1 \to \CP^1$ is a $K$-CaTherine wheel. 
\end{proposition}
\begin{proof}
Fix a round metric on $\CP^1$ in the correct conformal class. After making this
(arbitrary) choice, for each $K$-CaTherine wheel $f_i$ and
each interval $I \subset S^1$, the image $f_i(I)$ has positive Lesbesgue measure.
Since the boundary of $f_i(I)$ is a $K$-quasicircle, it has 
Hausdorff dimension at most $2K/(K+1) < 2$ and therefore Lesbesgue measure $0$
(see e.g.\/ \cite{Astala}, Corollary~1.3). 
Thus in particular, and because the images of disjoint intervals have disjoint
interiors, this assignment of Lesbesgue measure to intervals determines a
measure of full support on $S^1$ with total mass $4\pi$ and no atoms. We can therefore find 
a gauge $g_i$ (that we call the {\em Lebesgue gauge}) 
so that $\area(f_i\circ g_i(I)) = 2\cdot\length(I)$ for all $I$, and for the usual
metric on $S^1$.

Now it is a fact that for any $K$-quasidisk $D$ in the Euclidean plane (or the
round sphere) there is a constant $C(K)$ so that 
$$\frac {2} {\sqrt{\pi}}\cdot\sqrt{\area(D)} \le \text{diameter}(D) \le C(K)\cdot \sqrt{\area(D)}$$
The first inequality is just the usual isodiametric inequality for the plane, and has nothing
to do with quasidisks per se. The second inequality may be proved as follows.
Pick a pair of points $p,q\in \partial D$ with $d(p,q)=\text{diameter}(D)$. If
$\text{diameter}(D) \gg \sqrt{\area(D)}$ then by the coarea formula, we can find a
line $\ell$ perpendicular to the segment $pq$ and far from both $p$ and $q$ so that
every component of $\ell \cap D$ is very short. But this means we can find a pair
of points on $\ell \cap \partial D$ that are very close, but can't be joined by
a path of small diameter in $\partial D$, contrary to Ahlfors' turning criterion
for quasicircles. We remark that the dependence of $C(K)$ on $K$ is of order $e^{cK}$,
although we do not use this fact.

From this inequality it follows that
in the Lebesgue gauge, a $K$-CaTherine wheel is H\"older continuous with exponent $1/2$,
and has a uniform modulus of continuity depending only on $K$.

Thus after passing to a subsequence, we may assume $f_n$ (in the Lebesgue
gauge) converges in the compact-open topology to some (necessarily surjective) map
$f_\infty:S^1 \to \CP^1$. Evidently $f_\infty$ is not locally constant, since for any
fixed interval $I$ the images $f_i(I)$ have diameter bounded below; conversely it is not
all of $\CP^1$ since the same argument applies to $f_i(I^c)$. But $f_i(I)$
is a sequence of $K$-quasidisks converging to $f_\infty(I)$ and therefore by
Proposition~\ref{proposition:quasicircle} the limit is a $K$-quasidisk. 
Evidently the limit of $\partial f_n(I)$ is
$\partial f_\infty(I)$ and therefore by continuity, $f_\infty(\partial I) \subset
\partial f_\infty(I)$ and we are done.
\end{proof}
In other words, up to reparameterization (`gauge'), the space of $K$-CaTherine wheels
is compact. By the way, the existence of Lebesgue gauge only depends on the
fact that the area of $\partial f(I)$ is zero for all closed intervals $I \subset S^1$.

In the Lebesgue gauge, the action of $g \in G$ near its periodic points in $S^1$
is especially nice.
\begin{lemma}[Bilipschitz]\label{lemma:bilipschitz}
Suppose $f:S^1 \to \CP^1$ is a $G$-CaTherine wheel for some Kleinian group $G$
that admits a Lebesgue gauge (i.e.\/ the Lebesgue area of $\partial f(I)$ is zero for all
closed intervals $I\subset S^1$). Let $g \in G$ be nontrivial, and suppose that
$|g_*d\area/d\area| \in [C^{-1},C]$ on $\CP^1$. Then 
$g$ acts on $S^1$ by a $C$-bilipschitz homeomorphism of $S^1$ in the
Lebesgue gauge.

Furthermore, if $p \in S^1$ is an attracting resp. repelling periodic orbit for $g$,
and $\tau(g)$ is the translation length of $g$ on $\HH^3$, 
then for every $\epsilon$ there is an open neighborhood $I_p$ of $p$ so that 
the derivative of $g$ restricted to $I_p$ is within $\epsilon$ of $e^{-2\tau(g)}$
resp. $e^{2\tau(g)}$.
\end{lemma}
\begin{proof}
The first statement is immediate from the definition of the Lebesgue gauge, and the
way in which the action of $g$ transforms area on $\CP^1$.

Any nontrivial $g\in G$ has only finitely many periodic 
points in $S^1$, that alternate between attracting and repelling, by 
Theorem~\ref{theorem:automorphisms}. If $I_p$ is a sufficiently small neighborhood
of one such point, $f(I_p)$ is contained in a small neighborhood of a
fixed point of $g$ on $\CP^1$ where the corresponding statement is obvious.
\end{proof}

The hypothesis of Lemma~\ref{lemma:bilipschitz} is satisfied whenever
$f$ is a $K$-CaTherine wheel for any $K$. It follows that if $f_i:S^1 \to \CP^1$ 
are a sequence of $G$-CaTherine wheels which are also $K$-CaTherine wheels 
(an unfortunate clash of terminology) the action of each $g\in G$ on $S^1$ 
in the Lebesgue gauge is uniformly bilipschitz independent of $i$, 
and the limit $f_\infty:S^1 \to \CP^1$ 
is also a $G$-CaTherine wheel, and with the same bilipschitz 
constants for the $G$ action on $S^1$.

In fact for $G$ a fixed cocompact Kleinian group, we have the following proposition:
\begin{proposition}[Actions are precompact]\label{proposition:precompact}
Let $G$ be a cocompact Kleinian group, and let $f_i:S^1 \to \CP^1$ be a sequence
of $G$-CaTherine wheels. Then in the Lesbesgue
gauge, the sequence of associated $G$ actions $\rho_i:G \to \Homeo^+(S^1)$ 
has a convergent subsequence.
\end{proposition}
\begin{proof}
For $G$ cocompact, any $G$-CaTherine wheel is a $K$-CaTherine wheel for some $K$.
Let $\rho_i:G \to \Homeo^+(S^1)$ be the action of $G$ on $S^1$ corresponding to $f_i$
in the Lesbesgue gauge (the representation $\rho_i$ is only well-defined up to conjugacy
by a rigid rotation of $S^1$). By Lemma~\ref{lemma:bilipschitz} for each
$g\in G$ the element $\rho_i(g)$ is $C$-bilipschitz for some $C(g)$ independent of $i$.
Thus each $\rho_i(g)$ has a convergent subsequence, and since $G$ is finitely generated, 
there is a limit $\rho_\infty:G \to \Homeo^+(S^1)$ for which each $\rho_\infty(g)$
is still $C(g)$-bilipschitz with the same constant.
\end{proof}
If one could show that the limiting action $\rho_\infty:G \to \Homeo^+(S^1)$
extended to a $G$-CaTherine wheel $f:S^1 \to \CP^1$, then a posteriori there would be
a uniform $K$ for all $G$-CaTherine wheels, and (looking ahead) this would give a new
proof of Theorem~\ref{theorem:BTZ_finiteness}.

We remark that if $\LL^\pm_i$ are the $\rho_i(G)$-invariant laminar relations associated
to $f_i$, then some subsequence converges to $\rho_\infty(G)$-invariant laminar relations
$\LL^\pm_\infty$. These limit laminar relations have no isolated sides, but a priori
there is no reason why they should have no perfect fits.

We end this subsection with an inequality relating the analytic quality of 
a $G$-CaTherine wheel to the dynamics of $G$ on $S^1$:
\begin{proposition}[$K$ bounds valence]\label{proposition:K_bounds_valence}
Let $f:S^1 \to \CP^1$ be a $K$-CaTherine wheel. Then there is a bound $n(K)$ on
the number of $f$ preimages of any point in $\CP^1$. 

In particular, if $f$ is also a $G$-CaTherine wheel for some subgroup $G$
of M\"obius transformations of $\CP^1$, then with the same bound $n(K)$ as above, 
every nontrivial $g\in G$ has at most $2n(K)$ periodic points in $S^1$.
\end{proposition}
\begin{proof}
Since $f$ is a $K$-CaTherine wheel, it has a Lebesgue gauge.
Suppose there is a point $p\in \CP^1$ with $n$ preimages $q_1,\cdots,q_n \in S^1$.
In the Lebesgue gauge, we can put disjoint intervals $J_1 \cdots J_n \subset S^1$ 
centered at the $q_i$ each of length some small positive number $A$.
The images $f(J_1) \cdots f(J_n)$ have disjoint interiors, 
but have a point in common (the common fixed point of $g$ in $\CP^1$). 
But then the union $B: = \cup f(J_i)$ has area $nA$ but diameter at most 
$2 C(K)\cdot \sqrt{A}$ by the diameter--area estimate for $K$-quasidisks. On the other hand,
this violates the elementary diameter--area inequality in round $S^2$ for 
fixed $C(K)$ if $n$ is big enough.

Now suppose $f$ is a $G$-CaTherine wheel for some M\"obius group $G$.
By Theorem~\ref{theorem:automorphisms} any nontrivial $g\in G$ 
has only finitely many periodic orbits on $S^1$ 
that alternate between attracting and repelling, and the images of these in $\CP^1$ are
(respectively) the unique attracting and repelling fixed points for $g$.
\end{proof}

Bounds on the cardinality of point preimages of certain Cannon--Thurston maps are obtained via
$\delta$-hyperbolic (as distinct from analytic) methods in \cite{BHLM_bound}.

\begin{corollary}[No parabolics]\label{corollary:no_parabolics}
Let $G$ be a Kleinian group containing a parabolic element. Then no $G$-CaTherine wheel
is a $K$-CaTherine wheel.
\end{corollary}
\begin{proof}
Let $f:S^1 \to S^2$ be a $G$-CaTherine wheel.
Let $g\in G$ be parabolic and let $x\in S^2$ be its unique fixed point. 
Then by Theorem~\ref{theorem:automorphisms} the element
$g$ fixes exactly two points $p,q$ in $S^1$ in the preimage of $x$
with source-sink dynamics. Furthermore,
the leaf $\lbrace p,q\rbrace$ is contained in a nontrivial equivalence class $\mu$
in $\LL^+$ (say), but it is not a boundary leaf. In particular, $\mu$ contains points
on either side of the leaf $\lbrace p,q\rbrace$, and since $\mu$ is $g$-invariant it
contains infinitely many points, all in the preimage of $x$. Thus by
Proposition~\ref{proposition:K_bounds_valence}, $f$ is not a $K$-CaTherine wheel for
any $K$.
\end{proof}

\begin{example}[Injectivity radius]\label{example:injectivity}
The idea of this example was kindly suggested to us by Mahan Mj.

Let $G$ be a doubly degenerate closed surface group (i.e.\/ one without parabolics), 
but for which the infimum of the injectivity radius on $\HH^3/G$ is zero. It is still true that
the Cannon--Thurston map $f:S^1 \to S^2$ is a $G$-CaTherine wheel, but it is not true
that it is a $K$-CaTherine wheel for any $K$. 
For, by Proposition~\ref{proposition:K_is_compact} for any fixed $K$ 
the space of $K$-CaTherine wheels up to conjugacy is compact. On the other hand,
the group $G$ has the property that there is a geometric limit of $\HH^3/G$ with cusps.
Call this geometric limit $\HH^3/G'$ where $G'$ contains parabolic elements. If
$f$ were a $K$-CaTherine wheel we would obtain a sequence of $K$-CaTherine wheels
as conjugates of $f$ and in the limit we would obtain a $K$-CaTherine wheel invariant
under $G'$, contrary to Corollary~\ref{corollary:no_parabolics}. 

Compare with Proposition~\ref{proposition:uniform_K}.
\end{example}

\begin{example}[Numerical estimate of Hausdorff dimension]\label{example:numerical_estimate}
For a cocompact Kleinian group $G$ and a $G$-CaTherine wheel $f:S^1 \to \CP^1$,
every path in the zipper $Z^\pm$ has Hausdorff dimension at most $2K/(K+1)$ where $K$
is as in Proposition~\ref{proposition:uniform_K}. However, it can easily be that
$K \to \infty$ while the Hausdorff dimension stays bounded away from $2$.

Let $M_n$ be obtained by $(0,n)$ orbifold filling on the figure 8 knot complement,
and let $G_n$ be the orbifold fundamental group of $M_n$, realized as a Kleinian group.
The orbifold $M_n$ fibers over the circle, with fiber a torus with one cone point of
order $n$. The associated Cannon--Thurston map is a CaTherine wheel 
$f_n:S^1 \to \CP^1$ for $G_n$, and Table~\ref{table:dimension} contains numerical
estimates of the Hausdorff dimension of the zippers $Z^\pm$ as a function of $n$.

\begin{table}[ht]\label{table:dimension}
\begin{tabular}{r|l}
$n$ & $D$ \\
\hline
2 & 1.08108 \\
3 & 1.14815 \\
4 & 1.19000 \\
5 & 1.21446 \\
6 & 1.23218 \\
10 & 1.26401 \\
50 & 1.29661
\end{tabular}
\end{table}

These numerical estimates were obtained as follows.
Fix generators $a,b,t$ for $G_n$ where $a,b$ are generators for the fiber, and $t$ is
the meridian; thus $t$ conjugates $a$ to $ab$ and $b$ to $bab$. Let $p\in \CP^1$ be
a stable fixed point for $t$; thus $p\in Z^+(f_n)$. 

For each fixed $n$, choose a big integer $m$ and express the conjugate 
$t^{-m}at^m$ as a geodesic word $w_m$ in the subgroup $\langle a,b\rangle$.
The images of $p$ under the sequence of elements $t^mw_m(i)$ where $w_m(i)$ are the initial 
subwords of $w_m$ of length $i$ interpolate between $p$ and $ap$; taking
$m$ large produces a piecewise linear curve $\gamma_m\subset \CP^1$ 
which is a numerical approximation
to the true curve $\gamma \subset Z^+(f_n)$. For each $\epsilon$ we can perform this 
subdivision until successive points on $\gamma_m$ are approximately distance $\epsilon$
apart, and estimate the Hausdorff dimension $D$ by linear regression of $\log|\gamma_m|$
versus $-\log(\epsilon)$ using the heuristic $|\gamma_m| \sim C\cdot \epsilon^{1-D}$. 
Note that although the Hausdorff dimensions
are converging to a number much smaller than 2, the limit of the curves $\gamma$
as $n \to \infty$ is not a quasiarc, since the limit group $G_\infty$ 
has parabolic elements giving rise to visible cusps; see Figure~\ref{z_arcs}. 

\begin{figure}[htpb]
\centering
\includegraphics[scale=0.3]{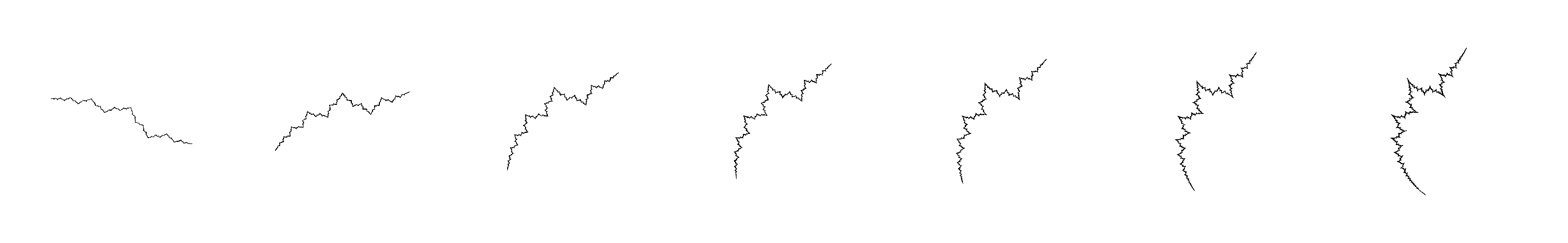}
\caption{Arcs in $Z^+(f_n)$ for $G_n$-CaTherine wheels for $n=2,3,4,5,6,10,50$.}
\label{z_arcs}
\end{figure}

In principle we may obtain explicit lower bounds on $K_n$, the least number such that 
$f_n$ is a a $K_n$-CaTherine wheel, as follows. In the faithful action of $G_n$ on
$S^1$, the fixed points of $t$ are nontrivially permuted by the (order $n$) commuting
element $[a,b]$. In particular, some power of $t$ has at least $n$ fixed points on $S^1$,
giving a lower bound on $K_n$ via Proposition~\ref{proposition:K_bounds_valence}. This
lower bound is of order $O(\log{n})$; we have no idea if this is the correct order
of magnitude.

The arc $\gamma$ associated to $G_\infty$ is an example of what is called a
{\em lightning curve} by Dicks et. al. see \cite{Dicks_Cannon_I,Dicks_Cannon_II,Dicks_Wright}. 
Automata are constructed in \cite{Dicks_Wright} to draw these curves, 
which (more or less) specialize to the
algorithm we use here. By considering the action of $G_n$ on $S^1$, one may use 
these automata to determine the Cantor set $C$ in $S^1$ mapped to $\gamma$ by $f_n$, 
with notation as in Proposition~\ref{proposition:monotonicity}.
\end{example}

\subsection{Pseudo-Anosov flows}\label{subsection:pseudo_Anosov}

Suppose $M$ is a closed hyperbolic 3-manifold $M$ with fundamental group $G$ acting on
$\CP^1$. In this subsection we show that a $G$-CaTherine wheel gives rise to a
pseudo-Anosov flow on $M$ without perfect fits. This is the converse of a deep theorem of
Fenley \cite{Fenley_no_perfect_fits}, and together with Fenley's theorem
this gives an equivalence between such flows up to orbit equivalence, 
and $G$-CaTherine wheels up to conjugacy.

There is a vast literature on the theory of pseudo-Anosov flows, and there are
many equivalent ways of representing such structures (bifoliated planes, branched surfaces, 
veering triangulations etc). See e.g.\/
\cite{Mosher, Agol, Landry_Tsiang, Baik_Jung_Kim, Baik_Wu_Zhao, Barbot_Fenley, 
Fenley_good, Fenley_quasigeodesic, Barthelme_Bonatti_Mann, 
Barthelme_Frankel_Mann, Frankel_Schleimer_Segerman}.
There is a considerable overlap of ideas and methods among these papers, and some of
the same ideas go into the proof of our theorem, and the foundations that it rests
on that were developed in Part~\ref{part:point_set_topology}.

\begin{theorem}[Pseudo-Anosov flows and CaTherine wheels]\label{theorem:wheels_and_flows}
Let $M$ be a closed hyperbolic 3-manifold with fundamental group $G$. 
Then there is a bijection between orbit equivalence classes of 
pseudo-Anosov flows  without perfect fits $X$ on $M$ and
conjugacy classes of $G$-invariant CaTherine wheels $f:S^1 \to S^2_\infty$.
\end{theorem}
\begin{proof}
One direction is the Main Theorem of Fenley from \cite{Fenley_no_perfect_fits} that
we summarize as follows: the lifted flow $\tilde{X}$ has an orbit space which is 
homeomorphic to a plane that admits a natural compactification by a circle,
and the stable/unstable foliations of the flow determine a pair of laminar relations 
of this circle with no perfect fits and no isolated sides. Thus there is a CaTherine wheel
$f:S^1 \to S^2$ (Fenley constructs this map directly). Fenley shows that the 
action of $G$ on $S^2$ is a convergence action, and therefore by Bowditch the action is
conjugate to the action of $G$ on $S^2_\infty$. This proves one direction, which
is Fenley's theorem restated in our language.

Conversely, let $f: S^1 \to S^2_\infty$ be a CaTherine wheel invariant under $G$.
Then there is a $G$-zipper $Z^\pm \subset S^2_\infty$ for which $Z^+$ and $Z^-$ are
homeomorphic to quotient spaces induced by decompositions of planes $P^\pm$
coming from a pair of laminar relations $\LL^\pm$ of $S^1$ without perfect fits. 
Identify $S^1$ with the ideal boundary $S^1_\un$ of
$\HH^2$, and thereby identify $P^\pm$ with $\HH^2$ in such a way that the boundary leaves
are honest geodesics.

Now let's intersect decomposition elements associated to $\LL^+$ and $\LL^-$. Identify $P^+$
and $P^-$ with a single copy of $\HH^2$, and form a new
decomposition of $\HH^2$ whose elements are pairwise intersections of hulls $H^+(\mu)$ and
$H^-(\nu)$ for positive resp. negative laminar equivalence classes $\mu$ and $\nu$. 
The topology (though not the geometry) of this decomposition does not depend on any choices,
and is preserved by $G$.
The decomposition elements are closed and convex, because they are the intersection of 
closed convex subsets of $\HH^2$. They are compact, because otherwise $H^+(\mu)$ and
$H^-(\nu)$ would have an ideal point in common, contradicting the no perfect fits property.
Frontier points are contained in boundary leaves, thus these decomposition elements are all
compact finite sided polyhedra (including the degenerate case of a closed geodesic interval
or a single point). Since boundary leaves are not isolated, it follows easily that this
decomposition is upper semi-continuous; since the elements are non-separating, the quotient
space is homeomorphic to a plane $P$.

The plane may be identified with a subset $P \subset Z^+ \times Z^-$ where the image in
each factor is the projection to the further quotient space of $\HH^2$ by the decomposition
associated to the positive and negative laminar relations respectively.

Now, the unit tangent bundle of $\HH^3$ admits a foliation by lifts (from $\HH^3$) of 
bi-infinite geodesics, giving it the structure of a product 
$UT\HH^3 = (S^2 \times S^2 - \Delta) \times \R$ and inside this space we take
$\tilde{N}: = P \times \R$, using the given embedding of $P$ in 
$Z^+\times Z^- \subset S^2 \times S^2 - \Delta$. Ignoring its subspace topology for now, 
$\tilde{N}$ is certainly homeomorphic to $\R^3$. Since the map from the plane to
$S^2 \times S^2 - \Delta$ is continuous and injective, its restriction to every
compact subset is a homeomorphism onto its image, and therefore compact subsets
of $\tilde{N}$ with its $\R^3$ topology are also compact in $UT\HH^3$. But
$G$ acts on $UT\HH^3$ freely and properly discontinuously, in the strong sense that
for any compact subset $K$ of $UT\HH^3$ the intersections $K\cap gK$ are empty for
all but finitely many $g$. This is true in particular for $K$ which are compact
subsets of $\tilde{N}$, and therefore the action of $G$ on $\tilde{N}$ is also free 
and properly discontinuous in the strong sense. This implies that the quotient $N:=\tilde{N}/G$
is Hausdorff (see e.g.\/ \cite{Kapovich}, Lemma~9 and Theorem~11), 
and the quotient map $\tilde{N} \to N$ is a covering space, so that $N$ is a 
(Hausdorff!) 3-manifold with fundamental group $G$.
Furthermore, $N$ comes equipped with a tautological flow $X$ coming from the $\R$ factors in 
$UT\HH^3$, which are M\"obius invariant. Hence in particular $N$ is a closed 3-manifold
which by Mostow rigidity is homeomorphic to $M$, and a posteriori $\tilde{N}$ 
is properly embedded in $UT\HH^3$ so that its subspace topology agrees with 
its $\R^3$ topology. 

It is easy to see from the structure of the decomposition space
that the tautological flow $X$ on $N$ is already a (topological)
pseudo-Anosov flow without perfect fits, since it has singular stable/unstable 
foliations coming from the separate decompositions associated to $\LL^+$ and $\LL^-$ 
respectively.

Alternately, by considering the intersection of $\tilde{N}$ with translates
of a compact fundamental domain for $G$ on $UT\HH^3$ that march along a flowline $\ell$
of $\tilde{X}$ we can see that $\ell$ is quasigeodesic in $UT\HH^3$ and therefore
also in $\tilde{N}$, and then the main theorem of Frankel--Landry \cite{Frankel_Landry_flows},
Theorem~1.1 implies that this quasigeodesic flow may
be straightened to a pseudo--Anosov flow (unique up to orbit equivalence), 
with the same action on the circle at infinity
(which is the $S^1$ from the CaTherine wheel). 

In any case we obtain a canonical route from $G$-CaTherine wheels to pseudo--Anosov
flow without perfect fits, precisely inverse to Fenley's construction.
\end{proof}

\begin{remark}
In fact, the major part of Frankel--Landry's argument, starting with a geodesic flow, 
is precisely to construct the data of a (universal) circle with a pair of especial
laminar relations. When the relations have no perfect fits, they exactly satisfy the
sufficient conditions to arise from a CaTherine wheel; Frankel--Landry's argument at
this point is parallel to ours, though more roundabout since they are not able to work
directly with the zipper $Z^\pm$ in $S^2$. Thus logically the appeal to Frankel--Landry 
is superfluous. 
\end{remark}
 
In fact, we may construct a pseudo-Anosov flow from a CaTherine wheel even without 
assuming the CaTherine wheel comes from a 3-manifold. Recall that Cannon conjectured
(this conjecture is implicit in \cite{Cannon}, \S~8) 
that if $G$ is a word hyperbolic group with Gromov boundary
$\partial_\infty G$ homeomorphic to $S^2$, then $G$ is virtually the fundamental group
of a closed hyperbolic 3-manifold. We shall now show that this is true whenever
there is a $G$-invariant CaTherine wheel $f:S^1 \to \partial_\infty G$.

\begin{theorem}[CaTherine wheels and Cannon's Conjecture]\label{theorem:wheel_Cannon_conjecture}
Suppose $G$ is a word hyperbolic group with Gromov boundary $\partial_\infty G$
homeomorphic to $S^2$, and suppose that there is a $G$-invariant CaTherine wheel
$f:S^1 \to \partial_\infty G$. Then $G$ is virtually the fundamental group of a closed
hyperbolic 3-manifold; i.e.\/ the Cannon Conjecture is true for $G$.
\end{theorem}
\begin{proof}
For brevity, let's denote $\partial_\infty G \times \partial_\infty G - \Delta$ by
$\partial^2_\infty G$.
As in the proof of Theorem~\ref{theorem:wheels_and_flows} we obtain a $G$-invariant
continuous injective map from a plane to $P \subset Z^+ \times Z^-$.
Mineyev's flow space (see \cite{Mineyev}, Theorem~60) is a metric space $\F(G)$ quasi-isometric
to $G$ with a cocompact isometric $G$-action, and is homeomorphic to 
$\partial^2_\infty G \times \R$ in such a way that the $G$ action preserves the foliation
by $\R$ factors and descends to the usual $G$ action on $\partial^2_\infty G$.
(Technically what Mineyev denotes by $\F(G)$ is the $\Z/2\Z$ quotient of this space
by a free involution that interchanges the two $\partial_\infty G$ factors, and reverses
the orientation of each $\R$ leaf; this is the difference between the unit tangent bundle
and the projective unit tangent bundle).

We may therefore form, just as above, the product $\tilde{N}:=P\times \R \subset \F(G)$ and observe
that it is $G$-invariant and homeomorphic to $\R^3$ with a $G$-invariant 
product-covered foliation by lines. Now, we cannot assume a priori that $G$
is virtually torsion free. But it does act properly discontinuously on $\tilde{N}$.
We claim the quotient space $N:=\tilde{N}/G$ has the structure of a smooth 3-dimensional
orbifold. Although finite
groups can act on $\R^3$ with wild fixed point sets, the action of $G$ on $\tilde{N}$
descends to an action on $P$. As is well-known, a theorem of K\'er\'ekjart\`o \cite{Kerekjarto}
easily implies that any action of a finite
group on $\R^2$ is topologically conjugate to a finite group of isometries of the Euclidean
plane, and therefore $N$ has the structure of a smooth 3-dimensional orbifold. Thus
a posteriori $G$ is virtually torsion free, and without loss of generality $N$ is a
3-manifold with fundamental group commensurable to $G$, and incidentally also admits
a pseudo-Anosov flow without perfect fits, by Theorem~\ref{theorem:wheels_and_flows}.

In particular, Cannon's Conjecture is true for $G$.
\end{proof}

\subsubsection{Finiteness and compactness}\label{subsubsection:finiteness_and_compactness}

If Conjecture~\ref{conjecture:injectivity_radius} were true, it would imply 
(together with Proposition~\ref{proposition:K_is_compact}) that
the space of $G$-CaTherine wheels for any fixed cocompact Kleinian group $G$ is compact.
By Theorem~\ref{theorem:wheels_and_flows} $G$-CaTherine wheels (up to gauge) are
in bijection with pseudo-Anosov flows on $M:=\HH^3/G$ without perfect fits. Since
pseudo-Anosov flows on a closed hyperbolic 3-manifold are structurally stable, a
compact family of $G$-CaTherine wheels is necessarily finite. Thus finiteness of
orbit-equivalence classes of pseudo-Anosov flows without perfect fits on $M$ is 
{\em equivalent} to a fixed $K(M)$ so that every $G$-CaTherine wheel is a $K(M)$-CaTherine
wheel. 

But in fact this finiteness is a very recent {\em theorem} of 
Barthelm\'e--Tsang--Zung \cite{BTZ}, building on work of Li \cite{Li_singular}.
Explicitly:

\begin{theorem}[Barthelm\'e--Tsang--Zung \cite{BTZ} finiteness]\label{theorem:BTZ_finiteness}
Let $M$ be a closed hyperbolic 3-manifold. Then there are finitely many orbit equivalent
pseudo-Anosov flows without perfect fits $X$ on $M$.
\end{theorem}

In turn, this theorem depends crucially on the following finiteness theorem of Li:

\begin{theorem}[Li \cite{Li_singular} finiteness]\label{theorem:Li_finiteness}
Let $M$ be a closed hyperbolic 3-manifold. Then there are only finitely many
isotopy classes of closed curves $\gamma$ in $M$ that are singular orbits of
pseudo-Anosov flows without perfect fits $X$ on $M$.
\end{theorem}

Of course, this trivially implies that for any $0<r$ and $v < \infty$ there is a
$K(r,v)$ so that any $G$-CaTherine wheel invariant under a cocompact Kleinian group
$G$ for which $\HH^3/G$ has injectivity radius $\ge r$ and volume $\le v$ is a
$K(r,v)$-CaTherine wheel, since there are only finitely many such $G$ up to conjugacy,
and each one is compatible with only finitely many CaTherine wheels.

\subsection{Zippers}\label{subsection:G_zippers}

Suppose $G$ is the fundamental group of a closed hyperbolic 3-manifold $M$,
thought of as a Kleinian group acting on $S^2_\infty$.
How can we construct $G$-CaTherine wheels without assuming that $M$ admits
a pseudo-Anosov flow without perfect fits and invoking Fenley~\cite{Fenley_no_perfect_fits}?

One way to do it is via the theory of $G$-zippers. Every $G$-CaTherine wheel 
$f:S^1 \to S^2$ gives rise to a unique (necessarily minimal) $G$-zipper. 
In this subsection we prove the 
converse: every minimal $G$-zipper gives rise to a $G$-CaTherine wheel, unique up 
to conjugacy; and this correspondence is a bijection. We prove this in two steps;
first we show (Theorem~\ref{theorem:G_zipper_to_G_wheel}) that a hairy $G$-zipper
always comes from a (unique up to conjugacy) $G$-CaTherine wheel, and then we show
(Theorem~\ref{theorem:G_zippers_are_hairy}) that every minimal $G$-zipper is hairy.

\begin{theorem}[Hairy $G$-zipper to $G$-CaTherine wheel]\label{theorem:G_zipper_to_G_wheel}
Let $G$ be the fundamental group of a closed hyperbolic 3-manifold $M$. 
Suppose that $Z^\pm \subset S^2_\infty$ is a minimal $G$-zipper
and suppose further that $Z^\pm$ is hairy. Then the laminar relations $\LL^\pm$
associated to $Z^\pm$ have no perfect fits, and there is a $G$-CaTherine wheel $f:S^1 \to S^2$
for {\em the given} $G$ action (unique up to conjugacy) and with associated zippers $Z^\pm$.
\end{theorem}
\begin{proof}
By hairiness, $S^1_\pm$ may be identified with the union of the ends and ideal gaps
of $Z^\pm$. For a minimal $G$-zipper, the action of $G$ on $S^1_\un$ is minimal, and 
therefore $S^1_\un$ may be identified canonically with each of $S^1_\pm$. 
By Lemma~\ref{lemma:zipper_relations} we obtain a pair of laminar relations $\LL^\pm$
each with no isolated sides. 

If we can show $\LL^\pm$ have no perfect fits, we will obtain a CaTherine wheel
$f:S^1_\un \to S^2$ for {\em some} $G$-action on $S^2$. A more careful analysis will show
both that there are no perfect fits, and that the range of $f$ may be taken to be
$S^2_\infty$ with the given $G$-action. In fact, the same argument will prove both
claims.

We would like to define the map $f:S^1_\un \to S^2_\infty$ as follows: a point
$e\in S^1_\un$ corresponds to a pair of points $e^\pm \in S^1_\pm$. Each of the
points $e^\pm$ either corresponds to an equivalence class of proper ray, or to an
ideal gap. 

The argument may be shortened by appealing to 
Kim \cite{Kim}, Theorem~B which says that for any minimal $G$-zipper,
landing rays of type 2 exist, and by minimality they are dense. 
A type 2 landing ray determines a pair of points $x^\pm \in S^1_\pm$
corresponding to a single point $x\in S^1_\un$
for which one of the $x^\pm$ is an end and the other is an ideal gap (where a proper ray
representing the end lands). 
A pair of type 2 landing rays determines two pairs of points $x^\pm,y^\pm$
and we may form a Jordan curve $\gamma$ in $Z^+\cup Z^-$ which is the union of the
convex hull of $x^+,y^+$ in $Z^+$ and the convex hull of $x^-,y^-$ in $Z^-$. 
Associated to $e\in S^1_\un$ we may choose any sequence of points $e_i \in S^1_\un$
that converge to $e$ from alternating sides, and for each consecutive 
pair $e_{2i}^\pm,e_{2i+1}^\pm \subset S^1_\pm$ we form a Jordan curve 
$\gamma_i \subset Z^+\cup Z^-$ as above. 

The $\gamma_i$ bound a nested decreasing sequence of disks $D_i$, and the intersection
of this family of disks is a compact, connected, cellular set $K_e$ associated
canonically to $e$. Notice by the way that the construction
of $K_e$ is functorial, so that $gK_e = K_{ge}$ for every $g\in G$.
A similar subset $K$ was constructed in 
Proposition~\ref{proposition:canonical_identification}; the key in that proposition was
to show that $K$ consisted of a single point, and that is the key here too.

We explain how the theorem follows if we can show that every $K_e$ consists of a
single point. First of all, $\LL^\pm$ can have no
perfect fits, since otherwise for some $e$ the points $e^\pm$ would both correspond
to ideal gaps with support $p^\pm \in Z^\pm$. By construction, $p^\pm \subset D_i$ for
every $D_i$ as above, and therefore $p^\pm \in K_e$. But if $K_e$ is a single point
then $p^+ = p^-$ and $Z^\pm$ intersect, contrary to the definition of zippers.
Second of all, if every $K_e$ is a single point, then we can define $f:S^1_\un \to S^2_\infty$
by $f(e) = K_e$. The shrinking families of disks $D_i$ associated to each $K_e$ certify
that this map is continuous; by construction the point preimages are precisely
the decomposition elements associated to $\LL^\pm$ (the use of shrinking point neighborhoods
associated to leaves of the stable/unstable lamination to prove continuity is
the heart of the classical Cannon--Thurston argument for hyperbolic surface
bundles; see \cite{Cannon_Thurston}). This completes the proof, modulo the
fact that every $K_e$ consists of a single point.

The argument to show that each $K_e$ consists of a single point is surprisingly delicate;
it would be much easier if we could assume a rather plausible general conjecture
about the action of $G$ on $S^2_\infty$ that we call `no wandering continua'
(Conjecture~\ref{conjecture:no_wandering_continua}) but for now this general statement
seems out of reach. Let us therefore assume that some $K_e$ consists of more than one
point, and derive a contradiction.

Let us first establish some basic properties of $K_e$. 
The first property is that $K_e$ can intersect $Z^\pm$ only in the support of
$e^\pm$ (if either or both of these are ideal gaps). For if $x\in K_e \cap Z^+$ (for instance)
we can construct the convex hull $\delta$ in $Z^+$ of $x$ with the support/end of $e^+$ and
observe that this entire (nontrivial!) segment $\delta$ is contained in all of $K_e$. But by
hairiness, there are many hairs on $\delta$ on either side separating $x$ from $K_e$
after all. This establishes the first property.
Let's refer to these points of $K_e \cap Z^\pm$ (if any) as the `endpoints' of $K_e$.

It follows that for the shrinking disks $D_i$ in the definition of $K_e$ that the
intersection of the {\em interiors} of the $D_i$ contains all of $K_e$ except possibly
for the endpoints. For, by construction, each $\partial D_i$ is
entirely contained in $Z^+\cup Z^-$, and now apply the previous observation.

The second property is that for all $g \in G$ if $gK_e\cap \inte(D_i)$ is nonempty, 
then $gK_e \subset D_i$. To see this, note that if the interior of 
$gD_j$ intersects $\partial D_i$, this intersection divides $gD_j$ into two subdisks,
and (by the definition of the $D_i$ and of $K_e$) the set 
$gK_e$ must be contained in one of them. 
A related fact, true for essentially the same reason, is that the intersection
$K_e \cap gK_e$ is either all of $K_e$, which can happen if and only if $ge=e$;
or it is contained in the common endpoints of both sets, which can happen if and only
if these endpoints are distinct ideal points in either $Z^+$ or $Z^-$ with the same support.

Notice by the way, since points in $S^1_\un$ have at most cyclic stabilizer, that
if $gK_e = K_e$ then $g$ acts on $K_e$ properly discontinuously away from the fixed
points of $g$ in $S^2_\infty$. It follows that if $K_e$ consists of more than one point,
we may take the intersection of $K_e$
with a fundamental domain for $g$ and obtain a possibly smaller compact connected
subset $K'\subset S^2_\infty$, nontrivial if $K_e$ is nontrivial, and so that $K'$ is
{\em disjoint} from all its nontrivial $G$ translates. We strongly suspect that this
already gives a contradiction (see Conjecture~\ref{conjecture:no_wandering_continua})
but this seems very subtle if true, and we do not pursue it further here in this generality.

The third property is that the set $K_e$ is {\em porous}. Because $G$ is cocompact, this
may be stated as the property that there is no sequence of translates $g_iK_e$ whose
Hausdorff limit is all of $S^2_\infty$. For, if this were true, some translate
$g_iK_e$ would intersect some $\inte(D_j)$ without being contained in $D_j$, contrary to
the second property. 

We now need an elementary geometric lemma:
\begin{lemma}[Porous multiple components]\label{lemma:porous_multiple}
Suppose $L\subset \CP^1$ is a nontrivial connected porous closed set. Then
there are a sequence of M\"obius transformations $g_i$ so that the Hausdorff
limit $L_\infty: = \lim g_iL$ exists, and for which the
complement $\CP^1 - L_\infty$ has at least two distinct (open) components. 
\end{lemma}
\begin{proof}
To prove this lemma, we may assume first of all that $U:=\CP^1 - L$ is connected, 
or else there is nothing to prove, and since $L$ is connected, 
$U$ is homeomorphic to an open disk. After a conformal change of coordinates, 
let's suppose $L \subset \C \subset \CP^1$ and that the intersection of $L$
with the real axis is contained in the interval $[-1,1]$ and contains the points
$\pm 1$. Let $\delta$ be the arc $\RP^1 - [-1,1]$ which is properly embedded in $U$, 
and divides it into two open subdisks, which we denote $U^\pm$
respectively. 

The property that $L$ is porous is equivalent to the fact that there is a 
positive constant $C_L>0$ so that every round disk in $\C$ of radius $r$ contains a 
round disk of radius $C_L\cdot r$ in $\C - L$. For every round disk $D$ of radius $r$ let
$C^\pm(D) \in [0,1]$ be the biggest numbers such that $D$ contains an open round disk
of radius $C^+(D)\cdot r$ in $D\cap U^+$ and an open round disk of radius 
$C^-(D)\cdot r$ in $D\cap U^-$. The pair of numbers $C^\pm(D)$ are continuous functions of
the disk $D$, and $\max(C^+(D),C^-(D))\ge C_L$ for all $D$. For $D$ entirely contained
in $U^+$ we have $C^+(D)=1$ and $C^-(D)=0$, and likewise for $D$ entirely contained in $U^-$.
It follows by the intermediate value theorem that we may find disks $D_r$ of any small radius $r$
whose center is far from $\delta$, and such that $C^+(D_r)=C^-(D_r) \ge C_L$. 
Now take a sequence of such disks with $r_i \to 0$ and rescale $L$ centered at these disks
so that the rescaled disks have center $0$ and radius $1$. A Hausdorff limit of any subsequence 
of these blowups necessarily has at least two complementary components, proving the lemma.
\end{proof}

Finally, let's apply Lemma~\ref{lemma:porous_multiple} to $K_e$ and let
$\delta$ be as in the proof of the lemma. 
Since $K_e$ is porous, we can find a sequence of elements $g_i\in G$ so that 
the Hausdorff limit $K_\infty:=g_iK_e$ 
has at least two distinct complementary components $U,V$. We may further choose $D_r$
as in the lemma whose center is far from the endpoints of $K_e$ (if any). It follows
that we may find a sequence of closed disks $E_i$ of diameter at most $1/i$ that
contain both $g_i\delta$ and the endpoints of $g_iK_e$ (if any). 

Choose points $a^\pm \in Z^\pm \cap U$ and $b^\pm \in Z^\pm \cap V$
and let $\alpha^\pm\subset Z^\pm$ be the unique embedded arcs joining $a^\pm$ to $b^\pm$.
Each $\alpha^\pm$ is disjoint from $g_iK_e$ except possibly at its endpoints. On the
other hand, $E_i$ separates $a^+$ from $b^+$ and $a^-$ from $b^-$ in $S^2_\infty -
g_i K_e$, when $i$ is large. It follows that $\alpha^\pm$ both intersect $E_i$.

After passing to a subsequence, we may assume $E_i$ converges to a single point 
$c \in K_\infty$. Since $\alpha^\pm$ are both closed arcs and hence compact, 
$\alpha^\pm$ must both contain $c$. But this means that $Z^+$ and $Z^-$ intersect, 
which is a contradiction after all.

This completes the proof of the theorem.
\end{proof}

\begin{theorem}[Minimal $G$-zippers are hairy]\label{theorem:G_zippers_are_hairy}
Let $G$ be the fundamental group of a closed hyperbolic 3-manifold $M$. 
Suppose that $Z^\pm \subset S^2_\infty$ is a minimal $G$-zipper. Then $Z^\pm$ is hairy.
\end{theorem}
\begin{proof}
We suppose not, so that there is some open interval $J$ in $Z^+$ which is not hairy
on at least one side. Choose a proper closed interval $K \subset J$.

\begin{lemma}[Porous]\label{lemma:K_is_porous}
The arc $K$ is porous.
\end{lemma}
\begin{proof}
We first construct a Jordan curve $X$ in $S^2$
which is the union of an open arc $\alpha \subset Z^+$, an open arc $\beta \subset Z^-$,
and two points $p,q$ neither of which are in $Z^+$ (they might be in $Z^-$ or not).
The construction of $X$ depends on some known facts about the dynamics of $G$ on a zipper,
proved in \cite{Calegari_Loukidou_Zippers} and \cite{Kim}. 

If there is some element $g\in G$ that acts on $Z^+$ without fixed points, it has an
axis in $Z^+$ that lands at its endpoints $p^\pm \subset S^2 - Z^+$ by
\cite{Calegari_Loukidou_Zippers} Lemma~2.23. Either one of these
endpoints is already in $Z^-$, or there is an axis in $Z^-$ that also lands at $p^\pm$.
In either case we can find a closed embedded interval $\gamma:[0,1] \to  S^2$
so that $\gamma\, [0,1/2) \subset Z^+$ and $\gamma\, (1/2,1] \subset Z^-$ while
$\gamma(1/2) \in S^2 - Z^+$. If $h \in G$ is not commensurable with $g$ then we
can find $X$ as a subset of the convex hulls of $\gamma(I) \cup h\gamma(I)$.

Kim \cite{Kim}, Theorem~A says that no nontrivial element $g\in G$ can have both
fixed points in $Z^+$ or both in $Z^-$. So let's suppose every nontrivial element
has exactly one fixed point in $Z^+$. Then  by \cite{Calegari_Loukidou_Zippers} Lemma~2.24
there is some nontrivial $g\in G$ fixing $p^+\in Z^+$ and leaving invariant some
proper ray $r\subset Z^+$ that lands on the other fixed point $p^-\in S^2-Z^+$.
Either $p^-$ is in $Z^-$, or $g$ has an axis in $Z^-$ which has an end that lands at $p^-$.
In either case we get a closed embedded interval $\gamma$ as above, and obtain $X$.

Now we show that $K$ is porous. Suppose not. The complement $S^2 - X$ consists of two
open disks, $U$ and $V$. If $K$ is not porous, there is some translate $gK$ which
intersects both $U$ and $V$, and therefore $gK$ must intersect the {\em open} arc
$\alpha$, necessarily in a closed connected segment. But then the complementary
arcs of $\alpha$ are hairs on either side of $gK$, contrary to hypothesis.
\end{proof}

The end of the argument closely follows the end of the proof of
Theorem~\ref{theorem:G_zipper_to_G_wheel}. First of all, by Lemma~\ref{lemma:porous_multiple}
we may choose a Jordan arc $X:=K \cup K'$ for some compact arc $K'$ meeting $K$ only
at its endpoints, and whose complement is the disjoint union of open disks $U$ and $V$,
and a sequence $g_i \in G$ which zooms in near some point in the interior of $K$, and
so that $g_i K \to K_\infty$ where $K_\infty$ has at least two complementary open regions
$U_\infty$ and $V_\infty$, and where points in $U_\infty$ are limits of points
in $g_i U$ and points in $V_\infty$ are limits of points in $g_i V$.
By the choice of $g_i$ the spherical diameters of $g_i K'$ must converge to $0$, and
after passing to a subsequence we can assume there is a specific point $x\in S^2$ with
$g_i K' \to x$.

Choose points $p^\pm \in U_\infty \cap Z^\pm$ and $q^\pm \in V_\infty \cap Z^\pm$. 
Note that $g_i X$ separates $p^\pm$ from $q^\pm$ when $i$ is big.

There are paths $\alpha \subset Z^+$ from $p^+$ to $q^+$ and $\beta \subset Z-$ 
from $p^-$ to $q^-$. Evidently 
$\beta$ is disjoint from every $g_i K$ and therefore $\beta$ intersects every $g_iK'$ so that 
in fact $\beta$ contains $x$. On the other hand, we claim $\alpha$ must also intersect every 
$g_i K'$. For otherwise, if $\alpha$ intersects $g_i K$ but not $g_i K'$ its intersection
must lie entirely in the interior of $g_i K$ so that $\alpha$ would give a hair 
on both sides of $g_iK$, contrary to hypothesis. But if $\alpha$ intersects every $g_i K'$ 
then $\alpha$ must contain $x$. So $\alpha$ and $\beta$ both contain $x$ and 
therefore $Z^+$ and $Z^-$ intersect, which is a contradiction. This final contradiction 
proves the theorem.
\end{proof}

\begin{corollary}[Properties of $G$-zippers]\label{corollary:G_zipper_properties}
Let $Z^\pm$ be a minimal $G$-zipper. Then
\begin{enumerate}
\item{$Z^\pm$ is hairy;}
\item{$Z^\pm$ satisfies the strong landing property;}
\item{$Z^\pm$ have {\em short hair}: for every $\epsilon > 0$ there is a finite simplicial
tree $T \subset Z^+$ so that every (path) component of $Z^+-T$ has spherical diameter at
most $\epsilon$;}
\item{there is an $n$ so that every point of $Z^\pm$ has valence at most $n$;}
\item{there is a $K$ so that every embedded path in $Z^\pm$ is a $K$-quasiarc; and}
\item{the Hausdorff dimension of each of $Z^\pm$ is strictly less than 2.}
\end{enumerate}
\end{corollary}
\begin{proof}
The first bullet is Theorem~\ref{theorem:G_zippers_are_hairy}.
Theorem~\ref{theorem:G_zipper_to_G_wheel} says that every 
hairy minimal $G$-zipper $Z^\pm$ arises from a $G$-CaTherine wheel $f:S^1 \to \CP^1$, and therefore
satisfies the strong landing property, proving the second bullet. 
Any zipper arising from a CaTherine wheel has short hair (\cite{Calegari_Gwynne} 
Proposition~1.4); this proves the third bullet. 

A $G$-CaTherine wheel is also a
$K$-CaTherine wheel for some $K$ by Proposition~\ref{proposition:uniform_K}.
The bound $n$ is then the same bound as in Proposition~\ref{proposition:K_bounds_valence},
proving the fourth bullet. Every embedded path in $Z^\pm$ is contained in $\partial f(I)$
for some closed interval $I\subset S^1$, proving the fifth bullet.
Finally, each of $Z^\pm$ is a countable union of embedded paths; each of these
has Hausdorff dimension at most $2K/(K+1)$ by the fifth bullet and Astala's estimate, and
therefore the same is true for each of $Z^\pm$; proving the last bullet.
\end{proof}

We end this subsection with a conjecture and a question: 

\begin{conjecture}[No wandering continua]\label{conjecture:no_wandering_continua}
Let $G$ be a cocompact Kleinian group. Suppose $K\subset S^2_\infty$ is a compact connected
set disjoint from all its $G$-translates. Then $K$ consists of a single point.
\end{conjecture}

We remark that any $K$ which is a counterexample to this conjecture cannot contain an
arc of a round circle. This depends on a theorem of Zeghib \cite{Zeghib} that
compact hyperbolic manifolds cannot contain nontrivial totally geodesic laminations
of dimension $\ge 2$ with some noncompact leaves; a refinement of this argument,
explained to us by Curt McMullen \cite{McMullen}, implies that $K$ cannot contain
a quasi-arc.

\begin{question}[No wandering continua]\label{question:no_wandering_continua}
Let $G$ be a word-hyperbolic group with Gromov boundary $\partial_\infty G$
homeomorphic to $S^2$. Is it possible for $\partial_\infty G$ to contain
a nontrivial compact connected subset $K$ disjoint from all its $G$-translates?
\end{question}
Of course, if we believed the Cannon Conjecture, we would expect the answer to
this question to be `no'. On the other hand, maybe the Cannon Conjecture is false, and
this question has a positive answer after all. It's best to keep an open mind.

\subsection{The coarse viewpoint}\label{subsection:coarse_correspondence}

At this point we have estabished an equivalence between pseudo-Anosov flows without
perfect fits on $M$ up to orbit equivalence, $G$-CaTherine wheels up to conjugacy,
and minimal $G$-zippers. In fact, we can add a fourth element to this correspondence.

\begin{definition}[Uniform quasimorphism]\label{definition:uniform_quasimorphism}
Let $G$ be a group. A {\em quasimorphism} is a function $\phi:G \to \R$ for which
there is some least non-negative real number $D(\phi)$ (called the {\em defect})
so that for all $g,h \in G$ there is an inequality
$$|\phi(gh) - \phi(g) - \phi(h)| \le D(\phi)$$
If $G$ is finitely generated, and we fix any word metric $d_G$, a quasimorphism on $G$
is {\em uniform} if it is unbounded, and if the coarse level sets are coarsely connected.
This means there are constants $C_1,C_2>0$ so that if $g,h \in G$ are any two elements
with $\phi(g),\phi(h) \in [-C_1,C_1]$ then there is a sequence of elements
$g=g_0,g_1,\cdots,g_n=h$ so that $\phi(g_i) \in [-C_1,C_1]$ for all $i$,
and $d_G(g_i,g_{i+1})\le C_2$ for all $i$.
\end{definition}

Any bounded function on a group is a quasimorphism; quasimorphisms on $G$ modulo bounded
functions form a vector space $Q(G)$ which admits a natural norm with respect to which
it is a (typically non-separated) Banach space.
For an introduction to the theory of quasimorphisms, see \cite{Calegari_scl}.

Uniform quasimorphisms were introduced by Calegari--Loukidou \cite{Calegari_Loukidou_Zippers}
(although the underlying concept has been familiar to geometric group theorists for some time)
and it is shown there (\cite{Calegari_Loukidou_Zippers} Theorem~4.10)
that if a hyperbolic group $G$ admits a uniform quasimorphism then it admits a $G$-zipper
(in this paper we are only concerned with the case that the ideal
boundary of $G$ is homeomorphic to $S^2$, but a more general definition makes sense
for any hyperbolic group).

Conversely, (Calegari--Zung \cite{Calegari_Zung} Theorem~2.10 part 1)
show that if $X$ is any pseudo-Anosov
flow on $M$ without perfect fits, one can construct uniform quasimorphisms $\phi:G \to \R$
that are {\em adapted to $X$}. This means that the values of $\phi$ grow linearly
along any quasigeodesic in $G$ that tracks (i.e.\/ stays a bounded distance from) a 
flowline of $\tilde{X}$, the lift of $X$ to the universal cover.

Associated to $X$, let $Q_X(G)$ denote the space of uniform quasimorphisms on $G$ adapted
to $X$. Here we think of $Q_X(G)$ as a subset of $Q(G)$. It turns out 
(Calegari--Zung \cite{Calegari_Zung} Theorem~2.10 part 2) that $Q_X(G)$ is
an open convex cone in $Q(G)$. On the other hand, by combining 
\cite{Calegari_Loukidou_Zippers} Theorem~4.10 with Theorem~\ref{theorem:G_zippers_are_hairy}, 
Theorem~\ref{theorem:G_zipper_to_G_wheel} and Theorem~\ref{theorem:wheels_and_flows}
every uniform quasimorphism $\phi$ on $G$ gives rise to a pseudo-Anosov flow $X$ without
perfect fits, and it is immediate from the construction that $\phi$ is adapted to $X$.

We therefore obtain the second main result of this paper:

\begin{theorem}[$G$-Equivalence Theorem]\label{theorem:G_equivalence}
For $M$ a fixed closed hyperbolic 3-manifold with fundamental group $G$ acting on $\CP^1$
there are canonical bijections between the following structures:
\begin{enumerate}
\item{pseudo-Anosov flows without perfect fits on $M$ up to orbit equivalence;}
\item{$G$-CaTherine wheels up to conjugacy;}
\item{(minimal) $G$-zippers; and }
\item{connected components in $Q(G)$ of the space of uniform quasimorphisms $\phi:G\to \R$.}
\end{enumerate}
\end{theorem}

In particular, the set of uniform quasimorphisms on $G$ is the union of finitely many
open convex cones; this is a coarse analog of part of Thurston's theory of fibrations
\cite{Thurston_norm}, that may be paraphrased as saying that the subset of $H^1(M;\R)$ 
represented by closed nowhere zero 1-forms is the union of finitely many open convex cones 
(which correspond precisely to the open fibered faces of the Thurston norm ball 
on $H_2(M;\R) = H^1(M;\R)$).

Let us end this subsection with a conjecture.
Uniform quasimorphisms make sense on any hyperbolic group, and in fact
\cite{Calegari_Loukidou_Zippers} Theorem~4.10 constructs a $G$-zipper 
$Z^\pm \subset \partial_\infty G$ for any
hyperbolic group $G$ from a uniform quasimorphism $\phi$ on $G$.

\begin{conjecture}[Zippers to flows]\label{conjecture:uniform_quasimorphism}
Let $G$ be a hyperbolic group with $\partial_\infty G$ homeomorphic to $S^2$, and
suppose there is a minimal $G$-zipper $Z^\pm \subset \partial_\infty G$.
Let $P \subset Z^+\times Z^-$ be the closure of the set of pairs of points
$p^+\times p^-$ where $p^\pm$ are the fixed points of any infinite order $g\in G$ that
fixes exactly one point in each of $Z^+$ and $Z^-$. Then $P$ is homeomorphic to $\R^2$.
\end{conjecture}

If this conjecture is true for some $G$ and some $G$-zipper $Z^\pm$, 
we may construct $\tilde{N}:=P \times \R$
in the flow space $\F(G)$ and deduce that $G$ is virtually a 3-manifold group by
the argument of Theorem~\ref{theorem:wheel_Cannon_conjecture}.

\section{Endomorphisms}

This section is mostly expository, and is concerned almost entirely with 
examples. An {\em endomorphism} of a CaTherine wheel $f:S^1 \to S^2$ is a pair of
endomorphisms $h:S^1 \to S^1$ and $H:S^2 \to S^2$ for which $Hf = fh$. There is
probably no hope for a full classification, and even if it were possible,
such a classification would leave
out some important and closely related examples (semigroups or pseudogroups of partially
defined endomorphisms of $S^1$ and $S^2$ respectively).

An important class of examples arise from {\em expanding Thurston maps}, whose 
theory is developed in great detail by Bonk--Meyer \cite{Bonk_Meyer}. For these examples,
$h$ is (conjugate to) the covering map $z \to z^d$ on $S^1$ for some $d>1$
and $H$ is a degree $d$ branched covering of $S^2$ (for the same $d$),
which in many cases can be taken to be a degree $d$ rational map on $\CP^1$. We
emphasize that it is only {\em some} of these examples that give rise to
CaTherine wheels; many (the majority) give rise to Peano curves that should be
P-CaTherine wheels in the sense of Definition~\ref{definition:P_CaTherine_wheel}.
The difference has to do with the absence of perfect fits; in the language of matings,
the ray equivalence classes should have diameter zero.

A closely related, though hitherto unexplored (as far as we know) class of examples
are those that arise from {\em expanding origamis}. For these examples the map $h$
is piecewise linear on $S^1$ with slopes of constant absolute
value $\lambda > 1$, whereas the map $H$ is piecewise conformal (alternately holomorphic
and anti-holomorphic) on a subdivision of $\CP^1$ into finitely many
tiles. This is a subject that deserves a more thorough treatment, that we plan
to develop in future projects (e.g.\/ \cite{Calegari_Herr}). 

Further afield, there are many examples of CaTherine wheels and related objects
that admit nontrivial families of partially defined symmetries --- Bowen--Series Markov maps
\cite{Bowen_Series}, Dicks' lightning paths \cite{Dicks_Cannon_I,Dicks_Cannon_II,Dicks_Wright},
Arnoux' IET suspensions \cite{Arnoux} --- even Hilbert's square-filling curve is of
this kind (\cite{Calegari_foliations} Example~10.14). This subject is too big to
survey here.

\subsection{Expanding Thurston maps}\label{subsection:expanding_Thurston}

In this subsection we discuss CaTherine wheels $f:S^1 \to S^2$ admitting endomorphisms
$h,H$ for which $h:S^1 \to S^1$ is conjugate to $z \to z^d$ and $H:S^2 \to S^2$ is a
degree $d$ branched cover.

Bonk--Meyer's approach is to start with the branched cover $H:S^2 \to S^2$; our 
account of their theory starts with the map $z \to z^d$ on $S^1$ and directly construct
invariant laminations and laminar relations. These laminations are well-studied, and
their analysis is rather subtle; see e.g.\/ \cite{Thurston_poly, Kiwi, Blokh_Levin,
Blokh_Oversteegen, Thurston_degree_d}. 

When is a degree $d$ covering map $S^1 \to S^1$ conjugate to $z \to z^d$? When it is
locally eventually onto:

\begin{lemma}[LEO]\label{lemma:LEO}
A degree $d>1$ covering map $h:S^1 \to S^1$ is conjugate to $z \to z^d$ for some $d>1$ 
if and only if it is LEO (locally eventually onto): for any open interval $U \subset S^1$ 
there is an $n$ such that $f^n(U) = S^1$.
\end{lemma}
\begin{proof}
This condition is evidently necessary. Conversely, it we start with a fixed point $p$ and
form the increasing union of preimages of $p$, this sequence of subsets converges to
all of $S^1$ and defines the coordinates of the obvious conjugacy.
\end{proof}

\subsubsection{Invariant laminations}

Let's fix an integer $d$ and discuss laminations. It is convenient to think of the
circle $S^1$ as $\R/\Z$ and the map $h$ as $x \to d\cdot x \pmod \Z$.

\begin{definition}[Major]\label{definition:major}
Fix $d$. A (degree $d$) {\em critical leaf of multiplicity $n$} is an unordered
$(n+1)$-tuple of distinct points $\lbrace x_0,\cdots, x_n\rbrace$ in $S^1$ 
for which $d(x_i-x_j)=0\pmod \Z$ for all $i,j$. 

A (degree $d$) {\em major} is a collection $C$ of disjoint (degree $d$)
critical leaves which are pairwise unlinked and have total multiplicity $d-1$.

A critical leaf is simple if it has multiplicity $1$, and a major is simple if
all its critical leaves are simple.
\end{definition}
A non-simple critical leaf of multiplicity $n$ may be thought of as a union of 
$n$ (ordinary) leaves, and a major may be thought of in this way as a finite lamination.
Equivalently, a major may be thought of as the nontrivial elements in an equivalence
relation on $S^1$.

Now let's discuss dynamics. If $\lambda:=\lbrace x,y\rbrace$ is a non-critical leaf, we 
define $h\lambda$ to be the leaf $\lbrace d\cdot x,d\cdot y\rbrace$, and otherwise 
we define $h\lambda$ to be empty. The {\em grand orbit} of a leaf $\lambda$
is the set of leaves $\mu$ such that either $h^n(\lambda)=\mu$ for some $n$, or
$h^n(\mu) = \lambda$ for some $n$.

\begin{definition}[Invariant lamination]\label{definition:invariant_lamination}
A lamination $\Lambda$ of $S^1$ is said to be {\em invariant} if
\begin{enumerate}
\item{for every non-critical leaf $\lambda$, the leaf $h\lambda$ is a leaf of $\Lambda$; and}
\item{for every leaf $\lambda$ there are exactly $d$ leaves $\mu$ of $\Lambda$ with 
$h\mu=\lambda$, and these leaves are all disjoint.}
\end{enumerate}
\end{definition}
The condition that preimages are disjoint implies that if $\Lambda$ is an invariant
lamination and $\lambda:=\lbrace x,y\rbrace$ is a leaf of $\Lambda$ then for every
$x' \in S^1$ with $hx' = x$ there is a unique $y' \in S^1$ with $hy' = y$ for which
$\lbrace x',y'\rbrace$ is a leaf of $\Lambda$.

\begin{proposition}[Invariant lamination]\label{proposition:invariant_lamination}
For any $d>1$ and a degree $d$ major $C$ there is an invariant lamination
$\Lambda$ containing $C$.

Furthermore, suppose there are no points $x,y \in S^1$ (not necessarily distinct)
contained in leaves of $C$ and an $n>0$ so that $h^n(x)=y$. Then $\Lambda$ is
unique.
\end{proposition}
\begin{proof}
For a proof see e.g.\/ \cite{Thurston_degree_d} Theorem~4.3 or 
\cite{Calegari_sausages}, Proposition~5.3 (in the very slightly different context of
invariant elaminations). For the sake of the reader we give a sketch of the construction.

The idea is very simple: the result of quotienting
$S^1$ by the equivalence relation $\sim$ associated to the major $C$ is a cactus, i.e.\/ a
tree of $d$ smaller circles $C_i$, for $i=0,\cdots,d-1$, each of perimeter $1/d$. 
The map $h$ factors through $\sim$ to give a map $h: C/\sim \; \to C$ which takes each 
$C_i$ to $C$ homothetically, scaling lengths by $d$. Define $\Lambda_0:=C$ and
inductively define $\Lambda_i$ by pulling back $\Lambda_{i-1}$ by $h$ to each $C_i$,
and then taking the union of the preimages in $C$. Then take $\Lambda$ to be the closure of
the increasing union $\cup_i \Lambda_i$.

The step that takes a collection of laminations on the $C_i$ and forms the
preimage in $C$ is ambiguous when one of the laminations has endpoints in $C_i$ at the
image of leaves of $C$ (which become points in $C/\sim$); one may resolve this
ambiguity by always making a leftmost choice of preimage at every stage. 
By construction, this ambiguity arises when, and only when, there are points 
$x,y\in S^1$ contained in leaves of $C$ (not necessarily distinct) and an $n>0$ 
so that $h^n(x) = y$.
\end{proof}

It should be emphasized that Proposition~\ref{proposition:invariant_lamination} is
entirely constructive, and invariant laminations (or: finite approximations to them)
can be computed efficiently. When endpoints of the majors are rational, the
leaves of the invariant lamination (expressed as unordered pairs of right-infinite
strings in their base $d$ expansion) are parameterized by infinite walks in an
explicit finite digraph, which may be thought of as the analog of the
(projectively weighted) invariant stable traintrack for a pseudo-Anosov automorphism of
a closed surface.

If $C$ is a degree $d$ major with invariant lamination $\Lambda$, then  
we can form the big invariant relation $\Rel(\Lambda)$. It is possible for an invariant
lamination $\Lambda$ to contain only countably many leaves, so that $\Rel(\Lambda)$ is
the complete equivalence class and $\Lam(\Rel(\Lambda))$ is empty!

\begin{example}[Countable lamination]\label{example:countable}
Set $d=2$ and let $C$ consist of the single leaf $\lbrace 0,1/2\rbrace$. Then there is
an invariant lamination with $2^n$ leaves of length $2^{-1-n}$ and no limit leaves at all.
See Figure~\ref{countable_invariant_lamination}.

\begin{figure}[htpb]
\centering
\includegraphics[scale=0.5]{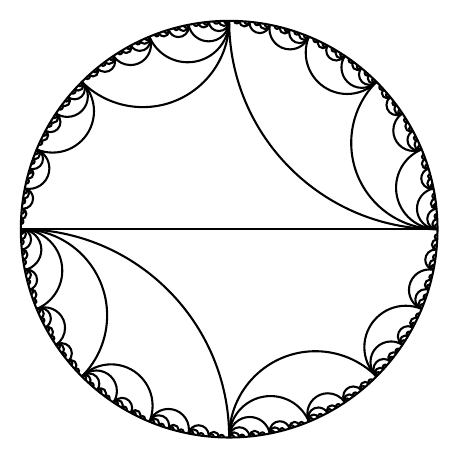}
\caption{A countable degree 2 invariant lamination with critical leaf $\lbrace 0,1/2\rbrace$.}
\label{countable_invariant_lamination}
\end{figure}
\end{example}

However, there are many situations where $\Lambda$ has only finitely many grand orbits of
isolated leaves, and $\Lam(\Rel(\Lambda))$ has no isolated leaves, and only finitely
many grand orbits of (finite sided) complementary polygonal regions.

\begin{example}[Misiurewicz point]\label{example:Misiurewicz_point}
A {\em Misiurewicz point} $c \in \C$ is a complex number for which the critical point
$0$ is preperiodic but not periodic under the associated polynomial map $H_c:z \to z^2 + c$.
Misiurewicz points are contained in the Mandelbrot set, and their Julia set $J_c$ is a dendrite.
See e.g.\/ \cite{Douady_Hubbard}.

There is a surjective map $f:S^1 \to J_c$ whose equivalence classes are $\Rel(\Lambda)$
where $\Lambda$ is the invariant degree 2 lamination of $S^1$ associated to the critical
major $C$ consisting of a single leaf $\ell:=\lbrace \theta, \theta + \frac{1}{2}\rbrace$ for some
$\theta \in \Q/\Z$ disjoint from its forward iterates. The leaf $\ell$ maps to $0$.

For example, take $c=i$, so that $\theta = 1/12$. The forward orbit
$$0 \to i \to -1+i \leftrightarrows -i$$
is preperiodic but not periodic. It corresponds to the dynamics on $S^1$:
$$\lbrace 1/12,7/12\rbrace \to 1/6 \to 1/3 \leftrightarrows 2/3$$
See Figure~\ref{i_example}.

\begin{figure}[htpb]
\centering
\includegraphics[scale=0.5]{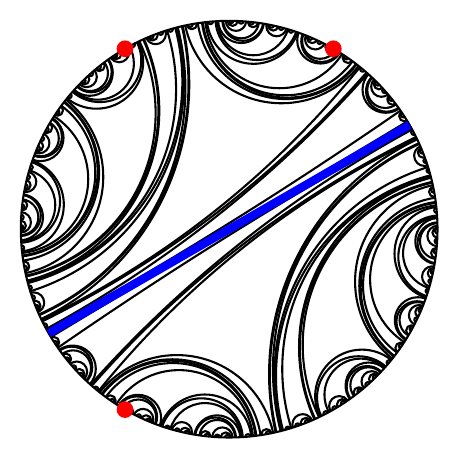}\quad\quad\quad\quad
\includegraphics[scale=0.423]{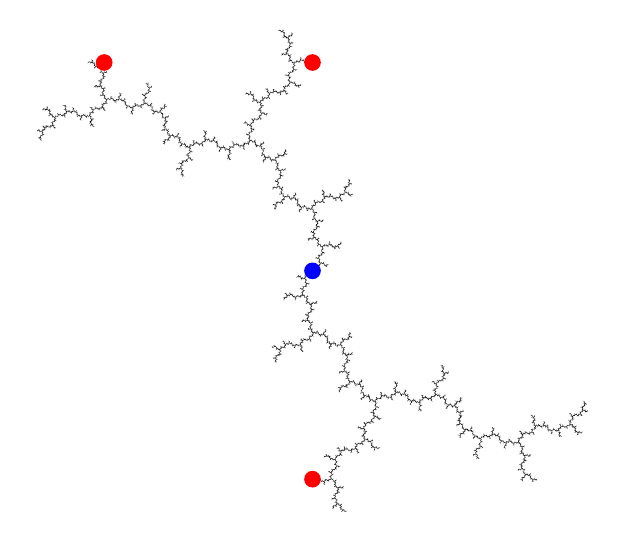}
\caption{The Julia set for $H_i:z \to z^2 + i$ is the quotient of $S^1$ by the
laminar relation associated to the degree 2 major $\lbrace 1/12,7/12\rbrace$.}
\label{i_example}
\end{figure}
\end{example}

\subsubsection{Mating}\label{subsubsection:mating}

Mating of (connected polynomial) Julia sets was first described by Douady and worked
out in detail by Tan Lei, Rees, Shishikura, Ch\'eritat
and others; see \cite{Tan_Lei_mate,Rees_mate,Shishikura_mate,Cheritat}.
Milnor \cite{Milnor_worked} worked out an explicit example of the mating of
two (degree 2) dendrite Julia sets both associated to $z \to z^2 + c$ for
$c \approx -0.228 + 1.115i$, the unique root of $4c^3 + 4c^2 + 1$ with a positive
imaginary part, and one can hardly do better than read his paper.

Let $H^\pm: \CP^1 \to \CP^1$ be degree $d$ polynomials with connected filled Julia
sets $K^\pm$. Assuming that $\partial K^\pm$ are locally connected,
there are surjective maps $f^\pm:S^1 \to \partial K^\pm$ 
conjugating $h:z \to z^d$ on $S^1$ to $H^\pm$ on $\partial K^\pm$. 
A {\em topological mating} is obtained as a quotient
$K:=K^+ \sqcup K^- /\sim$ where $p \in \partial K^+$ is identified with $q \in \partial K^-$
if there is $x\in S^1$ with $f^+(x)=p$ and $f^-(-x)=q$.
Evidently the maps $f^\pm$ paste together (after conjugating $f^-$ by the antipodal
map, which commutes with $z \to z^d$) to form $f:S^1 \to K$, and there is a 
degree $d$ endomorphism $H:K \to K$ agreeing with $H^\pm$ on $K^\pm$
respectively so that $fh = Hf$ as maps from $S^1$ to $K$.

When $K^\pm$ are dendrites, and $f^\pm$ are surjective and arise by quotienting
the circle by laminar relations $\LL^\pm$, then providing $\LL^\pm$ have no perfect
fits, the quotient $K$ is a sphere and $f:S^1 \to K$ is a CaTherine wheel.
More generally, whenever $K^\pm$ are dendrites and $K$ is a sphere, then $\LL^\pm$ 
form an especial pair and we conjecture that $f:S^1 \to K$ is a P-CaTherine wheel 
in the sense of \S~\ref{section:P_wheels}. In the terminology of mating, leaves of $\LL^\pm$
are called {\em rays}, and chains of leaves alternating between $\LL^\pm$ each with
one endpoint in common with its adjacent leaves define an equivalence relation on rays
called {\em ray equivalence}. The {\em diameter} of a ray equivalence class is one
less than the number of leaves in such a maximal chain. Much is known and
much is unknown about ray equivalence in general; see \cite{Jung} or \cite{Mating_Questions}, \S~3.
For instance, Epstein showed that whenever postcritically finite quadratic maps can be
geometrically mated, the diameter is always finite; see 
\cite{Shishikura_mate} or \cite{Petersen_Meyer} Proposition~4.12. 

If $K$ is a sphere one says $H:K \to K$ is obtained
by {\em topological mating}. When $H$ is conjugate to a (degree $d$) 
rational map on $\CP^1$ one says it is a {\em geometric mating}. One should beware
that the holomorphic conjugacy class of this rational map is not necessarily unique!

When $H^\pm$ are strictly postcritically quadratic polynomials (associated to
Misiurewicz points), Shishikura \cite{Shishikura_mate} showed there 
is a geometric mating if and only if 
$H^+$ and $H^-$ are in nonconjugate limbs of the Mandebrot set $\M$. This means the
following. For every $\lambda \in \C$ there is a unique $c(\lambda)$ for which
$H_{c(\lambda)}:z \to z^2 + c(\lambda)$ has a fixed point $z$ with 
$H'_{c(\lambda)}(z)=\lambda$, namely 
$$c(\lambda) = \lambda/2 - (\lambda/2)^2$$
For each nonzero rational number $p/q$ and each root of unity $\lambda = e^{2\pi i p/q}$
there is a unique component $\M(p/q)$ of $\M - c(\lambda)$ that is disjoint from
the main cardioid. For any $c \in \overline{\M}(p/q)$ the fixed point $z$ for $H_c$ has
multiplier of absolute value $\ge 1$, and is a cut point for the Julia set $J_c$
whose preimage in $S^1$ consists of $q$ distinct points $x_0,\cdots,x_{q-1}$ in cyclic
order on which the restriction of $h$ maps $x_i$ to $x_{i+p}$, indices taken mod $q$;
note that this determines the $x_i$ uniquely up to cyclic rotation. For example,
at $c=c(\lambda)$ the fixed point is rationally indifferent and is the
intersection of the closure of $q$ Leau domains.

In particular, whenever $H^\pm$ are on conjugate limbs of $\M$, the laminar relations
$\LL^\pm$ have an entire nontrivial equivalence class in common (the set 
$\lbrace x_i\rbrace$) and they do not form an especial pair. 

\begin{example}[Conjugate mating]\label{example:conjugate_mating}
Let $\LL^\pm$ be the degree 2 invariant laminar relations associated to the
majors $\lbrace 1/12,7/12\rbrace$ and $\lbrace 5/12,11/12\rbrace$ respectively.
The laminar relations $\LL^\pm$ have no perfect fits. The reader with good eyesight
can verify this informally by zooming in on Figure~\ref{i_pair} and observing that every 
blue ray ends on a red rainbow, and conversely. This can be certified rigorously by
restricting attention to a fundamental domain for the $h$ action (extended to the
germ of a collar neighborhood of the boundary), and making suitable numerical 
estimates.

\begin{figure}[htpb]
\centering
\includegraphics[scale=0.5]{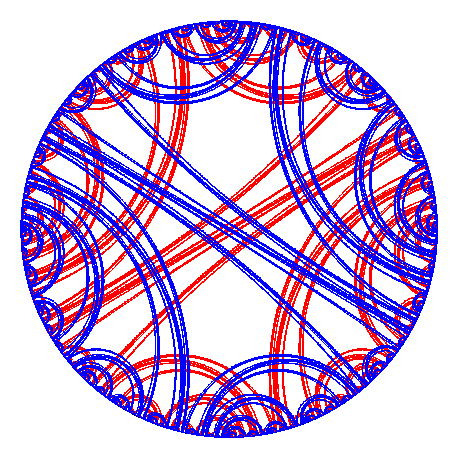}
\caption{The laminar relations $\LL^\pm$ have no perfect fits.}
\label{i_pair}
\end{figure}

The geometric mating is a degree 2 rational map $H:\CP^1 \to \CP^1$ whose Julia set
is all of $\CP^1$, and the map $f:S^1 \to \CP^1$ is a CaTherine wheel admitting a
degree 2 endomorphism. This is a Latt\`es example, derived from complex multiplication
on the Euclidean torus $E:=\C/\langle 1, \eta \rangle$
where one has
$$ \eta = \frac {1+i\sqrt{7}}{2} \text{ and } H(z) = \frac {\eta^2 z + 1} {z(z+\eta^2)}$$
(see \cite{Milnor_worked} \S~B.4). 

A fundamental domain for the CaTherine wheel $f:S^1 \to \CP^1$ lifts to
$\tilde{f}:I \to \C$, where its image covers half of a fundamental domain for the
lattice $\langle 1, \eta \rangle$. This image is a Jordan disk, and may be obtained
as the limit of an iterated piecewise-linear subdivision rule.
See Figure~\ref{disk_on_torus} and compare
with \cite{Milnor_worked}, Figure~7.

\begin{figure}[htpb]
\centering
\includegraphics[scale=1.5]{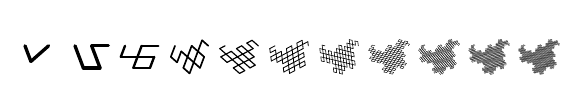}
\caption{A sequence of approximations to $\tilde{f}(I) \subset \C$}
\label{disk_on_torus}
\end{figure}
\end{example}

\subsubsection{Zippers}

If $\LL$ is an invariant laminar relation for a postcritically finite major $C$
there is a minimal nonempty finite forward-invariant subset $X \subset S^1$ containing at
least one point in the forward image of every leaf of $C$.
Let's suppose $\LL$ is the laminar relation associated to an 
equivariant parameterization $f:S^1 \to J \subset \CP^1$ of a 
dendritic Julia set associated to a rational map $H$. We may extend $\LL$ to a relation
in a hyperbolic plane $P$ bounded by $S^1$ and extend $f$ to $F:P \cup S^1 \to J$.
The image of the convex hull $C(X)$ of $X$ is a forward-invariant tree $T \subset J$
called the {\em Hubbard tree} (see \cite{Douady_Hubbard}, \S~IV), 
and the full Julia set $J$ may be obtained as the
closure of the union of iterated preimages of $T$.

If $\LL^\pm$ are a pair of invariant laminar relations arising from a CaTherine wheel
$f:S^1 \to S^2$ associated to the topological mating of two dendritic Julia sets, 
the interiors of the Hubbard trees $T^\pm \subset \CP^1$ are in the image of
$P^\pm$ and are therefore contained in $Z^\pm$.

\begin{definition}[Trimmed Hubbard trees]\label{definition:trimmed_trees}
The {\em trimmed} Hubbard trees $T^\pm_o \subset \CP^1$ are obtained from the Hubbard trees
$T^\pm$ by removing the 1-valent vertices.
\end{definition}

\begin{proposition}[Zipper is full preimage]\label{proposition:zipper_preimage}
Suppose a topological mating of dendritic Julia sets admits a CaTherine wheel
$f:S^1 \to S^2$. Then the unions of the iterated preimages of the trimmed Hubbard trees
$Z^\pm:= \cup_n H^{-n} T^\pm_o$ are the zipper for $f$.
\end{proposition}
\begin{proof}
Evidently $\cup_n H^{-n} T^\pm_o$ is contained in the zipper for $f$, so it suffices to
prove the converse. These unions are path-connected, since each successive preimage intersects
the preimages at the previous stage, and they are dense in $Z^\pm$ because the preimages
of the major (whose image is in $T^\pm_o$) are dense in $\LL^\pm$ by definition. 
Equality follows.
\end{proof}

For the Latt\`es Example~\ref{example:conjugate_mating} a lift to the torus of the
(positive) Hubbard tree is an arc that may 
itself be obtained as the limit of an iterated piecewise-linear subdivision rule.
In particular, the Hausdorff dimension is exactly $\log(\lambda)/\log(2)$ 
(approximately $1.2102$) where $\lambda$ is the biggest real root of $x^3 - 2x^2 + x - 4$.
Furthemore, one may directly verify (from the recursive construction) that 
this arc satisfies the bounded turning condition, and is therefore a quasiarc. 
See Figure~\ref{complex_multiplication_hubbard_arc}.

\begin{figure}[htpb]
\centering
\includegraphics[scale=1]{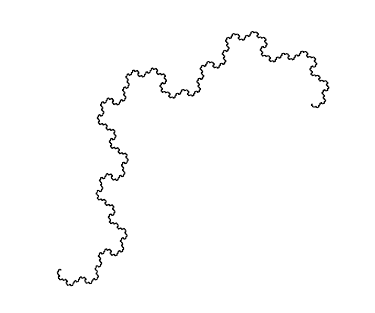}
\caption{The preimage of the Hubbard tree in $\C/\langle 1,\eta\rangle$ is a quasiarc.}
\label{complex_multiplication_hubbard_arc}
\end{figure}

The map from $\C/\langle 1,\eta\rangle$ to $\CP^1$ is branched over four points,
and there is one branch point in each of the half-zippers, namely the unique critical
point contained in either Hubbard tree. The preimage of each of the half-zippers
$Z^\pm$ in $\C$ is a disjoint union of topological $\R$-trees, each of which maps to $Z^\pm$ by a
double cover branched over the critical point. 

Figure~\ref{complex_multiplication_zipper} depicts a pair of these
lifts. Each is dense in a Jordan domain which is a fundamental domain for the torus
$\C/\langle 1,\eta\rangle$. These Jordan domains overlap even though
the lifted half-zippers are disjoint, and the overlap is itself a Jordan domain whose
boundary is the union of (lifts of) the two Hubbard trees.

\begin{figure}[htpb]
\centering
\includegraphics[scale=1]{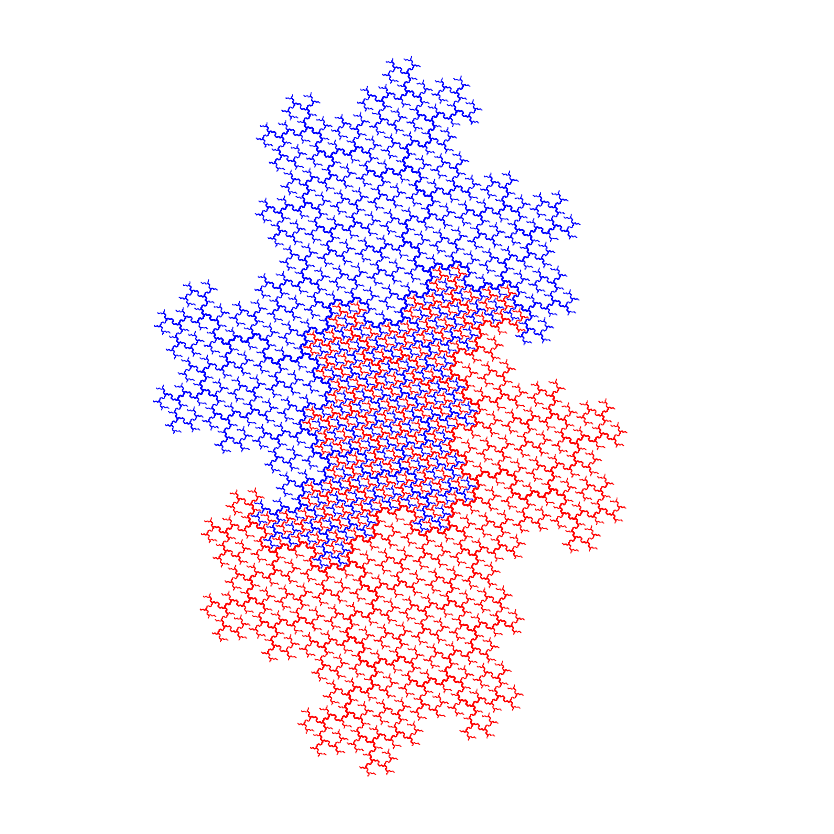}
\caption{The zipper is the union of the preimages of the trimmed Hubbard trees: $Z^\pm := \cup_n H^{-n} T^\pm_o$. 
Each half of the zipper is double covered by a topological $\R$-tree in $\C$, and these 
trees are in blue and red.}
\label{complex_multiplication_zipper}
\end{figure}

Based on this example, and in analogy with Proposition~\ref{proposition:uniform_K}, 
we make the following conjecture:

\begin{conjecture}[Quasicircles]\label{conjecture:quasicircles}
Let $f:S^1 \to \CP^1$ be a CaTherine wheel arising from an expanding Thurston map. 
Then $f$ is a $K$-CaTherine wheel for some $K$.
\end{conjecture}

Some evidence in favor of this conjecture comes from Lin--Rohde \cite{Lin_Rohde},
especially Theorem~1.5 and Proposition~2.2 which imply that for certain
quadratic dendritic Julia sets $J$ there is a $K$ so that every
embedded arc in $J$ is a $K$-quasiarc. These conditions can be expressed purely 
in terms of the associated lamination; we wonder whether a CaTherine wheel 
$f:S^1 \to \CP^1$ arising from a geometric mating of dendritic Julia sets, and
for which the associated laminar relations $\LL^\pm$ separately satisfy
the Lin--Rohde criterion, is a $K$-CaTherine wheel.

For expanding Thurston maps that do not arise from CaTherine wheels the construction
(and even the definition) of (P-)zipper is much more subtle; see \cite{AHLZ}.

\subsubsection{Expanding Thurston maps}

We are now in a position to make contact with Bonk--Meyer.

\begin{definition}
An orientation-preserving branched cover $H:S^2\to S^2$ is an {\em expanding Thurston map}
\begin{enumerate}
\item{if it is post-critically finite; i.e.\/ if the forward orbit of the set of critical
points is finite; and}
\item{if there is a Jordan curve $\Gamma$ containing the postcritical set so that the
mesh size of the graphs $\Gamma_n:=H^{-n}{\Gamma}$ goes to zero.}
\end{enumerate}
\end{definition}

With this definition, one has the following theorem of Meyer \cite{Meyer_unmate}:

\begin{theorem}[Meyer; Power is mating]\label{theorem:power_mating}
Let $H:S^2 \to S^2$ be an expanding Thurston map without periodic critical points. 
Then every sufficiently high iterate of $H$ is topologically conjugate to the 
topological mating of two polynomials.
\end{theorem}

We explain the idea of the proof in a simple case. Let $d$ be the degree of $H$,
denote the critical points by $C$ and the strictly forward orbit of $C$ by $P$.
Adjust $\Gamma$ by an isotopy if necessary so that the germ of $\Gamma$ through 
$P$ is invariant. We need one further combinatorial condition: there should
be an $H$-invariant Jordan curve $\tilde{\Gamma}$ (also containing the
postcritical set) and such that $\tilde{\Gamma}$ is isotopic to $\Gamma$ rel.
$P$. To achieve this one might need to pass to a power of $H$; that it can be
done at all is \cite{Bonk_Meyer} Theorem~15.1.

Let's suppose for simplicity the critical points are all simple (in particular
there are $2d-2$ of them) so that $\Gamma_1$ is a 4-valent graph with 
vertices at $C$. In the sequel we refer to a {\em 2-coloring of a graph} as a choice
of two colors (black and white) for the complementary regions, so that different
colors meet along every edge; and by abuse of notation we refer to a {\em 2-colored graph}
as a graph with such a choice of 2-coloring. 

Any planar graph with even valence vertices admits a 2-coloring. Since $\Gamma$ is
a Jordan curve it (trivially) admits a 2-coloring; choose one. This 2-coloring pulls
back to a 2-coloring on each of the graphs $\Gamma_n$. There are $d$ complementary white and $d$
complementary black disks to $\Gamma_1$, so if we resolve the 4-valent
vertices in such a way as to create $d$ new white isthmuses and $d$ new
black isthmuses, the result is a Jordan curve $\Gamma_1'$ bounding a single
black and a single white disk.

Now we discuss how to resolve $\Gamma_2$. At a vertex $p$ of $\Gamma_2$ which is
the preimage of a vertex $q$ of $\Gamma_1$, we pull back the resolution of $\Gamma_1$
at $q$ to a resolution of $\Gamma_2$ at $p$. This resolves all vertices of $\Gamma_2$
except for the $2d-2$ vertices at $C$.
Since the germ of $\Gamma$ through $P$ is invariant,
the germ of $\Gamma_1$ and $\Gamma_2$ near the critical set $C$ agrees, so it makes
sense to do `the same' resolution of $\Gamma_2$ at $C$ that we did at $\Gamma_1$.
This produces $\Gamma_2'$. The net result is to create as many black isthmuses and 
white isthmuses, so that $\Gamma_2'$ is a Jordan curve bounding a single
black and a single white disk. Iterate the process.

If we think of each resolution $\Gamma_n' \to \Gamma_n$ as a parameterization by a
circle, we get an inverse sequence of degree $d$ maps 
$$\cdots \to \Gamma'_{n+1} \to \Gamma'_n \to \Gamma'_{n-1} \to \cdots$$
parameterizing the inverse sequence 
$$\cdots \to \Gamma_{n+1} \to \Gamma_n \to \Gamma_{n-1} \to \cdots$$
The set of white vertices in $\Gamma_n$ (those that resolve to a white isthmus) pulls back 
to a set of unordered pairs of points $\Lambda^+_n$ in $\Gamma_n$, whereas the black
points pull back to $\Lambda^-_n$, and each of these is a (finite) lamination.
The circular orders on the $\Gamma_n$ let us identify each $\Lambda^\pm_n$ as a 
sublamination of $\Lambda^\pm_{n+1}$
and in the limit we find a pair of laminations $\Lambda^\pm$ of $S^1$ invariant
under a degree $d$ map $h:S^1 \to S^1$ which is conjugate to $z \to z^d$ by 
Lemma~\ref{lemma:LEO} (implicitly this depends on a choice of isotopy rel. $P$ of
each $\Gamma_n$ to $\Gamma_{n-1}$; this is where $\tilde{\Gamma}$ plays a role,
though we do not explain it here and refer the reader to \cite{Bonk_Meyer}).
Each of the laminations $\Lambda^\pm$ arises from
a dendrite Julia set of a postcritically finite polynomial, and by construction
$H:S^2 \to S^2$ is obtained by (topological) mating.

Notice that this construction is {\em not unique}: it depends on a choice of 
resolution of the vertices of $\Gamma_1$. Different resolutions give representations of
$H$ as a topological mating of different pairs of polynomials.

For most expanding Thurston maps, the laminations $\Lambda^\pm$ will have perfect fits.
Some examples where they do not (and therefore give rise to CaTherine wheels) include 
a mating that gives rise to the map $H: z \to 1 - (3z+1)^3/(9z-1)^2$. This 
map is a Latt\'es example, derived from a Euclidean 
crystallographic group of type $(236)$ (see \cite{Bonk_Meyer}, Example~3.25). It arises
as a mating in more than one way; some of these matings have perfect fits and do
{\em not} give rise to CaTherine wheels.

\subsection{Expanding Origamis}\label{subsection:expanding_origamis}

The construction of $\LL^\pm$ and $Z^\pm$ from a CaTherine wheel $f:S^1 \to S^2$ depends
crucially on a fixed choice of orientation on both $S^1$ and $S^2$. Endomorphisms that
do not (locally) preserve this orientation destroy this structure {\em unless} the
endomorphisms $h:S^1 \to S^1$ and $H:S^2 \to S^2$ simultaneously reverse the orientation
on both $S^1$ and $S^2$ on subsets related by $f$. Rather surprisingly, one can
find many such examples which respect the linear resp. conformal structure on $S^1$ resp. $S^2$.
In particular, many interesting examples exist for which the maps $h$ and $H$ are 
{\em expanding origamis}, to be defined in the sequel. 

\subsubsection{Invariant laminations}

An expanding origami $h:S^1 \to S^1$
is a piecwise linear map with slope $\pm \lambda$ for some fixed $\lambda>1$.
We shall always assume that $h$ achieves both positive and negative slopes (or else
$h$ would be a covering map).
LEO covering maps of $S^1$ are classified up to conjugacy by their degree. By contrast,
there are many different (topological) conjugacy classes of expanding origamis
$h:S^1 \to S^1$ of fixed degree and slope.

We say that a lamination $\Lambda$ is invariant for $h$ if
\begin{enumerate}
\item{for every leaf $\lambda$ either the points of
$\lambda$ are identified by $h$, or $h\lambda$ is a leaf of $\Lambda$; and}
\item{for every leaf $\lambda$ there is some leaf $\mu$ with
$h\mu = \lambda$.}
\end{enumerate}

For typical $h$ different points in $S^1$ will have different numbers of preimages.
If it happens that every point in $S^1$ has $d$ preimages then we say $h$ has 
{\em geometric degree} $d$.

Suppose $h$ has geometric degree $d$ and $\lbrace x,y\rbrace$ is a leaf of $\Lambda$. 
There are $d$ preimages $x_0,\cdots,x_{d-1}$ and $y_0,\cdots,y_{d-1}$ of each
of $x$ and $y$. We say that $\Lambda$ is {\em degree $d$ invariant}
if for every leaf $\lbrace x,y\rbrace$ the $2d$ preimages of $x$ and $y$ are 
partitioned into $d$ pairs which are all leaves of $\Lambda$.

By abuse of notation we refer to these as {\em preimage leaves}. A preimage
leaf of the form $\lbrace x_i,y_j\rbrace$ is {\em ordinary}; otherwise it is
{\em folded}. Note that folded leaves, strictly speaking, have empty forward
image under $h$. We must therefore distinguish between those leaves that have
no forward image and are preimages, and those that have no forward image and are
not preimages. In place of the term `major' we introduce a new term:

\begin{definition}[Seed]\label{definition:seed}
Let $h$ have geometric degree $d$, and let $\Lambda$ be degree $d$ invariant.
A finite subset $C \subset \Lambda$ is a {\em seed} if every leaf of $\Lambda$ is either an
iterated preimage leaf of $C$, or a limit of such leaves.
\end{definition}

\begin{figure}[htpb]
\centering
\includegraphics[scale=0.5]{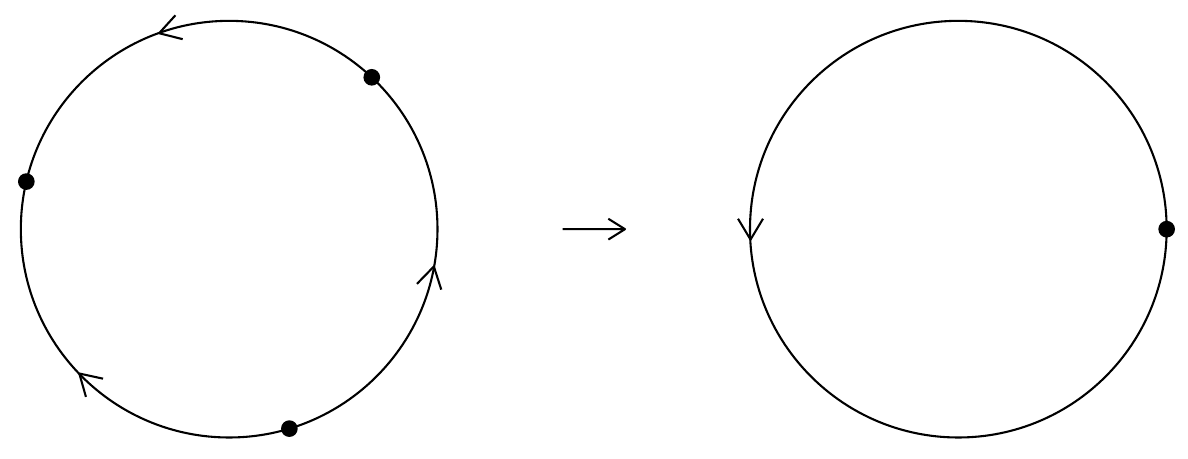}
\caption{The map $h_\theta$ for $\theta = 0.13$}
\label{deg_3_ordinary_map}
\end{figure}

Let us describe one family for which all preimages are ordinary.
\begin{example}[Ordinary family]\label{example:ordinary_family}
For each $\theta \in [0,1/3)$ define $\xi:S^1 \to S^1$ of geometric degree $3$ by
$$\xi(x) = \begin{cases}
3x \text{ if } x \in [0,1/3] \text{ or } x \in [2/3,1] \\
-3x \text{ else }
\end{cases}$$
and then define $h_\theta:S^1 \to S^1$ by $h_\theta(x) = \xi(x-\theta)$; see
Figure~\ref{deg_3_ordinary_map}.

Define $C_\theta$ to be the seed consisting of a single leaf 
$$\lbrace \theta,\theta + 1/3,\theta + 2/3\rbrace$$
Then there is a unique degree 3 invariant lamination $\Lambda_\theta$ for $h_\theta$ with
seed $C_\theta$ and for which every preimage leaf is ordinary.
See Figure~\ref{deg_3_ordinary_origami_laminations}.

\begin{figure}[htpb]
\centering
\includegraphics[scale=0.3]{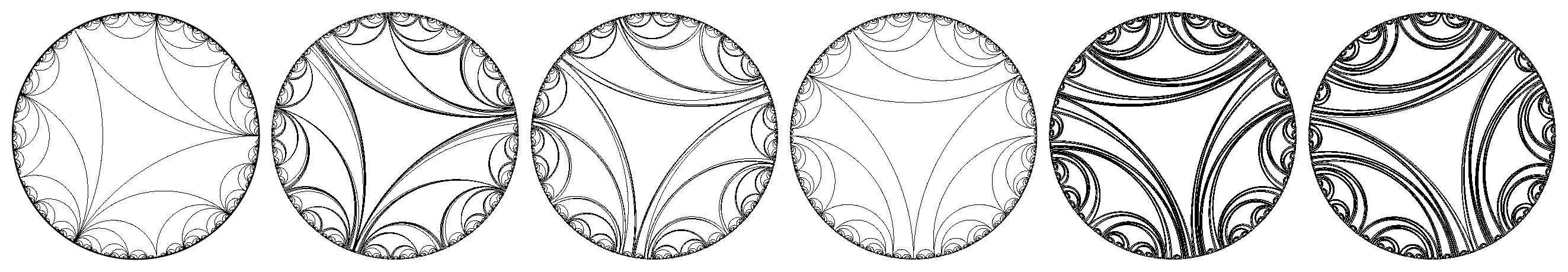}
\caption{Invariant degree 3 origami laminations $\Lambda_\theta$ for 
$\theta = i/36$ where $i = 0,\cdots,5$}
\label{deg_3_ordinary_origami_laminations}
\end{figure}\end{example}

Ordinary families are all alike; every folded family is folded in its own way.
In an ordinary family the map admits a unique (maximal) seed that generates a
degree $d$ invariant lamination.
In order to construct CaTherine wheels we need a pair of seeds generating laminations
without perfect fits; in particular the seeds themselves must be different! Hence
we must necessarily consider invariant laminations with folded preimages.
\begin{example}[Folded family]\label{example:folded_family}
For each $\theta \in [0,1/3)$ define $\xi:S^1 \to S^1$ of geometric degree 3 by
$$h_\theta(x) = \begin{cases}
3x \text{ if } x \in [-1/6,1/6] \text{ or } x \in [1/3,2/3] \\
- 3x \text{ else }
\end{cases}$$
and then define $h_\theta:S^1 \to S^1$ by $h_\theta(x) = \xi(x-\theta)$; see
Figure~\ref{deg_3_folded_map}.
\begin{figure}[htpb]
\centering
\includegraphics[scale=0.5]{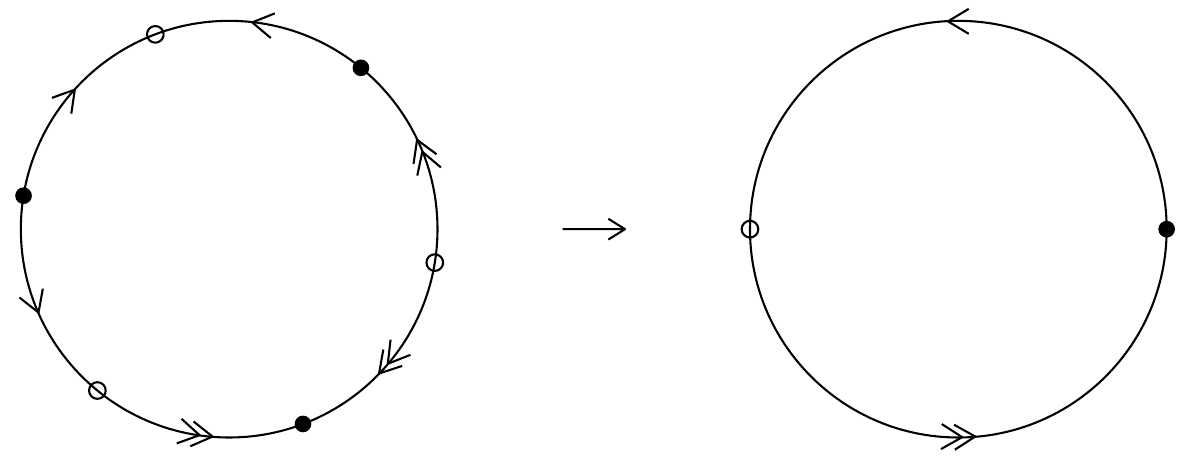}
\caption{The map $h_\theta$ for $\theta = 0.141$}
\label{deg_3_folded_map}
\end{figure}

Define $C_\theta^\pm$ to be the seeds
$$C_\theta^+ = \lbrace \theta, \theta + 1/3, \theta + 2/3\rbrace \quad 
C_\theta^- = \lbrace \theta + 1/6,\theta + 1/2, \theta + 5/6\rbrace$$
Each seed decomposes $S^1$ into three intervals. In either subdivision, one of the
three intervals is {\em ordinary}: it
maps with slope $3$ to the entire interval $[0,1]$. Again in either subdivision, 
the other two intervals are {\em folded}:
one maps with slopes $3$ and $-3$ to $[0,1/2]$ and the other maps with slopes $3$ and
$-3$ to $[1/2,1]$.

We define $\Lambda^\pm$ with seed $C_\theta^\pm$ 
recursively as follows. In order to define the preimage of
a leaf $\lbrace x,y\rbrace$ we first ask whether $x$ and $y$ are both in
$[0,1/2]$ or both in $[1/2,1]$; if so, they have three ordinary preimages. Otherwise they
have one ordinary preimage and two folded preimages. 
See Figure~\ref{deg_3_folded_origami_laminations}.

\begin{figure}[htpb]
\centering
\includegraphics[scale=0.5]{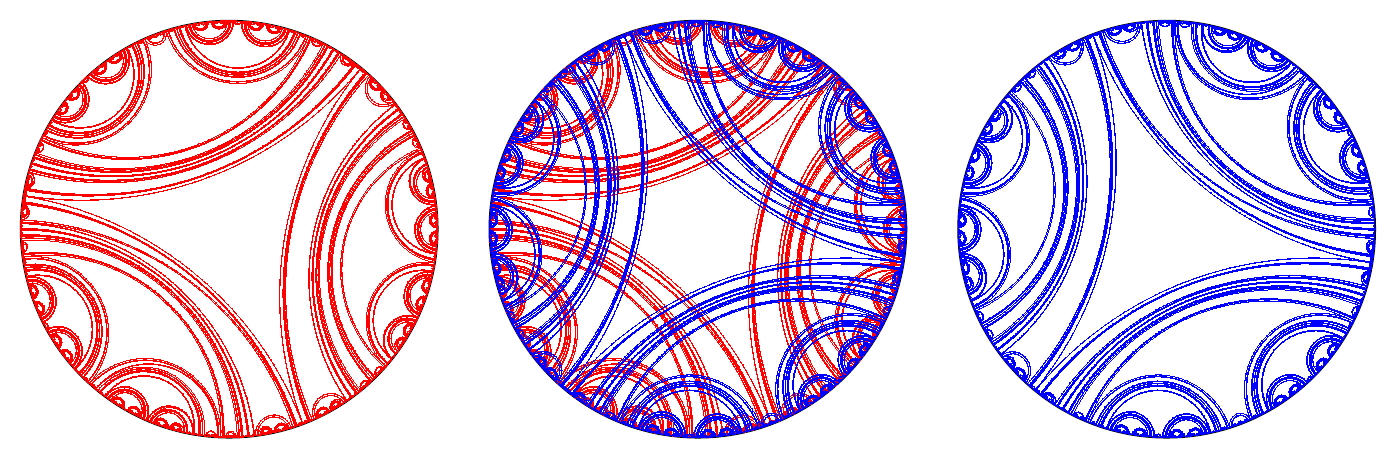}
\caption{Invariant degree 3 origami laminations $\Lambda^\pm_\theta$ for 
$\theta = 0.141$ }
\label{deg_3_folded_origami_laminations}
\end{figure}
\end{example}
One may produce many interesting endomorphisms of CaTherine wheels from folded
families.

\subsubsection{Expanding origamis}

Expanding origamis $H:S^2 \to S^2$ are defined as follows:
\begin{definition}[Expanding origami]\label{definition:expanding_origami}
A map $H: \CP^1 \to \CP^1$ is an {\em expanding origami} if there is a 2-colored
graph $\Gamma$ in $\CP^1$ so that
\begin{enumerate}
\item{the sequence of graphs on $\CP^1$ obtained as iterated preimages $H^{-n}(\Gamma)$ 
has mesh size going to 0; and}
\item{the restriction of $H$ to each white tile is holomorphic, and the
restriction of $H$ to each black tile is antiholomorphic; furthermore, any branch
points of these restrictions to each tile are either in the interior, or at a vertex.}
\end{enumerate}
An expanding origami is {\em special} if there is a subdivision $\Gamma'$ of $\Gamma$ with respect
to which $H$ is cellular.
\end{definition}

Under many conditions, expanding origamis give rise to CaTherine wheels with nontrivial
endomorphisms.

\begin{example}[Latt\`es origami]\label{example:lattes_origami}
As in the holomorphic case, the simplest examples are derived from the Latt\`es construction.
Fix the square lattice $\Z[i] \subset \C$ and consider the map
$\tilde{H}:\C \to \C$ defined by $\tilde{H}(x+iy) = G(x) + iG(y)$ where
$G:\R \to \R$ satisfies $G(1+x) = 1+G(x)$ and $G(1-x) = 1-G(x)$ and on $[0,1/2]$ is defined by
$$G(x) = \begin{cases}
3x \text{ if } x \in [0,1/3] \\
2-3x \text{ if } x \in [1/3,1/2]
\end{cases}$$
The map $\tilde{H}$ commutes with translation by elements of $\Z[i]$, and also with
the involution $z \to -z$ and therefore descends to a map $H:\CP^1 \to \CP^1$ which is
evidently a special expanding origami.

Let $\Delta$ be the group of isometries of $\C$ generated by translations in $\Z[i]$
and the involution $z \to -z$ so that $\C/\Delta = \CP^1$ as Riemann surfaces, and let
$R \subset \C$ be the fundamental domain for $\Delta$
consisting of points $x+iy$ for which $x \in [0,1/2]$ and $y \in [-1/2,1/2]$. The induced
(iterated) action of $\tilde{H}$ on $R$ is indicated in Figure~\ref{folded_map_on_rectangle}:

\begin{figure}[htpb]
\centering
\includegraphics[scale=0.5]{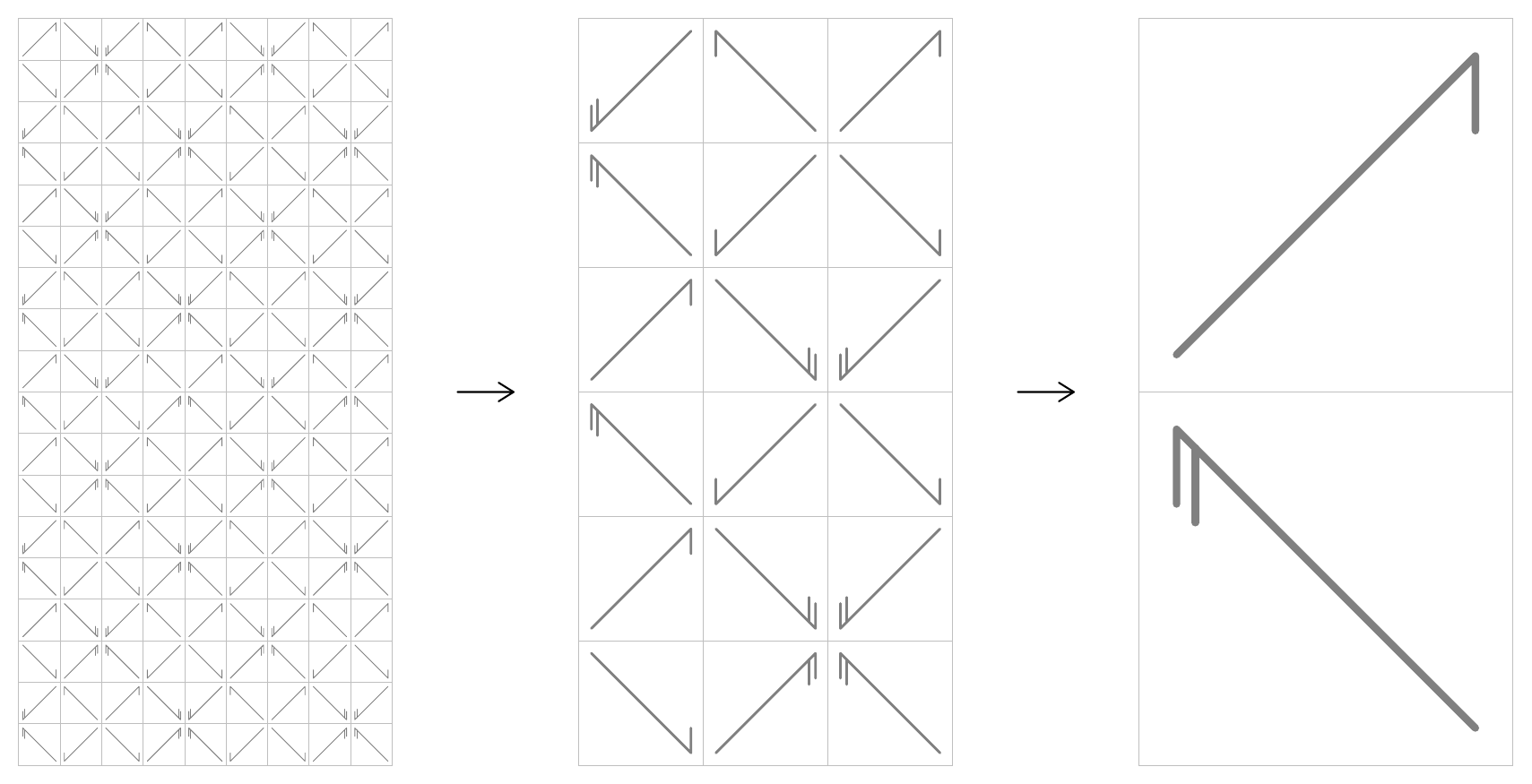}
\caption{The `restriction' of $\tilde{H}$ to $R$}
\label{folded_map_on_rectangle}
\end{figure}

At the same time one can build an origami $h:S^1 \to S^1$ as indicated in 
Figure~\ref{folded_map_on_circle}:

\begin{figure}[htpb]
\centering
\includegraphics[scale=0.5]{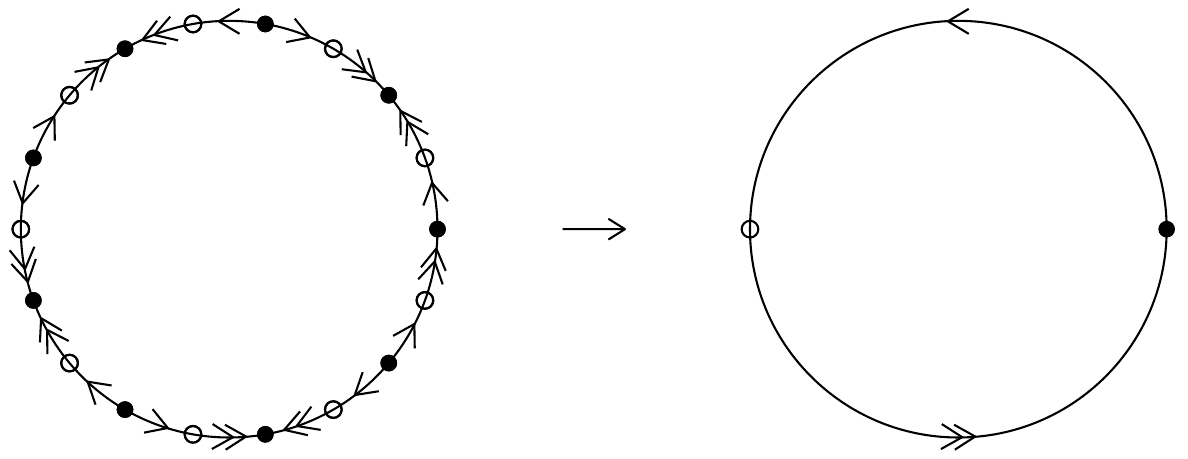}
\caption{The map $h$ on $S^1$}
\label{folded_map_on_circle}
\end{figure}

The two arrows in the rightmost copy of $R$ in Figure~\ref{folded_map_on_rectangle} close
up to form a circle $\gamma_0$ on $\CP^1$; likewise the two arrows in the rightmost
circle in Figure~\ref{folded_map_on_circle} close up to form $S^1$. We may define
$f_0:S^1 \to \CP^1$ to be the simplicial map that takes $S^1$ to $\gamma_0$, single arrow 
to single arrow and double arrow to double arrow. 
We may then inductively construct $f_{n+1}:S^1 \to \CP^1$
satisfying $H f_{n+1} = f_n h$. Actually there are four choices of the construction
at each stage corresponding to two different subdivision rules for the map between 
single and double arrow tiles. If we fix a (pair of) subdivision rules then in
the limit we obtain a CaTherine wheel $f:S^1 \to \CP^1$ satisfying $Hf = fh$. Varying
the choice of subdivision rules in a periodic manner with period $n$ gives $4^n$ CaTherine
wheels satisfying $H^nf = fh^n$, and varying the subdivision rules arbitrarily produces
uncountably many CaTherine wheels in a compact family parameterized by a 4-adic
Cantor set on which the morphism $H,h$ acts as a (one-sided) shift.

Figure~\ref{origami_curve_on_sphere_color} shows an approximation to $f(S^1)$ in $\CP^1$ 
for one invariant choice of $f$.
\begin{figure}[htpb]
\centering
\includegraphics[scale=1]{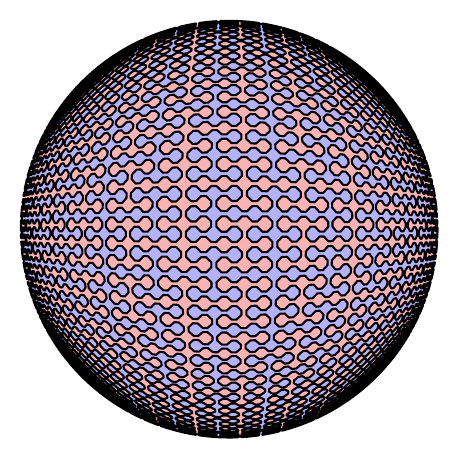}
\caption{An approximation to the origami curve $f(S^1)$ on $\CP^1$.}
\label{origami_curve_on_sphere_color}
\end{figure}
\end{example}

This example is just the tip of a rather large iceberg. The theory of expanding origamis
and their relation to holomorphic dynamics is explored in more detail in
\cite{Calegari_Herr}; also see \cite{Davis_Knuth_1,Davis_Knuth_2}.

\section{Probability}\label{section:SLE}

Conformally natural examples of CaTherine wheels, 
and much of the associated structure (laminar relations, zippers)
are extremely well-studied by probabilists, and arise naturally as
whole plane $\SLE_\kappa$ for $\kappa \ge 8$.
We are indebted to Ewain Gwynne and Greg Lawler for explaining some of this
theory to us, and guiding us to some key references. This is a vast subject, and the
purpose of this section is simply to state some of the relevant facts of the subject
in the language of the rest of this paper, and to refer the reader to the
relevant literature.

Let $\DD \subset \C$ denote the open unit disk.
Suppose $\eta:[0,\infty] \to \overline{\DD}$ is a continuous simple curve in the
unit disk with $\eta(0) \in \partial \overline{\DD}$, $\eta(\infty) = 0$ and
$\eta(0,\infty] \subset \DD$. For each $t \ge 0$ set $K_t: = \eta [0,t]$ and
$U_t:= \DD - K_t$ and
let $g_t:U_t \to \DD$ be the unique conformal homeomorphism with $g_t(0)=0$ and
$g_t'(0)$ a positive real number. After a suitable reparameterization of $\eta$
one may furthermore assume $g_t'(0)=e^t$ (one says $\eta$ is parameterized 
{\em by capacity} from $0$).

For each $t \in [0,\infty)$ the limit
$$W(t):=\lim_{z \to \eta(t)} g_t(z)$$
exists, and $W:[0,\infty) \to \partial \overline{\DD}$ is continuous.
Furthermore, if we assume $\eta$ is paramterized by capacity, $g_t$ satisfies
Loewner's differential equation \cite{Loewner}
$$\partial_t g_t(z) = -g_t(z) \frac {g_t(z) + W(t)} {g_t(z) - W(t)}$$
In fact, given $W$ one may use this differential equation to recover $K_t$
(which now might not be simple, or even a curve). Schramm \cite{Schramm} defined
radial $\SLE_\kappa$ to be $K_t$ when $W(t) = e^{iB(\kappa t)}$ where
$B:[0,\infty) \to \R$ is Brownian motion.

Whole-plane $\SLE_\kappa$ from $\infty$ to $\infty$ for $\kappa\ge 8$
is a variant of $\SLE_\kappa$ first studied in \cite{Miller_Sheffield_IG4}.
It can be built from chordal $\SLE_\kappa$ and whole plane $\SLE_{16/\kappa}$
\cite{DMS_MOT}, footnote~4. The result is a continuous proper map $\eta:\R \to \C$
that extends to $f:S^1 \to \CP^1$. 

The following theorem concatenates results from \cite{Miller_Sheffield_IG4} and 
\cite{Miller_Wu} Remark~5.3:

\begin{theorem}[SLE; \cite{Miller_Sheffield_IG4,Miller_Wu}]\label{theorem:SLE}
Whole-plane $\SLE_\kappa$ from $\infty$ to $\infty$ for $\kappa\ge 8$ gives
rise almost surely to a CaTherine wheel $f:S^1 \to \CP^1$. Furthermore, for such an $f$ 
almost surely every point has at most $3$ preimages; in detail
\begin{enumerate}
\item{a full measure set of points has only 1 preimage;}
\item{a measure 0 uncountable set has exactly 2 preimages; and}
\item{a countable set has exactly 3 preimages.}
\end{enumerate}
\end{theorem}

The laminar relations $\LL^\pm$ associated to $f$ for whole-plane $\SLE_\kappa$
are described in terms of a pair of Brownian motions on $\R$ in 
\cite{DMS_MOT} Theorem~1.9. These Brownian motions are independent for $\kappa=8$
and negatively correlated for $\kappa > 8$. This result has many applications
to SLE and Liouville quantum gravity (hereafter LQG); see e.g.\/ \cite{Gwynne_survey}.

\begin{example}[Uniform spanning tree]\label{example:UST}
Let $D$ be a Jordan domain. Intersect $D$ with the square grid with some small
mesh size $\epsilon$, and let $X_\epsilon$ be a maximal connected subgraph of the
square grid contained in $D$. Let $T_\epsilon$ be a spanning tree of $X_\epsilon$
(chosen randomly with respect to the uniform measure on the finite set of such trees).
Pick a vertex $v$ on $X_\epsilon$ in the closure of the unbounded complementary
component; then the 
boundary of the $\epsilon/4$ neighborhood of $T_\epsilon$ is a Jordan arc with endpoints
both within distance $\epsilon/4$ from $v$ and from $\partial D$.
If we parameterize this Jordan arc proportional to arclength we get a nearly filling 
embedding $f_\epsilon:I \to D$, and associated to this near filling there is a
Kerbs invariant (defined in the obvious way), 
which is a pair of laminations of the interval; see 
Figure~\ref{australia_almost_filled}.

\begin{figure}[htpb]
\centering
\includegraphics[scale=0.75]{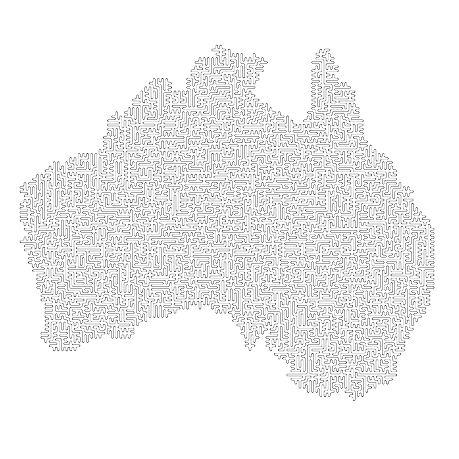} \quad
\includegraphics[scale=0.75]{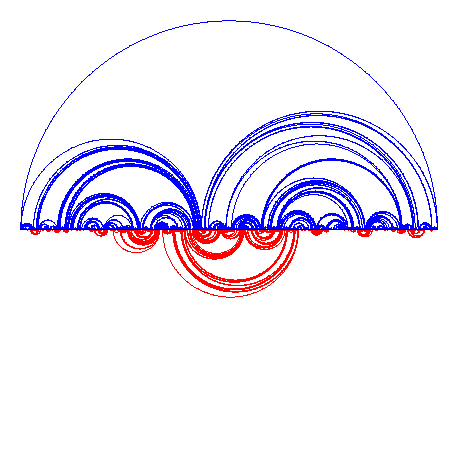}

\caption{A nearly filling map $f_\epsilon:I \to D$ and the pinching laminations}
\label{australia_almost_filled}
\end{figure}

Taking the mesh size to zero, a subsequence converges in probability to a measure
supported on filling maps $f:I \to D$. Two such maps can be welded along their
boundaries to produce a CaTherine wheel. This limiting map is a well-known variant
of $\SLE_8$; see e.g.\/ \cite{Lawler_Schramm_Werner}. 
\end{example}

\begin{example}[Half-zippers from Liouville quantum gravity]\label{example:LQG}
This example is taken from Calegari--Gwynne \cite{Calegari_Gwynne}, which was written
simultaneously with this paper.

For a parameter $\gamma \in (0,2)$, {\em Liouville quantum gravity} (hereafter
$\LQG$) is a random metric $D$ on $\CP^1$ heuristically associated to the
random `metric tensor' $g: = e^{\gamma \Phi}(dx^2 + dy^2)$, where $\Phi$ is the
Gaussian free field (GFF). The GFF is not exactly a random function, rather it
is a random distribution, and is characterized by the property that
for every smooth, compactly supported (real-valued) function $f$ on $\C$, 
the $L^2$ inner product $\langle \nabla \Phi,\nabla f\rangle$ is Gaussian with
mean 0 and variance $\|\nabla f\|^2$. In any case, although $\Phi$ itself
does not make sense as an honest function, the $\gamma$-LQG metric $D$ on 
$\C$ exists as an honest metric almost surely. 

It turns out in this metric that geodesics exist, and for almost every $p\in \C$
there is a $D$-geodesic from $p$ to $\infty$. The {\em $D$-geodesic tree rooted
at $\infty$}, denoted $Z_\infty$ is defined to be the union of all $D$-geodesics
to $\infty$, minus their starting points. In other words,
$$Z_\infty: = \lbrace \sigma(t) \text{ such that }\sigma \text{ is a $D$-geodesic from
a point of $\C$ to $\infty$ and } 0 < t < \infty\rbrace$$
See Figure~\ref{lqg_tree}.

\begin{figure}[htpb]
\centering
\includegraphics[scale=1]{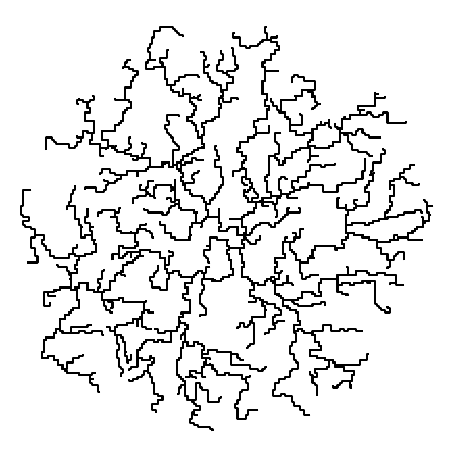} 
\caption{A numerical approximation to part of a LQG metric tree}
\label{lqg_tree}
\end{figure}

\cite{Calegari_Gwynne} Theorem~1.6 says that the set $Z_\infty$ 
is a half-zipper with short hair in the sense of
Definition~\ref{definition:half_zipper}, and \cite{Calegari_Gwynne} Theorem~1.3 
says that every half-zipper with short hair 
is $Z^+$ for some zipper $Z^\pm$ associated to a unique CaTherine wheel 
$f:S^1 \to S^2$; in particular, one obtains a `complementary' half-zipper $Z^-$ whose
meaning in the context of the $\gamma$-LQG metric for general $\gamma$ remains to be seen.
The method of proof is similar to, and depends on, the arguments in
\S~\ref{section:zippers}. 

It should be remarked that the existence of a complementary half-zipper $Z^-$ and
a CaTherine wheel were already known in the special case that $\gamma = \sqrt{8/3}$,
in which case $\gamma$ LQG is equivalent to the Brownian map; see
\cite{Miller_Sheffield_LQG,Miller_Sheffield_axiom}.
\end{example}

\vfill
\pagebreak

\part*{Back Matter}

\section*{Acknowledgements}

We would like to thank Ian Agol, Chris Bishop, Jeremy Brazas, Martin Bridgeman, 
Arnaud Ch\'eritat, Alex Eskin, Sergio Fenley, Steven Frankel, Ewain Gwynne, Yan Mary He, Fran Herr, 
Lucas Kerbs, Sarah Koch, Michael Landry, Greg Lawler, Chris Leininger, Kathryn Mann, Daniel Meyer, 
Curt McMullen, Mahan Mj, Rafael Potrie, Saul Schleimer, Rich Schwartz, Amie Wilkinson and 
Jonathan Zung for valuable comments, assistance and encouragement.

\begin{figure}[htpb]
\centering
\includegraphics[scale=0.1]{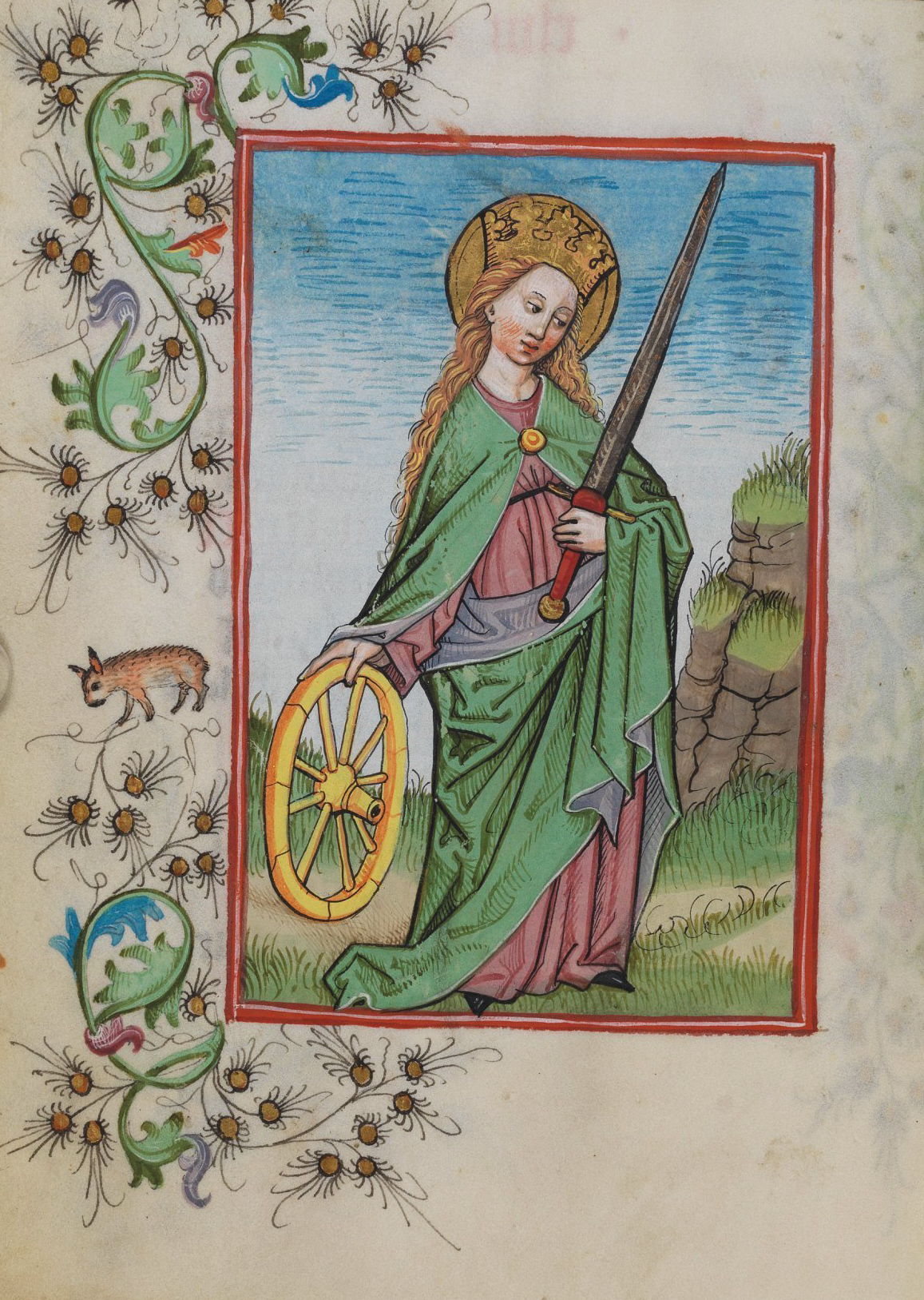}
\caption{Saint Catherine with her wheel \cite{Wikicommons}.}
\label{Saint_Catherine}
\end{figure}

\end{document}